\documentclass[a4paper,reqno]{amsart}

\usepackage{amsmath,amssymb,amsthm,amsfonts}
\usepackage[T1]{fontenc}
\usepackage{tikz}
\usepackage[ruled,vlined,commentsnumbered]{algorithm2e}
\usepackage{xspace}
\usepackage{url}
\usepackage{enumerate}
\usepackage{hyperref}
\usepackage{fullpage}
\usepackage{subcaption}
\usepackage{graphicx}

\theoremstyle{plain}
\newtheorem{theorem}{Theorem}
\newtheorem{lemma}[theorem]{Lemma}

\newtheorem{proposition}[theorem]{Proposition}
\newtheorem{corollary}[theorem]{Corollary}
\newtheorem{claim}[theorem]{Claim}

\newtheorem{case}[theorem]{Case}
\newtheorem{subcase}[theorem]{Subcase}
\theoremstyle{remark}
\newtheorem{remark}[theorem]{Remark}
\newtheorem*{question}{Question}

\theoremstyle{definition}

\newcommand{\pimi}[1][]{\ifthenelse{\equal{#1}{}}{\pi^{\circ}}{\pi_{#1}^{\circ}}}
\newcommand{\sigm}[1][]{\ifthenelse{\equal{#1}{}}{\sigma^{\circ}}{\sigma_{#1}^{\circ}}}
\newcommand{\rhom}[1][]{\ifthenelse{\equal{#1}{}}{\rho^{\circ}}{\rho_{#1}^{\circ}}}

\DeclareMathOperator{\WC}{WC}
\DeclareMathOperator{\WP}{WP}
\DeclareMathOperator{\Loz}{Loz}
\DeclareMathOperator{\supp}{supp}
\DeclareMathOperator{\PG}{PG}

\DeclareMathOperator{\epi}{epi}
\DeclareMathOperator*{\argmin}{arg\,min}

\SetEndCharOfAlgoLine{}
\SetArgSty{textnormal}

\DeclareMathOperator*{\Med}{Med}
\DeclareMathOperator{\lMed}{\Med_{\mathrm{loc}}}
\DeclareMathOperator{\lMedG}{\Med^{\mathrm{loc}}_G}
\DeclareMathOperator{\lpMedG}{\Med_{\mathrm{loc}}^{p}}

\begin{document}

\centerline{\Large\bf Graphs with $G^p$-connected medians}

\vspace{10mm}
\centerline{Laurine B\'en\'eteau, J\' er\'emie Chalopin, Victor Chepoi, Yann Vax\`es}
\medskip
\begin{small}
\medskip
\centerline{Laboratoire d'Informatique et Syst\`emes, Aix-Marseille Universit\'e and CNRS,}
\centerline{Facult\'e des Sciences de Luminy, F-13288 Marseille Cedex 9, France}
\centerline{\texttt{\{laurine.beneteau, jeremie.chalopin, victor.chepoi, yann.vaxes\}@lis-lab.fr}}
\end{small}

\bigskip\bigskip\noindent
{\footnotesize {\bf Abstract.} 	 The median of a graph $G$ with weighted vertices is the
set of all vertices $x$  minimizing the sum of weighted distances from $x$ to the vertices of $G$.
For any integer $p\ge 2$, we characterize the graphs in which, with respect to any non-negative weights,
median sets always induce connected subgraphs in the $p$th power $G^p$ of $G$. This extends some
characterizations of graphs with connected medians (case $p=1$) provided by Bandelt and Chepoi (2002).
The characteristic conditions can be tested in polynomial time for any $p$. We also show that several
important classes of graphs in metric graph theory, including bridged graphs (and thus chordal graphs), graphs with convex balls,
bucolic graphs, and bipartite absolute retracts, have $G^2$-connected medians. Extending the result of Bandelt and Chepoi that
basis graphs of matroids are graphs with connected medians, we characterize the
isometric subgraphs of Johnson graphs and of halved-cubes with connected medians.
}

\section{Introduction} The median problem (also called the Fermat-Torricelli problem or the
Weber problem) is one of the oldest optimization problems in Euclidean
geometry~\cite{LoMoWe}.  The \emph{median problem} can be defined for
any metric space $(X,d)$.  A {\it weight function}
is any mapping $\pi: X\rightarrow {\mathbb{R}}^+$ with finite support. The \emph{median function}
on a point $x$ of $X$  is given by $F_{\pi}(x)=\sum_{u: \pi(u)>0} \pi(u)d(u,x).$
A point $x$ minimizing $F_{\pi}(x)$ is called a {\it median}  of $X$
with respect to $\pi$, and the set of all medians is the {\it median set}
$\Med(\pi).$ By a {\it local median} one means a point $x$ such that
$F_{\pi}(x)$ does not exceed $F_{\pi} (y)$ for any $y$ in a neighborhood of $x$.
Denote by $\lMed(\pi)$ the set of all local medians.

Any connected unweighted graph $G=(V,E)$ endowed with its standard
graph-distance can be viewed as  a (discrete) metric space.
The median problem in graphs arises with one of the basic models
in network location theory~\cite{Hakimi,TaFrLa}, with majority consensus
in classification and data analysis \cite{BaBa,BaJa,BaLeMo,BaMo, McMoPo},
and with closeness centrality in network analysis~\cite{Bav,Beau,Sabi}.
The median problem also occurs in social group choice as a judgement aggregation rule. A \emph{judgment
aggregation rule}  takes the individual opinions of a collection of voters over a set of interconnected issues
and yields a logically consistent collective decision. The \emph{median rule} is a judgement aggregation rule
that selects the decision which minimizes the average distance to the individual opinions of the voters (where the distance between
two opinions is the number of issues on which they disagree). In the special case of
preference aggregation, the median rule is called the \emph{Kemeny rule} or
the \emph{Kemeny median}~\cite{Ke,KeSn}.
By the classical Arrow's impossibility
theorem, there is no consensus function satisfying natural
``fairness'' axioms. It is also well-known that the majority rule
leads to Condorcet's paradox, i.e., to the existence of cycles in the
majority relation. In this respect, the Kemeny median~\cite{Ke,KeSn}
is an important consensus function and satisfies several desirable fairness properties. It
corresponds to the median
problem in the graphs of permutahedra. Recently, Nehring and Pivato \cite{NePi} proved, under appropriate
regularity conditions, that the general median rule is the unique judgement aggregation rule which
satisfies three axioms: \emph{Ensemble Supermajority Efficiency, Reinforcement, and Continuity}.

Instead of considering rankings as individual opinions, Balinski and Laraki \cite{BaLa1,BaLa2} described
and successfully used in practice a
multi-criterion majority rule based on grades, which they  called \emph{majority judgement}. In this judgement aggregation rule,
voters grade the candidates with regard to their suitability; multiple candidates may be given the same grade by a voter.
Then the candidates are ranked according to their median grades and the candidate with the highest median rating is elected.
Nehring and Pivato \cite{NePi} showed that the majority judgement is equivalent to the weighted median rule in the
uniform decision model, which they describe.

The median problem in Euclidean spaces can be solved numerically by Weiszfeld's
algorithm~\cite{Weis} and its convergent modifications (see e.g.\
\cite{Os}). The convergence of such algorithms is based on the convexity
of the median function $F_{\pi}(x)$, yielding that local medians are medians ($\Med(\pi)=\lMed(\pi)$) and that median sets
$\Med(\pi)$ are convex. In fact, $F_{\pi}$ is convex in all metric spaces in which
the distance function $d$ is convex. This is the case of CAT(0) geodesic metric
spaces \cite{BrHa} and of all normed spaces. Analogously to the Euclidean and hyperbolic spaces, in a
CAT(0) space any pair of points $x,y$ can be connected by a unique geodesic (shortest path).
In a normed space however a pair of points $x,y$ may be connected by several geodesics  and the linear
segment $[x,y]$ is one of them. The distance and the median functions are convex on $[x,y]$ but
are not necessarily convex on all $(x,y)$-geodesics. In view of this, a function
$f: X\rightarrow {\mathbb R}$ on a metric space $(X,d)$ can be called \emph{weakly convex} if $f$
is convex on \emph{some} $(x,y)$-geodesic $g(x,y)$ for any pair of points $x,y$ of $X$ and \emph{convex}
if $f$ is convex on \emph{any} $(x,y)$-geodesic $g(x,y)$. The convexity of $f$ implies that
its set of minima is a convex set and the weak convexity of $f$ implies that its set of minima is an isometric subset of $X$. Consequently,  the median functions and the median sets in  CAT(0) spaces are convex while the median functions in normed
spaces are weakly convex and the median sets are isometric. Peakless functions, introduced by Busemann \cite{Busemann} and investigated in \cite{BuPh}
in the setting of geodesic metric spaces, represent a generalization of convex functions: for $z=\lambda\cdot x+(1-\lambda)\cdot y$,
the quantitative inequality $f(z)\le \lambda\cdot f(x)+(1-\lambda)\cdot f(y)$ is replaced by the qualitative inequality
$f(z)\le \max \{ f(x), f(y)\}$, where equality holds only if $f(x)=f(z)=f(y)$. Weakly peakless functions can be defined similarly
to weakly convex functions and they are also unimodal.

Returning to the discrete setting of graphs, first note that the distance, the eccentricity, and  the median functions in trees are convex~\cite{TaFrLa}.
For general graphs, the median function is no longer convex or weakly convex and the median sets may have an arbitrary structure.  Slater \cite{Sl} proved
that every (possibly disconnected) graph can be realized as a subgraph induced by the median set. Dankelman and Sabidussi \cite{DaSa} proved that any connected graph can be
realized as a median set, inducing an isometric subgraph of the host graph. On the other hand, it was shown by Bandelt and Chepoi \cite{BaCh_median} that the
graphs in which all median sets are connected are the same as the graphs in
which all median sets are isometric. More importantly, it was proved in \cite{BaCh_median} that
this class of graphs is the same as the class of graphs in which  $\Med(\pi
)=\lMed(\pi)$  holds for all weight functions $\pi$
(i.e., for which all median functions are unimodal)  and is the same as the class of graphs in which all median functions are weakly convex/peakless.
It is shown in \cite{BaCh_median} that graphs with connected medians are meshed.
All  these conditions can be tested in polynomial time \cite{BaCh_median}. Finally, the theorem of \cite{BaCh_median} provides a
local-to-global characterization of graphs with connected medians, allowing to show that median and weakly median graphs, Helly graphs, and basis graphs of matroids
and of even $\Delta$-matroids are graphs with connected medians.
The fact that the median function of median graphs is weakly convex was established in \cite{SoCh_Weber} (it was previously shown in \cite{BaBa} that the median
function on median graphs is unimodal). We used these results to compute medians in median graphs in linear time \cite{BeChChVa}. Using the unimodality of
the median function on Helly graphs \cite{BaCh_median}, Ducoffe \cite{Du_Helly}  presented an algorithm with complexity $\tilde{O}(m\sqrt{n})$ for computing medians of
Helly graphs with $n$ vertices and $m$ edges. Chepoi et al. \cite{ChFaVa} used unimodality of the median fuction to design a linear time algorithm for computing medians
in trigraphs (planar bridged triangulations). Finally, Bandelt and Chepoi \cite{BaCh_median} used the unimodality of the median function in basis graphs of matroids
to reduce the computation of medians to a problem solvable by the greedy algorithm. The unimodality was also investigated for eccentricity functions.
Dragan \cite{Dragan,Dr_helly} proved that the graphs in which the eccentricity functions are unimodal are exactly the Helly graphs.
More recently, the  unimodality of generalized eccentricity functions on trees have been investigated in  \cite{GuHiTo}.

\subsection*{Our results}
In this paper, we extend a part of the results of \cite{BaCh_median}, in which we replace weakly peakless/convex functions and connectivity/isometricity of the median sets by step-$p$ (or $p$-)weakly peakless/convex functions and step-$p$ connectivity/isometricity. Namely, for any integer $p\ge 2$, we characterize the graphs $G$ in which all median sets induce connected subgraphs in the $p$th power $G^p$ of $G$.
We show that those are exactly the graphs in which all median functions are $p$-weakly peakless/convex (such functions on $G$ are unimodal on $G^p$).
To obtain these results, we prove a local-to-global characterization of general $p$-weakly peakless/convex functions, which may be of independent interest. A consequence of this equivalence is that
testing if a graph $G$ has $G^p$-connected medians can de done in polynomial time via linear programming. In fact, the algorithmic problem reduces to one of checking the unsolvability of a system of linear inequalities.
Our results also allow to considerably extend the list of classes of graphs in which the median function is unimodal (in $G$ or $G^2$). Completing  the list of classes of  graphs with connected medians established in \cite{BaCh_median}, we show that several other important classes of graphs in metric graph theory have $G^2$-connected medians: these are  bridged and weakly bridged graphs (and thus chordal graphs), graphs with convex balls, bucolic graphs, bipartite absolute retracts, and benzenoids. We also characterize the isometric subgraphs of Johnson graphs and halved-cubes with connected medians. In case of Johnson graphs, we show that those are exactly the meshed isometric subgraphs of Johnson graphs. We hope that similarly to median, Helly, and basis graphs of matroids, our unimodality results can be used to design fast algorithms for computing medians in some of these classes of graphs.

\section{Definitions}

\subsection{Graphs}
A \emph{graph} $G=(V,E)$ consists of a set of vertices $V:=V(G)$ and a set of edges $E:=E(G)\subseteq V\times V$. All graphs considered in this paper are
undirected, connected, and contain no multiple edges, neither loops. All of our non-algorithmic  results hold for infinite graphs. 
We write $u \sim v$ if $u, v \in V$ are adjacent and $u\nsim v$ if $u,v$ are not adjacent. For a subset
$A\subseteq V,$ the subgraph of $G=(V,E)$  {\it induced by} $A$
is the graph $G(A)=(A,E')$ such that $uv\in E'$ if and only if $uv\in E$. The \emph{length}
of a $(u,v)$-path $(u=u_0,u_1,\ldots,u_{k-1},u_k=v)$  is the number $k$ of
edges in this path. A \emph{shortest $(u,v)$-path} (or a
$(u,v)$-\emph{geodesic})
is a $(u,v)$-path with a minimum number of edges.  The \emph{distance} $d_G(u,v)$ between two vertices $u$ and $v$ of a graph $G$ is the length of a $(u,v)$-geodesic.
If there is no ambiguity, we will denote $d(u,v)=d_G(u,v)$.
The \emph{interval} $I(u,v)$ between two vertices
$u$ and $v$ is the set of all vertices on $(u,v)$-geodesics, i.e.  $I(u,v)=\{w\in V: d(u,w)+d(w,v)=d(u,v)\}$. We will denote by
$I^{\circ}(u,v)=I(u,v)\setminus \{ u,v\}$ the ``interior'' of the interval $I(u,v)$. If $d(u,v)=k$, then $I(u,v)$ is called a \emph{k-interval}. A graph $G$ is \emph{thick} if each pair of vertices at distance 2 is
included in a cycle of length 4.
For a vertex $v$ of $G$ and an integer $r\ge 1$, we will denote  by $B_r(v)$
the \emph{ball} in $G$
(and the subgraph induced by this ball)  of radius $r$ centered at  $v$, i.e.,
$B_r(v)=\{ x\in V: d(v,x)\le r\}.$ More generally, the $r$--{\it ball  around a set} $A\subseteq V$
is the set (or the subgraph induced by) $B_r(A)=\{ v\in V: d(v,A)\le r\},$ where $d(v,A)=\mbox{min} \{ d(v,x): x\in A\}$.
As usual, $N(v)=B_1(v)\setminus\{ v\}$ denotes the set of neighbors of a vertex $v$ in $G$.
For an integer $p\ge 1$, the \emph{$p$th power} of a graph $G=(V,E)$ is a graph $G^p$ having
the same vertex-set $V$ as $G$ and two vertices $u$ and $v$ are adjacent in $G^p$ if and only if
$d_G(u,v)\le p$.

A graph $G=(V,E)$
is {\it isometrically embeddable} into a graph $G'=(W,F)$
if there exists a mapping $\varphi : V\rightarrow W$ such that $d_{G'}(\varphi (u),\varphi(v))=d_G(u,v)$ for $u,v\in V$.
We call a subgraph $H$  of $G$ an \emph{isometric subgraph}  if $d_H(u,v)=d_G(u,v)$
for each pair of vertices $u,v$ of $H$. The induced subgraph $H$ of $G$
(or the corresponding vertex set of $H$) is called {\it  convex}
if it includes the interval of $G$ between any pair of its
vertices. The smallest convex subgraph containing a given subgraph $S$
is called the {\it convex hull} of $S$ and is denoted by conv$(S)$.
An induced subgraph $H$ (or the corresponding vertex set of $H$)
of a graph $G$
is {\it gated}  if for every vertex $x$ outside $H$ there
exists a vertex $x'$ in $H$ (the {\it  gate} of $x$)
such that  $x'\in I(x,y)$ for any $y$ of $H$. Gated sets are convex and
the intersection of two gated sets is gated. A graph $G$ is a {\it gated amalgam} of two
graphs $G_1$ and $G_2$ if $G_1$ and $G_2$ are (isomorphic to) two intersecting
gated subgraphs of $G$ whose union is all of $G.$ Let $G_{i}$, $i \in \Lambda$ be an arbitrary family of graphs. The
\emph{Cartesian product} $\prod_{i \in \Lambda} G_{i}$ is a graph
whose vertices are all functions $x: i \mapsto x_{i}$, $x_{i} \in
V(G_{i})$.  Two vertices $x,y$ are adjacent if there exists an index
$j \in \Lambda$ such that $x_{j} y_{j} \in E(G_{j})$ and $x_{i} =
y_{i}$ for all $i \neq j$. Two edges $xy$ and $x'y'$ of $G$ are called \emph{parallel} if there exists an index
$j \in \Lambda$ such that $x_{j} y_{j}=x'_jy'_j \in E(G_{j})$.
Note that a Cartesian product of infinitely
many nontrivial graphs is disconnected (where connectivity between two nodes is defined, as usual,
by the existence of a finite path joining them). Therefore, in this case the
connected components of the Cartesian product are called {\it weak
  Cartesian products}.  A \emph{retraction} $f$ from a graph $H$
to a subgraph $G$ is a mapping $f$ of the vertex set $V(H)$ of $H$
onto the vertex set $V(G)$ of $G$ such that for every edge $uv$ of $H$
the image $f(u)f(v)$ is an edge in $G$ and $f(w)=w$ for all vertices
of $G$. A graph is an {\it absolute retract} exactly when it is a retract of any larger
graph into which it embeds isometrically.

As usual, $C_n$ is the  \emph{cycle} on $n$ vertices, $K_n$ is the complete graph on $n$ vertices, and
$K_{n,m}$ is the complete bipartite graph with $n$ and $m$ vertices on each
side.  $C_3=K_3$ is called a \emph{triangle}, $C_4$ is called a \emph{square}, $C_5$ is called a \emph{pentagon},
and $C_6$ is called a \emph{hexagon}.
By $W_n$ we denote
the \emph{wheel} consisting of $C_n$ and a central vertex adjacent to all vertices of
$C_n$. By $W^-_n$ we mean the wheel $W_n$ minus an edge connecting the central vertex to a vertex of $C_n$.
Analogously $K^-_4$ and $K^-_{3,3}$ are the graphs obtained from $K_4$ and $K_{3,3}$ by removing one edge. A {\it propeller} is the graph $K_5-K_3$, i.e.,
the graph consisting of three triangles glued along a common edge.
An $n$--{\it octahedron} $K_{n\times 2}$ (or, a {\it hyperoctahedron}, for short) is the complete graph $K_{2n}$
on $2n$ vertices minus a perfect matching. Any induced subgraph of $K_{n\times 2}$ is called a {\it subhyperoctahedron}.
A {\it hypercube} $H(X)$ is a graph having the finite subsets of $X$ as vertices
and two  such sets $A,B$ are adjacent in $H(X)$
if and only if $|A\triangle B|=1$ (where the {\it symmetric difference} of two
sets $A$ and $B$ is written and defined by $A\triangle B=(A\setminus
B)\cup (B\setminus A)$). A {\it halved-cube} $\frac{1}{2}H(X)$ has the vertices
of a hypercube $H(X)$ corresponding to finite subsets of $X$ of even cardinality
as vertices and two  such vertices are adjacent in $\frac{1}{2}H(X)$  if and only if their  distance
in $H(X)$ is 2 (analogously one can define a half-cube on finite subsets of odd cardinality).  For a positive integer $k$, the {\it Johnson graph}  $J(X,k)$ has the subsets of $X$ of size $k$ as vertices and
two such vertices are adjacent in $J(X,k)$ if and only if their  distance in $H(X)$ is $2$.
All Johnson graphs $J(X,k)$ with even $k$ are isometric subgraphs of the half-cube $\frac{1}{2}H(X)$.  If $X$ is finite and $|X|=n$,
then the hypercube, the half-cube, and the Johnson graphs are usually denoted by $H_n, \frac{1}{2}H_n,$ and $J(n,k),$
respectively. One can easily show that $\frac{1}{2}H_n$ is isomorphic to the square $H^2_{n-1}$ of $H_{n-1}$.

\subsection{Functions} In this paper, we consider real-valued functions of one variable $f: {\mathbb R}\rightarrow {\mathbb R}$ and real-valued functions
$f:V\rightarrow {\mathbb R}$ defined on the vertex-set $V$  of a graph $G=(V,E)$. We refer to the second type of functions as
\emph{functions defined on a graph} $G$. The \emph{epigraph} of a function $f: {\mathbb R}\rightarrow {\mathbb R}$ is the set
$\epi(f)=\{( x,\mu): \mu\ge f(x)\}\subset{\mathbb R}^2$. For $\alpha\in {\mathbb R}$, the \emph{level set} of a function $f: V\rightarrow {\mathbb R}$
is the set $L_{\preceq}(f,\alpha)=\{ v\in V: f(v)\le \alpha\}\subset V$. The \emph{minimum set} (or the \emph{set of global minima})
of $f$ is the set $\argmin(f)$ of all vertices of $G$  minimizing the function $f$ on $G$. By a {\it local minimum} one means a vertex $v$ of $G$
such that $f(v)$ does not exceed $f(w)$ for any $w$ in the neighborhood $B_1(v)$ of $v$. A function $f$ is called {\it unimodal} on $G$ if every
local minimum of $f$ is global, that is, if the inequality $f(v)\le f(y)$ holds for all neighbors $y$ of a vertex $v$ of $G$ then
it holds for all vertices $y$ of $G.$

Recall that a function $f: [a,b]\rightarrow {\mathbb R}$ of one variable is \emph{convex} if for
any $\lambda\in [0,1]$  and $x=\lambda\cdot a+(1-\lambda)\cdot b$, the inequality $f(x)\le \lambda\cdot f(a)+(1-\lambda)\cdot f(b)$ holds.
It is well known that $f$ is convex if and only if its epigraph $\epi(f)$ is a convex set of ${\mathbb R}^2$.
A function $f[a,b]\rightarrow {\mathbb R}$ is \emph{peakless} \cite[p. 109]{Busemann} if  $a',c',b'\in [a,b]$  with $a'<c'<b'$ implies
$f(c')\le \max\{ f(a'), f(b')\}$ and equality holds only if $f(a')=f(c')=f(b')$.

A {\it profile} (or a {\it weight function})  is any mapping $\pi: V\rightarrow
{\mathbb{R}^+\cup \{ 0\}}$ with finite support $\supp(\pi)=\{ v\in V: \pi(v)>0\}$. A profile $\pi$ is an {\it integer profile} if $\pi(v)\in {\mathbb{N}}$ for any $v\in V$.
The
\emph{median function} on
a vertex $v$ of $V$  is given by $F_{\pi}(v)=\sum_{u: \pi(u)>0} \pi(u)d(u,v).$
A vertex $v$ minimizing $F_{\pi}(v)$ is called a {\it median}  of $G$
with respect to the profile $\pi$.  The set of all medians is the {\it median set}
$\Med(\pi)$, i.e., $\Med(\pi)=\argmin(F_{\pi})$. By a {\it local median} we mean any local minimum of  $F_{\pi}$ and
we denote by  $\lMed(\pi)$ the set of all local medians of $G$. The simplest median function on a graph is obtained
when considering the profile $\pi(u)=1$ for a fixed vertex $u\in V$ and $\pi(x)=0$ for $x\ne u$.
This particular median function is called the \emph{radius (\mbox{or} distance) function} $r_u(v)=d(v,u)$ with respect to $u$.

\subsection{Weakly modular and meshed graphs}\label{wm-classes} In this subsection, we define weakly modular and meshed graphs.
We also briefly review median, weakly median, and Helly graphs, three
subclasses of weakly modular graphs  with connected medians.
Other subclasses of weakly modular and meshed graphs, which are relevant to the
results of this paper, are defined in Section~\ref{classes}.
For a more extensive discussion of all those subclasses, see the survey
\cite{BaCh_survey} and the paper \cite{CCHO}.

Vertices $v_1,v_2,v_3$ of a graph $G$ form a {\it metric triangle}
$v_1v_2v_3$ if the intervals $I(v_1,v_2), I(v_2,v_3),$ and
$I(v_3,v_1)$ pairwise intersect only in the common end-vertices, i.e.,
$I(v_i, v_j) \cap I(v_i,v_k) = \{v_i\}$ for any $1 \leq i, j, k \leq 3$.
If $d(v_1,v_2)=d(v_2,v_3)=d(v_3,v_1)=k,$ then this metric triangle is
called {\it equilateral} of {\it size} $k.$ An equilateral metric triangle
$v_1v_2v_3$ of size $k$ is called \emph{strongly equilateral} if $d(v_i,x)=k$ for any vertex $x\in I(v_j,v_k)$, where $\{ i,j,k\}=\{ 1,2,3\}$.
A metric triangle
$v_1v_2v_3$ of $G$ is a {\it quasi-median} of the triplet $x,y,z$
if the following metric equalities are satisfied:
\begin{align*}
d(x,y)&=d(x,v_1)+d(v_1,v_2)+d(v_2,y),\\
d(y,z)&=d(y,v_2)+d(v_2,v_3)+d(v_3,z),\\
d(z,x)&=d(z,v_3)+d(v_3,v_1)+d(v_1,x).
\end{align*}

If $v_1=v_2=v_3=v$ (i.e., $v_1v_2v_3$ is an equilateral metric triangle of size 0), then the vertex $v$ is called a {\em median} of $x,y,z$.
A median may not exist and may not be unique.
On the other hand, a quasi-median of every $x,y,z$ always exists:
first select any vertex $v_1$ from $I(x,y)\cap I(x,z)$ at maximal
distance to $x,$ then select a vertex $v_2$ from $I(y,v_1)\cap
I(y,z)$ at maximal distance to $y,$ and finally select any vertex
$v_3$ from $I(z,v_1)\cap I(z,v_2)$ at maximal distance to $z.$

A graph $G$ is \emph{weakly modular}  \cite{BaCh_helly,Ch_metric} if for any vertex $u$
its distance function $d$ satisfies the following triangle and quadrangle conditions
(see Fig.~\ref{fig-conditions}):
\begin{itemize}
\item
\emph{Triangle condition} TC($u$):  for any two vertices $v,w$ with
$1=d(v,w)<d(u,v)=d(u,w)$ there exists a common neighbor $x$ of $v$
and $w$ such that $d(u,x)=d(u,v)-1.$
\item
\emph{Quadrangle condition} QC($u$): for any three vertices $v,w,z$ with
$d(v,z)=d(w,z)=1$ and  $2=d(v,w)\leq d(u,v)=d(u,w)=d(u,z)-1,$ there
exists a common neighbor $x$ of $v$ and $w$ such that
$d(u,x)=d(u,v)-1.$
\end{itemize}

\begin{figure}[h]
\begin{center}
\includegraphics{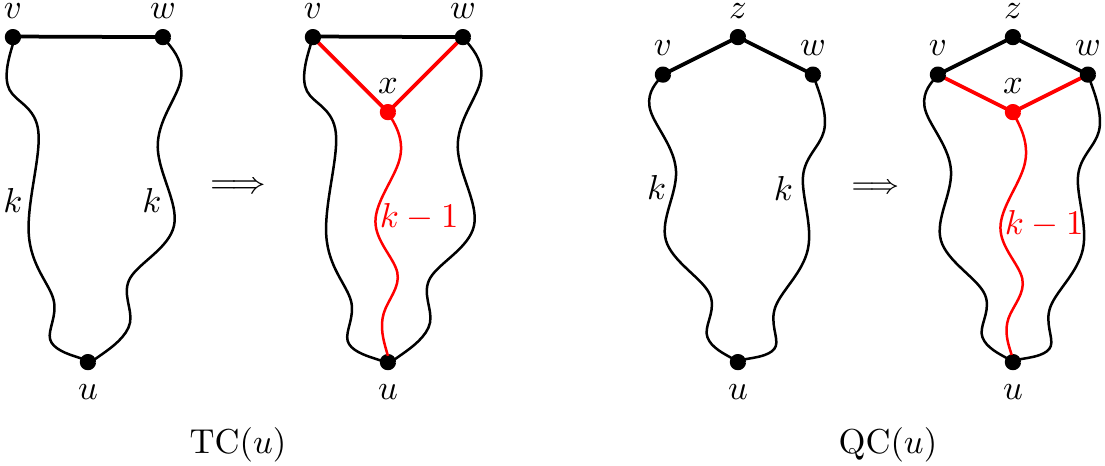}
\end{center}
\caption{Triangle and quadrangle conditions}\label{fig-conditions}
\end{figure}

All metric triangles of weakly modular graphs are equilateral. In fact,  a graph $G$
is weakly modular if and only if  any metric triangle $v_1v_2v_3$ of $G$ is strongly
equilateral \cite{Ch_metric}.

A graph $G=(V,E)$ is called {\it meshed} \cite{BaCh_median} if for any
three vertices $u,v,w$ with $d(v,w)=2,$ there exists a common
neighbor $x$ of $v$ and $w$ such that $2d(u,x)\le d(u,v)+d(u,w).$
Metric triangles of meshed graphs are equilateral \cite{BaCh_median}.
All weakly modular graphs are meshed; for a  proof, see \cite[Lemma 2.22]{CCHO}.
Also all graphs with connected medians are meshed \cite{BaCh_median}.

A graph $G$  is called {\it median} if  $|I(u,v)\cap I(v,w)\cap I(w,u)|=1$ for
every triplet $u,v,w$ of vertices,
i.e., every triplet of vertices has a unique median. Median graphs constitutes the most important  main class of graphs in metric graph theory.
They can be characterized in several different ways and they play an important role in geometric group theory, concurrency, and combinatorics.
Two characterizations of median graphs motivated the characterization of some classes of graphs occurring in this paper. First, median graphs are exactly the
retracts of hypercubes \cite{Ba_retracts} and second, finite median graphs are exactly the graphs obtained from hypercubes (Cartesian products of edges)
by successive gated amalgams. Median graphs are graphs with connected medians \cite{BaBa,SoCh_Weber} and this property together with the majority rule were used in our linear-time
algorithm \cite{BeChChVa} for computing medians in median graphs. For properties and characterizations of median graphs, see the survey \cite{BaCh_survey}.
More generally, a weakly modular graph $G$ is \emph{weakly median} \cite{BaCh_weak} if each triplet of vertices of $G$ has a unique quasi-median.
{\it Quasi-median graphs} \cite{BaMuWi} are the $K^-_4$ and $K_{2,3}$--free weakly modular
graphs. Quasi-median graphs are weakly median.  It was shown in \cite{BaCh_median} that weakly median graphs have connected medians.

A graph $G$ is called {\it modular} if $I(u,v)\cap I(v,w)\cap I(w,u)\ne
\varnothing$ for every triplet $u,v,w$ of vertices, i.e., every triplet of
vertices admits
a (not necessarily unique) median. Clearly  median graphs are modular. Modular graphs are exactly the weakly modular graphs in which all metric triangles of $G$
have size 0.  A graph $G$ is called {\it pseudo-modular} if
any three pairwise intersecting balls of $G$ have a nonempty
intersection \cite{BaMu_pmg}. This condition easily implies both the triangle
and quadrangle conditions, and thus pseudo-modular graphs are weakly modular. In fact,
pseudo-modular graphs are quite specific weakly modular graphs: from the definition
also follows that all metric triangles of pseudo-modular graphs have size 0 or 1, i.e., each metric triangle is either a single vertex or
is a triangle of $G$.

A graph $G$ is a {\it (finitely) Helly graph} if the family
of balls of $G$ has the Helly property, that is, every
finite collection of pairwise intersecting balls of $G$ has a
nonempty intersection.  From the definition it immediately
follows that Helly graphs are pseudo-modular.   Helly graphs are the
discrete analogues of
hyperconvex spaces: namely, the requirement that radii of balls are
from the nonnegative reals is modified by replacing the reals by the
integers.  In perfect analogy with hyperconvexity, there is a close
relationship between Helly graphs and absolute retracts: absolute retracts and
Helly graphs are the same \cite{HeRi}.  In particular, for any graph
$G$ there exists a smallest Helly graph comprising $G$ as an isometric
subgraph. As shown in \cite{BaCh_median}, Helly graphs are graphs with
connected medians.

\section{$p$-Weakly peakless and $p$-weakly convex functions on graphs} In this section, first we introduce geodesic strings in graphs and
convex/peakless functions on geodesic strings.
This allows us to define in a unified way the weakly peakless/convex functions on graphs \cite{BaCh_median} and their generalization, the so-called (step-$p$)
$p$-weakly peakless/convex functions. The $p$-weakly peakless/convex functions are functions for which the peaklessness/convexity conditions are satisfied
only for pairs of vertices at distance strictly larger than $p$. We provide   local-to-global characterizations of $p$-weakly peakless/convex functions on graphs.
These characterizations are similar to the characterizations of weakly peakless/convex functions \cite{BaCh_median,Ch_dpo}, however the proofs are more involved.
Finally, we prove that $p$-weakly peakless/convex functions on a graph $G$ are unimodal in the $p$th power $G^p$ of $G$.

\subsection{Convex and peakless functions on geodesic strings}
A \emph{geodesic string}  of a graph $G=(V,E)$ is a finite sequence of vertices $P=(u=w_0,w_1,\ldots ,w_n=v)$
included in a $(u,v)$-geodesic of $G$, i.e., for any $0<i<j<n$ we have $d(u,w_j)=d(u,w_i)+d(w_i,w_j)$ and $d(w_i,v)=d(w_i,w_j)+d(w_j,v)$.
Any geodesic is a geodesic string. Let $f: V\rightarrow {\mathbb R}$ be a function on $G$ and  let $P$  be a geodesic string. We denote by $f_P$
the piecewise-linear function on the segment $[0,d(u,v)]$ obtained by
considering $(0,f(w_0)),(d(u,w_1),f(w_1)),\ldots
,(d(u,w_{n-1}),f(w_{n-1})),(d(u,v),f(w_n))$
as points in the plane ${\mathbb R}^2$ and connecting  the consecutive points by segments.
We say that  a function $f$ is \emph{convex}
on a  geodesic  string $P$ if for all $i<j<k$,
\[f(w_j)\leq \frac{d(w_k,w_j)}{d(w_i,w_k)}\cdot f(w_i)+\frac{d(w_i,w_j)}{d(w_i,w_k)}\cdot f(w_k).\]
Equivalently, $f$ is convex on $P$ if  $f_P$ is convex.  We say that  $f$
is \emph{peakless} on a geodesic string $P$  if $0\leq i<j<k\leq n$ implies $f(w_j)\leq \max\{ f(w_i),f(w_k)\}$
and equality holds only if $f(w_i)=f(w_j)=f(w_k).$ Equivalently, $f$ is peakless on $P$ if  the piecewise-linear function $f_P$ is peakless.
We continue with a simple but useful local-to-global characterization of functions which are convex or peakless on geodesic strings.

\begin{lemma}\label{geodesic-string-convex}
  A function $f$ on a graph $G$ is convex on a geodesic string $P=(u=w_0,w_1,\ldots ,w_n=v)$
if and only if $f$ is locally-convex on $P$, i.e., for any $i=1,\ldots,n-1$, we
have \[f(w_i)\leq \frac{d(w_{i+1},w_i)}{d(w_{i-1},w_{i+1})}\cdot
f(w_{i-1})+\frac{d(w_{i-1},w_i)}{d(w_{i-1},w_{i+1})}\cdot f(w_{i+1}).\]
If $P$ is a geodesic, then $f$ is convex on $P$ if and only if $f(w_i)\leq \frac{1}{2}(f(w_{i-1})+f(w_{i+1}))$ for any $1\le i\le n-1$.
\end{lemma}

\begin{proof} By definition, $f$ is convex if an only if the piecewise-linear function $f_P$ is convex and, therefore, if and only if the epigraph $\epi(f_P)$ of $f_P$ is convex. The latter condition  is equivalent to the requirement that each  point $(i,f(w_i)), i=1,\ldots,n-1$ of ${\mathbb R}^2$ is a convex vertex
of $\epi(f_P)$, i.e., to the condition $f(w_i)\leq \frac{d(w_{i+1},w_i)}{d(w_{i-1},w_{i+1})}\cdot f(w_{i-1})+\frac{d(w_{i-1},w_i)}{d(w_{i-1},w_{i+1})}\cdot f(w_{i+1}).$
If $P$ is a geodesic of $G$, then $d(w_{i-1},w_{i+1})=2$ for any $i=1,\ldots,n-1$ and the previous inequality reduces to the inequality $f(w_i)\leq \frac{1}{2}(f(w_{i-1})+f(w_{i+1}))$.
\end{proof}

\begin{lemma}\label{geodesic-string-peakless} For a  function $f$ on  a graph $G$ and a geodesic string $P=(u=w_0,w_1,\ldots ,w_n=v)$ of $G$, the following conditions are equivalent:
\begin{itemize}
\item[(i)] $f$ is peakless on $P$;
\item[(ii)] $f$ is locally-peakless on $P$, i.e., for any $i=1,\ldots,n-1$, $f(w_i)\leq  \max \{ f(w_{i-1}),f(w_{i+1})\}$ and equality holds only if $f(w_{i-1})=f(w_i)=f(w_{i+1})$;
\item[(iii)] $P$ contains two (not necessarily distinct) vertices $w_{\ell},w_m$ such  that $f(w_0)>f(w_1)>\cdots>f(w_{\ell})=f(w_{\ell+1})=\cdots=f(w_m)<f(w_{m+1})<\cdots<f(w_n).$
\end{itemize}
\end{lemma}

\begin{proof} The implications (i)$\Rightarrow$(ii) and (iii)$\Rightarrow$(i) are trivial. To prove that (ii)$\Rightarrow$(iii),
suppose that $f$ is locally-peakless on $P$. Let $w_{\ell}$ and $w_m$ be respectively the minima closest to $u$ and to $v$  of the function $f$ on $P$.
First we assert that $f(w_{\ell})=f(w_{\ell+1})=\cdots=f(w_m)$. If this is not true, from the choice of $w_{\ell}$ and $w_m$ it follows that we can find two indices $\ell\le i<j<m$ such that
$f(w_{\ell})=\cdots = f(w_i)<f(w_{i+1})<\cdots <f(w_{j-1})<f(w_j)\ge
f(w_{j+1})$.
Consequently, $f$ is not locally-peakless on the triplet
$(w_{j-1},w_j,w_{j+1})$. This contradiction shows
that $f(w_{\ell})=f(w_{\ell+1})=\cdots=f(w_m)<f(w_{m+1})$. Now we show that $f(w_m)<f(w_{m+1})<\cdots<f(w_n)$ (the proof of the inequalities $f(w_0)>f(w_1)>\cdots>f(w_{\ell})$ is similar).
Suppose that $f(w_{m})<f(w_{m+1})<\cdots<f(w_{j-1})<f(w_j)$ and that $f(w_j)\ge f(w_{j+1})$. Consequently, $f(w_j)\ge \max\{f(w_{j-1}),f(w_{j+1})\}$ and since $f(w_{j-1})<f(w_{j})$, we do not have $f(w_{j-1}) = f(w_j) = f(w_{j+1})$ and thus we obtain a contradiction with the assumption that $f$ is locally-peakless on $P$. This contradiction shows that $f$ satisfies the monotonicity property (iii), concluding the proof.
\end{proof}

\subsection{Weakly convex and weakly peakless functions}
A function $f$ on a graph $G=(V,E)$ is {\it weakly convex} if  for any $u,v\in V$  $f$ is convex on some $(u,v)$-geodesic.
Analogously, $f$ is \emph{weakly peakless} if  $f$ is peakless on some $(u,v)$-geodesic for any $u,v\in V$.
Weakly convex functions were introduced in  \cite{ArSe} under the name ``$r$--convex functions''
via condition (ii) of Lemma \ref{weakly-convex} below; see also \cite{LeSeSo}. Weakly peakless functions were introduced and studied in
\cite{BaCh_median,Ch_dpo} under the name ``pseudopeakless functions''. 

Let $f$ be a function on  $G$ and $u,v\in V$ with $u\nsim v$.
We consider the following conditions on $f$:
\begin{enumerate}[$\WC(u,v)$:]
\item there exists $w\in I^{\circ}(u,v)$ such that
$f(w)\leq \frac{d(v,w)}{d(u,v)}\cdot f(u)+\frac{d(u,w)}{d(u,v)}\cdot f(v).$
\end{enumerate}

\begin{enumerate}[$\WP(u,v)$:]
\item there exists $w \in I^{\circ}(u,v)$ such that
  $f(w) \leq \max \left\{f(u),f(v)\right\}$, and  equality holds
  only if $f(u) = f(w) =f(v)$.
\end{enumerate}

\begin{enumerate}[$\Loz(u,v)$:]
\item there exist  $w,w' \in I^{\circ}(u,v)$ (not necessarily distinct) so that
  $f(w)+f(w') \leq  f(u)+f(v)$.
\end{enumerate}
A function $f$ on $G$ is a {\it lozenge function} if  $f$ satisfies $\Loz(u,v)$ for any $u,v\in V, u\nsim v$.

\begin{remark}  Lozenge functions on graphs were introduced and investigated in the context of median functions in \cite{BaCh_median}. They
represent a generalization of $L^{\#}$-convex functions \cite{FaTa} and $N$-convex functions \cite{HiIk} on particular classes of graphs,
both investigated in the theory of discrete convexity \cite{Hirai,Murota}.
\end{remark}

We continue by recalling the following  local-to-global characterizations  of weakly convex and weakly
peakless functions on graphs (which essentially follow from Lemmas \ref{geodesic-string-convex} and \ref{geodesic-string-peakless}).

\begin{lemma}[{\cite[Lemma 1]{BaCh_median}}]\label{weakly-convex} For a  function $f$ on a graph $G$ the following conditions are equivalent:

\begin{enumerate}
\item[(i)] $f$ is weakly convex;
\item[(ii)] $f$ satisfies  $\WC(u,v)$ for any $u,v\in V$ with $u\nsim v$;
\item[(iii)] $f$ satisfies $\WC(u,v)$ for any $u,v\in V$ with $d(u,v)=2$, i.e., any two vertices $u$ and $v$
at distance 2 have a common neighbour $w$ with $2f(w)\leq f(u)+f(v)$.
\end{enumerate}
\end{lemma}

\begin{lemma}[{\cite[Remark 1]{BaCh_median}}]\label{weakly-peakless} For a function $f$ defined on  a graph $G$ the following conditions are equivalent:

\begin{enumerate}
\item[(i)] $f$ is weakly peakless;
\item[(ii)] $f$ satisfies $\WP(u,v)$ for any $u,v\in V$ with $u\nsim v$;
\item[(iii)] $f$ satisfies $\WP(u,v)$ for any $u,v\in V$ with $d(u,v)=2$;
\item[(iv)] the composition $\alpha\circ f$
is weakly convex for some strictly isotone transformation $\alpha$ of the
reals ($\alpha (r)<\alpha (s)$ for any two reals $r,s$ such that $r<s$).
\end{enumerate}
Any weakly peakless function is unimodal.
\end{lemma}

The idea of the proof that (iii) implies (i) in Lemma \ref{weakly-convex} is as follows. According to \cite{BaCh_median},  for
nonadjacent vertices $u$ and $v$ consider a shortest $(u,v)$-path $P=(u=w_0,w_1,\ldots ,w_{n-1},w_n=v)$
such that $\sum_{i=0}^n f(w_i)$ is as
small as possible. From the choice of $P$ and condition (iii) it follows that $f(w_i)\leq \frac{1}{2}f(w_{i-1})+\frac{1}{2}f(w_{i+1})$ for each $1<i<n$.
Thus, $f$ is locally-convex on $P$ and by Lemma \ref{geodesic-string-convex} $f$ is convex on $P$.
Analogously, to prove that (iii) implies (i) in  Lemma \ref{weakly-peakless}, it suffices to note that $f$ is
locally-peakless on $P$ and to use Lemma \ref{geodesic-string-peakless}.

\begin{lemma}\label{lozenge} Each of the conditions $\WC(u,v)$ and $\Loz(u,v)$ implies  $\WP(u,v)$, thus weakly convex and
lozenge functions are weakly peakless.
\end{lemma}

\begin{proof} That $\WC(u,v)$ implies $\WP(u,v)$ is trivial. Suppose
  $f$ satisfies $\Loz(u,v)$ and $w,w' \in I^{\circ}(u,v)$ are two
  vertices such that $f(w)+f(w') \leq f(u)+f(v)$. Then clearly,
  $\min \{f(w),f(w')\}\le \max\{ f(u),f(v)\}$. Let $f(w)\le f(w')$ and
  $f(u)\le f(v)$. If $f(w)=f(v)$, since $f(w)+f(w') \leq f(u)+f(v)$,
  we conclude that $f(w')\le f(u)$ and thus
  $f(w) = f(w') = f(u) = f(v)$. \end{proof}

\begin{remark}  The motivation of using weakly peakless and lozenge functions is that often
it is easier to establish  $\WP(u,v)$ or $\Loz(u,v)$ than $\WC(u,v)$. On the other hand, the proof of
 \cite{BaCh_median}  that graphs with connected medians can be recognized in polynomial time uses $\WC(u,v)$.
\end{remark}

We also recall the following characterization of graphs with connected medians:

\begin{proposition}[{\cite[Proposition 1]{BaCh_median}}]\label{prop-cmed}
  For a graph $G$, the following conditions are equivalent:
  \begin{enumerate}[(1)]
  \item\label{th-cmed-1} $\lMedG(\pi) = \Med(\pi)$ for any profile $\pi$;
\item\label{th-cmed-2} $F_{\pi}$ is weakly convex  for any profile $\pi$;
  \item\label{th-cmed-3} $F_{\pi}$ is weakly peakless   for any profile $\pi$;
  \item\label{th-cmed-4} all level sets
    $L_{\preceq}(F_{\pi},\alpha)=\left\{x : F_{\pi}(x) \leq \alpha\right\}$ are isometric;
  \item\label{th-cmed-5} all median sets $\Med(\pi)$ are isometric;
  \item\label{th-cmed-6} all median sets $\Med(\pi)$ are connected.
  \end{enumerate}
\end{proposition}

\begin{remark}\label{connected->meshed} Meshed graphs are characterized by the weak convexity
of all radius functions $r_u$, $u\in V$. Since radius functions  are particular median functions, the
graphs with connected medians are meshed.
\end{remark}

To formulate the main result of \cite{BaCh_median}, we need the following notions.  For two vertices $u,v$, let $J(u,v)=\{z'\in V : I(z',u)\cap I(z',v)=\{z'\}\}$.
Let also $M(u,v)=\{ z'\in J(u,v): d(u,z')=d(v,z')\}$ and $J^{\circ}(u,v)=J(u,v)\setminus M(u,v)$.

Intuitively, the set $J(u,v)$ can be interpreted in the following way. For any vertex $z$, any quasi-median of the triplet $z,u,v$ is a metric triangle $z'u'v'$ with
the side $I(u',v')$ included in $I(u,v)$ and the opposite vertex $z'$ in $I(z,u)\cap I(z,v)$. Then the set $J(u,v)$ is exactly the set of vertices $z'$ of all
quasi-medians $z'u'v'$ for all vertices $z$ of $G$. If $z\in I(u,v)$, then obviously $z'=z=u'=v'$, whence $I(u,v)\subseteq J(u,v)$. The sets $J(u,v)$ have a more precise
form for vertices $u$ and $v$ in meshed graphs. In such a graph, any quasi-median $z'u'v'$ is an equilateral metric triangle. Since $2=d(u,v)=d(u,u')+d(u',v')+d(v',v)$, each such
metric triangle $z'u'v'$ either has size 0, 1, or 2. Then $M(u,v)$ consists of all $z'$ such that $z'u'v'$ has size 0 or 2, i.e., of vertices of $I^\circ(u,v)$ and
all vertices $z'$ at distance 2 from $u$ and  $v$ and not having any neighbor in $I^{\circ}(u,v)$. Analogously, $J^{\circ}(u,v)=J(u,v)\setminus M(u,v)$ consists of all vertices $z'$
such that $z'u'v'$ has size 1. Since $d(u,v)=2$, this is possible only if either $u'=u$, $v'\in I^\circ(u,v)$, and $z'\sim u,v'$ or $v'=v$, $u'\in I^\circ(u,v)$, and $z'\sim u',v$.

We conclude this section by recalling the main result of \cite{BaCh_median}.

\begin{theorem}[{\cite[Theorem]{BaCh_median}}]\label{th-cmed}
A graph  $G$ is a graph with connected medians if and only if $G$ is meshed and for any vertices $u$ and $v$ at distance 2
there exist a nonempty subset $S$ of $I^{\circ}(u,v)$ and  a weight function $\eta$ with nonempty support included in $S$
having the following two properties:

\begin{itemize}
\item[$(\alpha)$] every vertex $s \in S$ has a companion $t\in S$ (not necessarily distinct
from $s$) such that $d(s,x)+d(t,x)\leq d(u,x)+d(v,x) \mbox{ for  all } x\in M(u,v),$
and $\eta(s)=\eta(t)$ whenever $d(s,t)=2;$

\item[$(\beta)$] the joint weight of the neighbours of any $x\in J^{\circ}(u,v)$ from
$S$ is always at least half the total weight of $S:$
$\sum_{s\in S\cap N(x)}\eta (s)\ge \frac{1}{2}\sum_{s\in S} \eta(s) \mbox{ for all } x\in J^{\circ}(u,v).$
\end{itemize}
\end{theorem}

\subsection{$p$-Weakly peakless and $p$-weakly convex functions}

We consider a generalization of weakly convex, weakly peakless, and lozenge  functions by requiring that the
conditions $\WC(u,v), \WP(u,v)$, and $\Loz(u,v)$ are satisfied only for pairs of vertices at distance strictly larger than $p$.

Let $G=(V,E)$ be a graph and $p$ be a positive integer. A \emph{$p$-geodesic} (or a \emph{step-$p$ geodesic})
between two vertices $u,v$  is a geodesic string $P=(u=w_0,w_1,\ldots,w_{n-1},w_n=v)$ of $G$ such that
$d_G(w_i,w_{i+1})\le p$ for any $i=0,\ldots,n-1$.  A \emph{taut $p$-geodesic} is a $p$-geodesic such that $p+1<d_G(w_{i-1},w_{i+1})\le 2p$ for any $i=1,\ldots,n-1$.
The definition implies that any $p$-geodesic  is a path of the graph $G^p$ and
that any taut $p$-geodesic is an induced path of $G^p$. Notice however that a
$p$-geodesic
is not necessarily a geodesic of $G^p$. Conversely, not every $(u,v)$-geodesic $Q$ of $G^p$ is a $p$-geodesic of $G$ because $Q$ is not necessarily contained in the interval $I(u,v)$.
We say that a subset of vertices $S$ of $G$ is {\it $p$-step isometric} (\emph{$p$-isometric} for short) if for any two vertices $u,v\in S$
with $d(u,v)\ge p+1$, there exists a vertex $w\in I^{\circ}(u,v)\cap S$. We say that a subset of vertices
$S$ of $G$ is {\it $p$-step connected} (or {\it connected in $G^p$}) if for any two vertices $u,v\in S$ there exist vertices
$u_0:=u,u_1,\ldots,u_k,u_{k+1}:=v$ of $S$ such that $d(u_i,u_{i+1})\le p$ for
any $i=0,\ldots,k$, or, equivalently, if $S$ induces
a connected subgraph of $G^p$.

A function $f$ on $G$ is {\it $p$-weakly peakless}
(respectively, {\it $p$-weakly convex}) if any pair of vertices $u,v$ of $G$
can be connected by a $p$-geodesic
along which $f$ is peakless (respectively, convex). In view of Lemmas \ref{weakly-convex}
and \ref{weakly-peakless}, the $1$-weakly peakless (respectively, $1$-weakly convex) functions are exactly
the weakly peakless (respectively, weakly convex functions). Again, $p$-weakly convex functions are $p$-weakly peakless.

A function $f$ is \emph{locally-$p$-weakly peakless} (respectively, \emph{locally-$p$-weakly convex}) on $G$
if it satisfies $\WP(u,v)$ (respectively, $\WC(u,v)$) for any $u,v$ such that $p+1 \leq d(u,v) \leq 2p$.

Next, we establish a local-to-global characterization of $p$-weakly peakless function, which is an analog  of Lemma \ref{weakly-peakless}.
We also show that $p$-weakly peakless functions are unimodal in $G^p$.

\begin{theorem}\label{wp-ltg} For a function $f$ on a graph $G=(V,E)$ the following conditions are equivalent:
\begin{itemize}
\item[(i)] $f$ is $p$-weakly peakless;
\item[(ii)] $f$ satisfies the property $\WP(u,v)$ for any pair of vertices $u,v$ such that $d(u,v)\ge p+1$;
\item[(iii)] $f$ is locally-$p$-weakly peakless.
\end{itemize}
Any $p$-weakly peakless function $f$ on a graph $G=(V,E)$ is unimodal on $G^p$.
\end{theorem}

\begin{proof} The implications (i)$\Rightarrow$(ii) and (ii)$\Rightarrow$(iii) are trivial.
Now we prove (ii)$\Rightarrow$(i). Let $u,v$ be two vertices of $G$ and we proceed by induction on
$d(u,v)$. If $d(u,v)\le p$, then the string $(u,v)$ is a $p$-geodesic along which $f$ is peakless.
Now, let $d(u,v)\ge p+1$. Let $w$ be a vertex minimizing  $f$ on $I^{\circ}(u,v)$ and $w'$ be a vertex of $I^{\circ}(u,v)$
satisfying the condition $\WP(u,v)$. From the choice of the vertices $w$ and $w'$ we conclude that either (1) $f(w)\le f(w')<\max\{ f(u),f(v)\}$ or (2) $f(w)\le f(w')=f(u)=f(v)$.
This implies that either  $f(w)<\max\{ f(u),f(v)\}$ or $f(w)=f(w')=f(u)=f(v)$.
Since $d(u,w)<d(u,v)$ and  $d(w,v)<d(u,v)$, by induction hypothesis the pairs $u,w$ and $w,v$ can be connected by
$p$-geodesics $P'=(u=w_0,w_1,\ldots,w_i=w)$  and $P''=(w=w_i,w_{i+1},\ldots,w_{n-1},w_n=v)$ so that $f$ is
peakless along $P'$ and $P''$.  Concatenating $P'$ and $P''$, we  obtain a $p$-geodesic
$P=(u=w_0,w_1,\ldots,w_i=w,w_{i+1},\ldots,w_{n-1},w_n=v)$. We assert that there exists a $p$-geodesic $P_0$ between
$u$ and $v$ included in $P$ such that $f$ is peakless on $P_0$ (in most of the cases presented below, we have $P_0=P$).

First suppose that $f(w)=f(u)=f(v)$. From the choice of $w$ and of the $p$-geodesics $P'$ and $P''$, we conclude that $f$ is constant on all
vertices of $P$, thus $f$ is peakless on $P$.
Now suppose that $f(w)<\max\{ f(u),f(v)\}$ and let $f(u)\le f(v)$. Since $f(w)<f(v)$, from the choice of $w$ and $P''$ and
by Lemma \ref{geodesic-string-peakless}(iii), we conclude that when moving on  $P''$ from $w$ to $v$, the function $f$ first is constant and then it  strictly
increases: $f(w_i)=\cdots=f(w_m)<f(w_{m+1})<\cdots <f(w_n)$ (it may happen that $m=i$). Analogously, if $f(w)\le f(u)$, while moving from $u$ to $w$, the
function $f$ first strictly decreases and then it is constant: $f(w_0)<f(w_1)<\cdots <f(w_{\ell})=f(w_{\ell+1})=\cdots=f(w_i)$ (it may happen that
$\ell=0$,  $\ell=i$, or $\ell=0=i$). In this case, we conclude that on the $p$-geodesic $P$ the function $f$ satisfies the condition (iii) of
Lemma \ref{geodesic-string-peakless}, thus $f$ is peakless on $P$. Finally, suppose that $f(w)>f(u)$. Since $w$ is a minimum vertex of
$f$ on $I^{\circ}(u,v)$, applying Lemma \ref{geodesic-string-peakless}(iii) to $P'$, we conclude that $P'=(u,w)$ and thus $w=w_1$.
If $d(u,w_m)\le p$, then $P_0=(u=w_0,w_m,\ldots,w_{n-1},w_n=v)$ is a $p$-geodesic between $u$ and $v$ included in $P$ and such that $f$ is strictly increasing
on $P_0$ while moving from $u$ to $v$. consequently, $f$ is peakless on $P_0$. Now suppose that $d(u,w_m)>p$, i.e., $m>i=1$. Applying $\WP(u,w_m)$ we can find a vertex
$w'\in I^{\circ}(u,w_m)\subset I^{\circ}(u,v)$ such that $f(w')<\max\{ f(u),f(w_m)\}$ or $f(w')=f(u)=f(w_m)$. Since $f(w_m)=f(w_1)=f(w)>f(u)$, the second
possibility cannot occur. But the first possibility implies that $f(w')<f(w_m)=f(w)$, contrary to the minimality choice of $w$.
This shows that the case $d(u,w_m)>p$ is impossible. This concludes that $f$ is peakless on a $p$-geodesic $P_0\subseteq P$. 
The proof of the implication  (iii)$\Rightarrow$(ii) is based on the following lemma:

\begin{lemma}\label{lemma-Gp-cmed-2}
  Let $f$ be a locally $p$-weakly peakless
  function on $G$.  Then for any $u, v \in V$ with $f(u)<f(v)$ and
  $d(u,v)>p$, there exists $w \in I^{\circ}(u,v)$ with $d(v,w)\le p$ such that
  $f(w) < f(v)$.  \end{lemma}

\begin{proof}
Suppose that the property does not hold and pick two vertices $u,v \in V$ at minimum distance $d(u,v)>p$ such that $f(u) < f(v)$
and $f(w) \geq f(v)$ for any $w\in I^{\circ}(u,v)\cap B_p(v)$. Since $f$ is locally-$p$-weakly peakless, necessarily,
$d(u,v) > 2p$. Indeed, if $p+1\le d(u,v)\le 2p$, by  $\WP(u,v)$ there exists a vertex $w_1 \in I^{\circ}(u,v)$ such that
$f(w_1)<f(v)$.  If $d(w_1,v)>p$, applying $\WP(w_1,v)$, we will find a vertex $w_2\in I^{\circ}(w_1,v)$ such that $f(w_2)<f(v)$.
Continuing this way, we will eventually find a sequence of vertices $w_0=u,w_1,\ldots,w_k$ such that $w_i\in I^{\circ}(w_{i-1},v)$,
$f(w_i)<f(v)$ for all $i=1,\ldots,k$, and $d(w_k,v)\le p$. Consequently, we can suppose that $d(u,v)>2p$.

Consider a vertex $w\in I^{\circ}(u,v)\cap B_p(v)$ minimizing the function $f$ in this intersection.
Note that $p<d(u,v)-d(v,w)=d(u,w)<d(u,v)$. Since $f(w) \geq f(v) >f(u)$, by the choice of $u,v$, there exists a
vertex $x' \in I(w,u) \cap B_{p}(w)$ such that $f(x')<f(w)$. Observe that $x' \in I(w,u) \subseteq I(v,u)$ and that
$d(v,x') \leq d(v,w)+d(w,x') \leq 2p$.  Let $x$ be the closest vertex from $v$  in  $I^{\circ}(u,v)$ such that $f(x) < f(w)$. By our
choice of $w$, $x \notin B_p(v)$ and thus $p+1 \leq d(v,x) \leq 2p$. Since $f$ is locally-$p$-weakly peakless, by property $\WP(v,x)$,
there exists a vertex $z \in I^{\circ}(v,x)$ such that $f(z) \leq \max \{f(v),f(x)\}$ and equality holds only if $f(v) = f(x)$. Consequently, either
$f(z)< \max \{f(v),f(x)\} \leq f(w)$ or $f(z)=f(v)=f(x) < f(w)$.  In  both cases, $f(z) <f(w)$ and $d(v,z) < d(v,x)$, contradicting our
choice of $x$.  This ends the proof of the lemma.
\end{proof}

We can now prove (iii)$\Rightarrow$(ii). Let $u,v$ be two vertices of
$G$ with $d(u,v)\ge p+1$.  If $d(u,v)\le 2p$, then we can apply the
condition $\WP(u,v)$, because $f$ is locally-$p$-weakly peakless. Thus
we can suppose that $d(u,v)>2p$.  If $f(u)<f(v)$, by
Lemma~\ref{lemma-Gp-cmed-2}, there exists a vertex
$w\in I^{\circ}(u,v)\cap B_p(v)$ such that $f(w)<f(v)$ and we are
done. Now suppose that $f(u)=f(v)$. Let $w$ be a vertex of
$I^{\circ}(u,v)$ minimizing $f(w)$. If $f(w)\le f(u)=f(v)$, then we
are done. Therefore, suppose that $f(w)>f(u)=f(v)$.  Since
$d(u,v)>2p$, either $d(u,w)>p$ or $d(v,w)>p$, say the first. By
Lemma~\ref{lemma-Gp-cmed-2}, there exists
$z\in I^{\circ}(u,w)\cap B_p(w)$ such that $f(z)<f(w)$.  Since
$z\in I^{\circ}(u,w)\subset I^{\circ}(u,v)$, this contradicts the
choice of the vertex $w$. Consequently, $f$ satisfies the property
$\WP(u,v)$ for all $u,v$ such that $d(u,v) \geq p+1$.

This establishes the equivalence between the conditions (i), (ii), and (iii).
It remains to show that any $p$-weakly peakless function $f$ is unimodal on $G^p$.
Indeed, let $u$ be a global minimum and $v$ a local minimum of $f$ on $G$. Consider a $p$-geodesic
$P$ between $u$ and $v$ along which $f$ is peakless. Let $w$ be the predecessor  of $v$ on $P$. Then
$f(w)\leq \max \{ f(u),f(v)\}=f(v)$. Since $v$ is a local minimum, $f(w)=f(v)$,
whence $f(u)=f(v)=f(w),$ as required. This concludes the proof of the theorem.
\end{proof}

\begin{remark} If
$p+1 \leq d(u,v) \leq 2p$ and $w\in I^{\circ}(u,v)$ satisfies the condition
$\WP(u,v)$, the geodesic string $(u,w,v)$ is not necessarily a  $p$-geodesic. This explain the difficulty in the proof of (iii)$\Rightarrow$(i).
\end{remark}

A function $f$ is called a \emph{$p$-lozenge function} if $f$ satisfies $\Loz(u,v)$ for any vertices $u,v$ with $d(u,v)\ge p+1$.
A function $f$ is  a \emph{locally $p$-lozenge function} if $f$ satisfies $\Loz(u,v)$ for any vertices $u,v$ with $p+1\le d(u,v)\le 2p$.
Since $\Loz(u,v)$ implies $\WP(u,v)$ (Lemma \ref{lozenge}), from Theorem \ref{wp-ltg} we obtain the following corollary:

\begin{corollary}\label{lozenge->wp} Any $p$-lozenge or locally-$p$-lozenge function $f$ on a graph $G=(V,E)$ is $p$-weakly peakless.
\end{corollary}

Another consequence of Theorem \ref{wp-ltg} is the following corollary:

\begin{corollary}\label{wp-Gp-level} All level sets  $L_{\preceq}(f,\alpha)$ of a $p$-weakly peakless function $f$ on a graph $G$ are $p$-isometric
(and thus induce connected subgraphs of $G^p$). In particular, $\argmin(f)$ is $p$-isometric.
\end{corollary}

We leave the next question as an open question:

\begin{question} Do the level sets of any $p$-weakly peakless function induce isometric subgraphs of $G^p$?
\end{question}

The difficulty in dealing with this question comes from the fact that the $p$-geodesics along which a $p$-weakly peakless function $f$ is peakless are not
necessarily geodesics (shortest paths) of the graph $G^p$.

We continue with an analog of Theorem \ref{wp-ltg} for $p$-weakly
convex functions.  Its proof in the ``grandes lignes'' is similar to
the proof of Theorem \ref{wp-ltg}, but it is technically more
involved.

\begin{theorem}\label{wc-ltg} For a function $f$ on a graph $G=(V,E)$ the following conditions are equivalent:
\begin{itemize}
\item[(i)] $f$ is $p$-weakly convex;
\item[(ii)] $f$ satisfies the property $\WC(u,v)$ for any pair of vertices $u,v$ such that $d(u,v)\ge p+1$;
\item[(iii)] $f$ is locally $p$-weakly convex.
\end{itemize}
\end{theorem}

\begin{proof}
The implications (i)$\Rightarrow$(ii) and (ii)$\Rightarrow$(iii) are trivial. To prove (ii)$\Rightarrow$(i),
for a pair $(u,v)$ of vertices of $G$, let $\alpha_{uv}:= (f(v)-f(u))/d(u,v)$.  Then $\alpha_{vu}=-\alpha_{uv}$, $f(v)=f(u)+\alpha_{uv}d(u,v)$, and $f(u)=f(v)-\alpha_{uv}d(u,v).$ For any vertex $t\in I^{\circ}(u,v),$ we denote by $g_{uv}(t)$ the difference between $f(t)$ and $f(u) + \alpha_{uv} d(u,t)=f(v) - \alpha_{uv} d(t,v)$. With these notations, the property WC($u,v$) can be rewritten as follows: \emph{there exists $w \in I^{\circ}(u,v)$ such that $g_{uv}(w)\le 0$.}

The beginning of the proof is analogous to the proof of Theorem \ref{wp-ltg}. The first difference is that we choose a vertex $w\in I^\circ(u,v)$ minimizing $g_{uv}(w)$ instead of minimizing $f(w).$ Then, analogously
to the proof of Theorem \ref{wp-ltg}, we apply induction hypothesis to derive a
$p$-geodesic $P'$ between $u$ and $w$ and a $p$-geodesic $P''$ between $w$ and
$v$ so that $f$ is convex along $P'$ and along $P''.$
We claim that $f$ is also convex on the $p$-geodesic $P$ obtained by concatenating $P'$ and $P''.$ To prove this assertion it suffices to show that for
any pair of  internal vertices $x,y$ of $P$ with $x\in P'$ and $y\in P'',$ and the vertex $w$ defined above,
we have  $d(x,y)f(w)\le d(y,w)f(x)+d(x,w)f(y).$

By the choice of $w$, any convex combination of $g_{uv}(x)$ and $g_{uv}(y)$ is at least $g_{uv}(w).$
Therefore $d(y,w)g_{uv}(x)+d(x,w)g_{uv}(y) \ge d(x,y)g_{uv}(w).$ In this expression we can replace $g_{uv}(x)$ by $f(x)-(f(u)+\alpha_{uv}d(u,x))$ and do the same for
$g_{uv}(y)$ and $g_{uv}(w).$ We obtain:
\begin{align*}
  d(x,y) g_{uv}(w)
  & \le d(w,y) g_{uv}(x) + d(x,w) g_{uv}(y),\\
  d(x,y) \left[f(w) - \left(f(u) + \alpha_{uv} d(u,w)\right)\right]
  & \le   d(w,y) \left[f(x) - (f(u) + \alpha_{uv} d(u,x))\right] \\
  &\quad  + d(x,w) \left[f(y) - (f(u) + \alpha_{uv} d(u,y))\right],\\
  d(x,y) f(w) - d(x,y) f(u) - d(x,y) d(u,w) \alpha_{uv}
  & \le   d(w,y) f(x) + d(x,w) f(y) - \left[d(x,w) + d(w,y)\right] f(u) \\
  &\quad  -\left[d(w,y)d(u,x) +  d(x,w)d(u,y)\right] \alpha_{uv}.\end{align*}
Since $d(x,y)=d(x,w)+d(w,y),$ the terms involving $f(u)$ cancel. Now, consider the terms having $\alpha_{uv}$. Replacing $d(x,y)$ by $d(x,w)+d(w,y),$ $d(u,w)$ by $d(u,x)+d(x,w)$, and $d(u,y)=d(u,x)+d(x,w)+d(w,y),$ it easy to check that  $d(x,y)d(u,w)=d(w,y)d(u,x)+d(x,w)d(u,y),$ and therefore the terms involving $\alpha_{uv}$ also cancel. Consequently, we obtain that $d(x,y)f(w)\le d(w,y)f(x)+d(x,w)f(y),$ showing  that $f$ is indeed convex along $P$.
This concludes the proof of the implication  (ii)$\Rightarrow$(i).

The proof of the implication (iii)$\Rightarrow$(ii) is identical to
the proof of (iii)$\Rightarrow$(ii) in Theorem~\ref{wp-ltg} where
Lemma~\ref{lemma-Gp-cmed-2} is replaced by the following lemma.

\begin{lemma}\label{lemma-Gp-cmed-3}
  Let $f:V\rightarrow {\mathbb R}$ be a locally $p$-weakly convex
  function on $G$.  Then for any $u, v \in V$ with $f(u)\le f(v)$ and
  $d(u,v)>p$, there exists $w \in I^{\circ}(u,v)\cap B_p(v)$ such that
  $g_{uv}(w) \le 0$.  \end{lemma}

\begin{figure}[]
	\begin{center}
		\includegraphics[width = 0.55\textwidth, page=4]{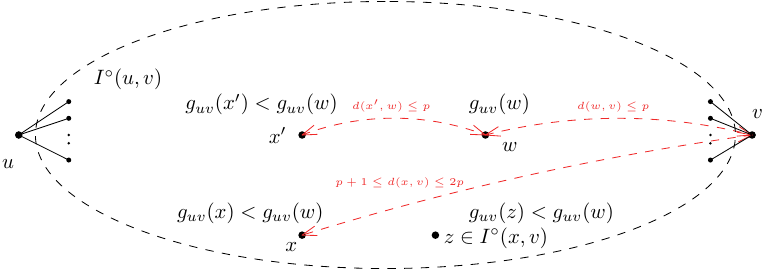}
	\end{center}
	\caption{An illustration of the proof of Lemma
	\ref{lemma-Gp-cmed-3}}\label{fig-Gp-cmed-3}
\end{figure}

\begin{proof}
  The proof is similar to the proof of Lemma \ref{lemma-Gp-cmed-2} but
  it requires some additional computations (an illustration of the
  proof is given in Fig.~\ref{fig-Gp-cmed-3}).  As in Lemma
  \ref{lemma-Gp-cmed-2}, we proceed by contradiction and suppose that
  $u,v\in V$ are two vertices at minimum distance such that
  $g_{uv}(x)>0$ for all $x\in I^{\circ}(u,v) \cap B_p(v)$. For the
  same reasons as in the proof of Lemma~\ref{lemma-Gp-cmed-2}, we can
  assume that $d(u,v)>2p$.  Then, we consider a vertex $w$ of
  $I^{\circ}(u,v)\cap B_p(v)$ minimizing $g_{uv}(w)$ (instead of
  $f(w)$). Observe that $d(u,v) > d(u,w) \geq d(u,v) - d(v,w) > p$.
Since $g_{uv}(w) > 0$, we have $\alpha_{uw} > \alpha_{uv}$. By
  minimality of $d(u,v)$ and since $d(u,w) > p$, there exists a vertex
  $x' \in I^{\circ}(u,w)\cap B_p(w)$ such that
  $f(x')\le f(u)+\alpha_{uw} d(u,x').$ The following sequence of
  inequalities show that $g_{uv}(x')<g_{uv}(w)$.
\begin{align*}
  g_{uv}(x')
  &=f(x')-(f(u)+\alpha_{uv}d(u,x'))\\
  &\le f(u)+\alpha_{uw}d(u,x')-(f(u)+\alpha_{uv}d(u,x')) = (\alpha_{uw}-\alpha_{uv})d(u,x')\\
  &<(\alpha_{uw}-\alpha_{uv})d(u,w) = (f(u)+\alpha_{uw}d(u,w))-(f(u)+\alpha_{uv}d(u,w))\\
  &=f(w)-(f(u)+\alpha_{uv}d(u,w))=g_{uv}(w).
\end{align*}
As in the proof of Lemma \ref{lemma-Gp-cmed-2}, pick a vertex $x$ closest to $v$ in  $I^{\circ}(u,v)$ such that $g_{uv}(x)<g_{uv}(w).$
Since $d(x',v)\le 2p,$ by our choice of $x$ and $w,$ $p+1\le d(x,v)\le d(x',v) \le 2p.$
Since $f$ is locally $p$-weakly convex, there exists a vertex  $z \in I(x,v)$ such that $f(z)\le f(v)-\alpha_{xv}d(z,v).$
Since $\alpha_{xv}\le \alpha_{zv}$ by the choice of $z$ and  $\alpha_{zv}\le \alpha_{uv}$ because $g_{uv}(z)\ge g_{uv}(w)> 0$ by the choice of $x$, we obtain:
\begin{align*}
  g_{uv}(z)
  &= f(z)-(f(v)-\alpha_{uv}d(z,v))\\
  &= f(v)-\alpha_{zv}d(z,v)-(f(v)-\alpha_{uv}d(v,z)) = (\alpha_{uv}-\alpha_{zv})d(z,v) \\
  &\le (\alpha_{uv}-\alpha_{zv})d(x,v) = (f(v)-\alpha_{zv}d(x,v))-(f(v)-\alpha_{uv}d(x,v)) \\
  &\le (f(v)-\alpha_{xv}d(x,v))-(f(v)-\alpha_{uv}d(x,v)) \\
  &=f(x)-(f(v)-\alpha_{uv}d(x,v)) =g_{uv}(x)<g_{uv}(w).
\end{align*}
  The inequality $g_{uv}(z)<g_{uv}(w)$ contradicts the choice of $w$ if $z\in B_p(v)$ and the choice of $x$ otherwise. This contradiction proves the lemma.
\end{proof}

This establishes the equivalence between (i), (ii) and (iii) and finishes the proof.
\end{proof}

\section{Graphs with $G^p$-connected medians}

In this section, we characterize the graphs $G=(V,E)$ such that for any
profile $\pi$, the median set $\Med_G(\pi)$ induces a connected
subgraph in the $p$th power $G^p$ of $G$, thus extending \cite[Proposition 1]{BaCh_median}
for the case $p=1$. Using this result, we show that for any $p\ge 1$, the graphs
with $G^p$-connected medians can be recognized in polynomial time. We also establish
that the class of graphs with $G^p$-connected medians is closed under Cartesian products
and gated amalgams.

\subsection{Characterization of graphs with $G^p$-connected medians}
In this subsection, we characterize the graphs with $G^p$-connected medians, thus generalizing Proposition \ref{prop-cmed}. For each profile $\pi$, the
median function $F_{\pi}(x)=\sum_{v\in V} \pi(v)d_G(x,v)$ and the median set $\Med(\pi)=\argmin(F_{\pi})$ are considered in the graph $G$, i.e.,
with respect to the distance function of $G$. However, we denote by $\lpMedG(\pi)$ the set of all local minima of the function $F_{\pi}$ in the graph $G^p$:
namely, a vertex $x$ belongs to
$\lpMedG(\pi)$ if and only if for any neighbor $y$ of $x$ in $G^p$ (i.e., for any vertex $y$ such that $d_G(x,y)\le p$) we have  $F_{\pi}(x)\le F_{\pi}(y)$.

\begin{theorem}\label{th-cmed-p}
  For a graph $G$ and an integer $p\ge 1$, the following conditions are equivalent:
  \begin{enumerate}[(1)]
  \item\label{th-cmedp-0} $\lpMedG(\pi) = \Med(\pi)$ for any integer profile $\pi$;
  \item\label{th-cmedp-1} $\lpMedG(\pi) = \Med(\pi)$ for any profile $\pi$;
\item\label{th-cmedp-2} $F_{\pi}$ is $p$-weakly convex  for any profile $\pi$;
  \item\label{th-cmedp-3} $F_{\pi}$ is $p$-weakly peakless   for any profile $\pi$;
  \item\label{th-cmedp-4} all level sets
    $L_{\preceq}(F_{\pi},\alpha)=\left\{x : F_{\pi}(x) \leq \alpha\right\}$ of $F_{\pi}$ are $p$-isometric;
  \item\label{th-cmedp-5} all median sets $\Med(\pi)$ are $p$-isometric;
  \item\label{th-cmedp-6} all median sets $\Med(\pi)$ are connected in $G^p$ (i.e., are $p$-connected).
  \end{enumerate}
\end{theorem}

\begin{proof}
We start with the following lemma, which can be viewed as an analogue of~\cite[Lemma 2]{BaCh_median} for $p=1$:

\begin{lemma}\label{nonweaklyconvex} If the median function $F_{\pi}$ is not $p$-weakly convex for some profile (respectively, integer profile)  $\pi$ on $G$, then there
exists a profile (respectively, an integer profile) $\pi^+$ and vertices $u,v$ at distance $p+1\le d(u,v)\le 2p$ such that $\Med(\pi^+)=\{ u,v\}.$
  \end{lemma}

\begin{proof}
If $F_{\pi}$ is not $p$-weakly convex, then by Theorem \ref{wc-ltg} $F_{\pi}$ is not locally $p$-weakly convex.
Therefore,  there exist two vertices $u,v$ at distance $p+1\le d(u,v)\le 2p$ such that for any vertex
$w \in I^{\circ}(u,v)$, we have $F_{\pi}(w)>F_{\pi}(u)+\alpha_{uv}d(u,w)$.  Suppose without loss of generality
that $F_{\pi}(v)\ge F_{\pi}(u)$. Let $k:=d(u,v)$, $\epsilon:=F_{\pi}(v)-F_{\pi}(u)$, and $\mu := kF_{\pi}(v)+1$. Clearly, if $\pi$ is an integer profile,
$\epsilon$ is a nonnegative integer and $\mu$ is a positive integer. Define a profile (respectively, an integer profile) $\pi^+$ by setting
\begin{equation*}
    \pi^+(x) =
    \begin{cases}
      k\pi(u)+ \mu & \text{if } x =u,\\
      k\pi(v)+ \mu +\epsilon & \text{if } x =v,\\
      k\pi(x) & \text{otherwise.}
    \end{cases}
  \end{equation*}

  Then, for any vertex $x$, we have
  \[F_{\pi^+}(x)= kF_{\pi}(x)+ (d(u,x)+d(v,x))\mu +
    d(v,x)\epsilon.\]

  Consequently, we have
  \[F_{\pi^+}(u)=kF_{\pi}(u)+k\mu+k\epsilon=k(F_{\pi}(u)+\epsilon)
    +k\mu=kF_{\pi}(v)+k\mu=F_{\pi^+}(v).\]

  For any vertex $x\notin I(u,v)$, $d(u,x)+d(v,x) \geq k+1$ and since
  $\mu > kF_{\pi}(v)$, we have
  \[F_{\pi^+}(x) \geq (k+1)\mu > k\mu + kF_{\pi}(v) =
    F_{\pi^+}(v).\]

  For any vertex $w\in I^{\circ}(u,v)$, we have $d(u,w)+d(v,w) = k,$ and thus
  \[F_{\pi^+}(w)= kF_{\pi}(w)+ k\mu + \epsilon d(w,v).\]

  Since $F_{\pi}$ is not $p$-weakly convex, $F_{\pi}(w) > F_{\pi}(v)-\frac{\epsilon}{k} d(v,w),$
  \[F_{\pi^+}(w)= kF_{\pi}(w)+ k\mu + d(v,w)\epsilon >
    k(F_{\pi}(v)-\frac{\epsilon}{k} d(v,w))+ k\mu +\epsilon d(v,w) = F_{\pi^+}(v).\]

  Therefore the only minima of $F_{\pi^+}$ are $u$ and $v$ and thus
  $\Med(\pi^+) = \{u,v\}$, concluding the proof.
\end{proof}

The implications  (\ref{th-cmedp-2})$\Rightarrow$(\ref{th-cmedp-3})$\Rightarrow$ (\ref{th-cmedp-4})
$\Rightarrow$ (\ref{th-cmedp-5}) $\Rightarrow$ (\ref{th-cmedp-6})
are trivial. The implication (\ref{th-cmedp-2}) $\Rightarrow$
(\ref{th-cmedp-1}) follows from the fact that $p$-weakly convex functions are $p$-weakly peakless and that $p$-weakly peakless
functions are unimodal on $G^p$ (Theorem \ref{wp-ltg}): therefore the set $\lpMedG(\pi)$ of local minima of $F_{\pi}$ in $G^p$ coincides
with the set of global minima of $F_\pi$, yielding $\Med(\pi)=\lpMedG(\pi)$.
The implication (\ref{th-cmedp-6})
$\Rightarrow$ (\ref{th-cmedp-2}) follows from Lemma
\ref{nonweaklyconvex}. Finally, we prove that (\ref{th-cmedp-1})
$\Rightarrow$ (\ref{th-cmedp-2}). Suppose by way of contradiction
that some function $F_{\pi}$ is not $p$-weakly convex. By Lemma
\ref{nonweaklyconvex} there exists a profile $\pi^+$ and vertices
$u,v$ at distance $p+1\le d(u,v)\le 2p$ such that
$\Med(\pi^+)=\{ u,v\}$, i.e., such that
$F_{\pi^+}(x) \geq F_{\pi^+}(u) +1 = F_{\pi^+}(v) +1$ for all
$x \notin\{u,v\}$. Let $k:=d(u,v)$ and consider the profile $\pi'$
defined by $\pi'(u) = k \pi^+(u)+1$ and $\pi'(x) = k \pi^+(x)$ for any
$x \neq u$.
Observe that
$F_{\pi'}(u) = kF_{\pi^+}(u) < kF_{\pi^+}(u) +k = kF_{\pi^+}(v) +k =
F_{\pi'}(v)$. Note also that for any $x \notin \{u,v\}$,
$F_{\pi'}(x) = k F_{\pi^+}(x) + d(x,u) \geq k (F_{\pi^+}(v) + 1) +1
= F_{\pi'}(v)+1$. Therefore both $u$ and $v$ are local minima of
$F_{\pi'}$ in $G^p$, but $v\notin \mbox{Med} (\pi')$.
This establishes the equivalence of conditions (\ref{th-cmedp-1})-(\ref{th-cmedp-6}).

With the same proof, but applying Lemma \ref{nonweaklyconvex} with integer profiles, we conclude that
condition (\ref{th-cmedp-0}) is equivalent to each of the conditions (\ref{th-cmedp-2})-(\ref{th-cmedp-6}) in which we
require that the profile $\pi$ is integer. It remains to show that (\ref{th-cmedp-0}) is equivalent to the conditions (\ref{th-cmedp-1})-(\ref{th-cmedp-6}) for all profiles.
The implication (\ref{th-cmedp-1})$\Rightarrow$(\ref{th-cmedp-0}) is trivial.
To prove (\ref{th-cmedp-0})$\Rightarrow$(\ref{th-cmedp-4}), suppose by way of contradiction  that for some profile $\pi$ and some
$\alpha>0$, the level set  $L_{\preceq}(F_{\pi},\alpha)$ is not $p$-isometric.
Therefore, there exist two vertices $u,v$ with $d(u,v)>p$ such that $F_{\pi}(u)\le \alpha, F_{\pi}(v)\le \alpha$, and $F_{\pi}(x)>\alpha$
for any $x\in I^{\circ}(u,v)$.  For an integer $n>0$, denote by $\pi_n$ the profile obtained by setting $\pi_n(w)=n\cdot \pi(w)$ for all $w\in \supp(\pi)$.
Since $F_{\pi_n}(z)=n\cdot F_{\pi}(z)$ for any vertex $z$ of $G$, we conclude
that $F_{\pi_n}(u)\le n\alpha$, $F_{\pi_n}(v)\le n\alpha$, and
$F_{\pi_n}(x)>n\alpha$
for any $x\in I^{\circ}(u,v)$.

Since  $\supp(\pi)$ is finite and  for any two vertices $w\in \supp(\pi)$ and $x\in I(u,v)$ the distance $d(x,w)$ is upper bounded
by $\max\{d(u,w),
d(v,w)\}+d(u,v)$,  in $I(u,v)$ the median functions $F_{\pi}$
may take only a finite number of distinct values. Let $\Delta$
denote the maximum of this list of values. Then we can find  an integer $n^*>0$
such that the profile $\pi^*=\pi_{n^*}$ has the following property:  if
$y,z\in I(u,v)$ and $F_{\pi}(y)>F_{\pi}(z)$, then
$F_{\pi^*}(y)>F_{\pi^*}(z)+\Delta$. Let $\pi'$ be the integer profile obtained
by
rounding the profile $\pi^*$: namely, let $\pi'(w)=\lfloor \pi^*(w)\rfloor$ for any $w\in \supp(\pi^*)=\supp(\pi)$. Then $\pi'(w)d(z,w)\ge (\pi^*(w)-1)d(z,w)$ for
any $w\in \supp(\pi')=\supp(\pi^*)$ and $z\in V$. Summing up over all vertices of $\supp(\pi')$, we conclude that $F_{\pi'}(z)\ge F_{\pi^*}(z)-\Delta$. Therefore,
if for $y,z\in I(u,v)$ we have  $F_{\pi}(y)>F_{\pi}(z)$, then
$F_{\pi^*}(y)>F_{\pi^*}(z)+\Delta$, and, consequently,
$F_{\pi'}(y)>F_{\pi'}(z)$.
Applying these inequalities to the pairs $\{ u,x\}$ and  $\{ x,v\}$ with $z\in I(u,v)$, we conclude that
$F_{\pi'}(x)>F_{\pi'}(u)$ and $F_{\pi'}(x)>F_{\pi'}(v)$ for any $x\in I^{\circ}(u,v)$. Setting $\alpha'=\max\{ F_{\pi'}(u), F_{\pi'}(v)\}$,
we conclude that the level set  $L_{\preceq}(F_{\pi'},\alpha')$ is not $p$-isometric, contrary to the fact that $\pi'$ is an integer profile and that
(\ref{th-cmedp-0}) is equivalent to (\ref{th-cmedp-6}) for integer profiles.
\end{proof}

It is shown in \cite[Proposition~2 ]{BaCh_median} that $G$ is a graph with connected
medians ($p=1$) if and only if $\WC(u,v)$ holds for all pairs of vertices $u,v$ with $d(u,v)=2$ and any profile
$\pi$ included in $J(u,v)$, $\Med(\pi)=\lMed(\pi)$. We extend this result for all $p$.

\begin{proposition}\label{J(uv)}
A graph $G$ is a graph with $G^p$-connected medians if and only if for any vertices $u,v$
with $p+1\leq d(u,v)\le 2p$ and every weight function $\pi$ with $\supp(\pi)\subseteq J(u,v)$, the median
function $F_{\pi}$ satisfies the condition $\WP(u,v)$ (respectively, $\WC(u,v)$).
\end{proposition}
\begin{proof} Necessity immediately follows from Theorem \ref{th-cmed-p}.
To prove sufficiency, in view of Theorem \ref{th-cmed-p} assume that  there exists a weight function $\pi$
with finite support for which the median function $F_\pi$ is not $p$-weakly peakless (or $p$-weakly convex).
By Lemmas \ref{lemma-Gp-cmed-2} and \ref{lemma-Gp-cmed-3},
there exist vertices $u,v$ with $p+1 \leq d(u,v)\le 2p$ such that $\WP(u,v)$ (respectively, $\WC(u,v)$) does not hold. We assume that among all profiles for which $\WP(u,v)$ (respectively, $\WC(u,v)$)  does not hold, $\pi$ is a profile minimizing $|\supp(\pi)\setminus J(u,v)|.$
If $\supp(\pi)\subseteq J(u,v)$, then we are done. Now, suppose that there exists a vertex $x\in \supp(\pi)\setminus J(u,v)$.
Let $x'$ be a vertex in $I(u,x)\cap I(v,x)$ at maximal distance from $x$.
We define the following profile $\pi'$:
	\begin{equation*}
	\pi'(y) =
	\begin{cases}
	\pi(y) & \text{if } y \neq x,x'\\
	0 & \text{if } y =x\\
	\pi(x)+\pi(x') & \text{if } y =x'.
	\end{cases}
	\end{equation*}
Since $x'\in I(u,x)\cap I(v,x)$, we obtain $F_{\pi'}(u)=F_\pi(u)+\pi(x)(d(u,x')-d(u,x))=F_\pi(u)-\pi(x)d(x,x')$ and analogously $F_{\pi'}(v) = F_\pi(v) - \pi(x)d(x,x')$.
Analogously,  for any $w\in I^{\circ}(u,v)$, we have $F_{\pi}(w)-F_{\pi'}(w)=(d(w,x)-d(w,x'))\pi(x)$, and by the triangle inequality we obtain that $F_{\pi'}(w)\ge F_{\pi}(w)-d(x,x')\pi(x)$.
Consequently, we deduce that $F_{\pi'}$ also violates the condition $\WP(u,v)$. Since $|\pi'\setminus J(u,v)|<|\pi\setminus J(u,v)|$, we obtain a contradiction with the choice of $\pi$.
Now, we assert that $F_{\pi'}$ also violates the convexity condition  $\WC(u,v)$. Indeed,
\begin{align*}
	d(v,w)F_{\pi'}(u)+d(u,w)F_{\pi'}(v)&=
	d(v,w)(F_{\pi}(u) -\pi(x)d(x,x')) +d(u,w)(F_{\pi}(v)-\pi(x)d(x,x')) \\
	&= d(v,w)F_{\pi}(u) + d(u,w)F_{\pi}(v) - d(x,x')\pi(x)(d(v,w) + d(u,w))\\
	&= d(v,w)F_{\pi}(u) + d(u,w)F_{\pi}(v) - d(x,x')\pi(x)d(u,v).
	\end{align*}
Since $d(u,v)F_{\pi'}(w)=d(u,v)F_\pi(w)+d(u,v)\pi(x)(d(x',w)-d(x,w))$ and $\pi$ violates $\WC(u,v)$, to
show that $F_{\pi'}$ also violates $\WC(u,v)$, we only have to prove that $d(x',w)-d(x,w)\geq -d(x,x')$,
which is true by the triangle inequality.
\end{proof}

When $p=1$, Proposition \ref{J(uv)} can be viewed as a local-to-global characterization because for any $u,v$ with $d(u,v)=2$,
$J(u,v)$ is included in the 2-neighborhood of the interval $I(u,v)$. This is a consequence of the fact that metric triangles of graphs
with connected medians are equilateral \cite[Remark 2]{BaCh_median}. Under the condition that $G$ is a graph with equilateral metric triangles,
Proposition \ref{J(uv)} can be viewed as local-to-global characterization of graphs with $G^p$-connected medians
because of the following observation:

\begin{lemma}\label{equilateral-triangles} If $G=(V,E)$ is a graph with equilateral metric triangles, then for any $p\ge 1$ and $u,v\in V$ with
$p+1\le d(u,v)\le 2p$, $J(u,v)\subseteq B_{2p}(I(u,v))$.
\end{lemma}

\begin{proof} Let $z\in J(u,v)$ and let $x$ be the furthest from $u$ vertex of $I(u,z)\cap I(u,v)$ and let $y$ be a furthest from $v$ vertex of $I(v,x)\cap I(v,z)$.
Then one can easily check that $(xyz)$ is a metric triangle which is a quasi-median of the triplet $u,v,z$. Since $x$ and $y$ belong to a common $(u,v)$-geodesic,
$d(x,y)\le 2p$. Since $(xyz)$ is an equilateral metric triangle, $d(z,x)=d(z,y)=d(x,y)\le 2p$, whence $z\in B_{2p}(I(u,v))$.
\end{proof}

For a graph $G$ we denote by $p(G)$ the smallest integer $p$ such that $G$ is a graph with $G^p$-connected medians ($p(G)=\infty$ if such minimal $p$ does not exist).
We conclude this subsection by showing that for any integer $m\ge 1$ there exists a graph $G_m$ such that $p(G_m)\ge m$.

\begin{proposition}
  Let $C_{n}$ be the cycle of length $n=2k+m$, where $k \geq 2$ and
  $2 \leq m\leq 2k-1$. Then $p(C_n)\ge m$.
\end{proposition}

\begin{proof}
  Let $u$ and $v$ be two vertices of $C_n$ at distance $m$ and let $x$
  be the vertex of $C_n$ at distance $k$ from $u$ and $v$.  Let
  $\alpha>k$. Define the profile $\pi$ in the following way: set
  $\pi(u)=\alpha$, $\pi(v)=\alpha$, $\pi(x)=1$, and $\pi(y)=0$ for any
  $y\notin \{u,v,x\}$.  Then $F_{\pi}(u)=F_{\pi}(v) = m\alpha+k$. We
  assert that $F_{\pi}(z)>F_{\pi}(v)=F_{\pi}(u)$ for any vertex
  $z\notin \{u,v\}$, establishing this way that the median set
  $\Med(\pi) = \{u,v\}$ induces a disconnected subgraph of
  $C_n^{m-1}$. Without loss of generality, assume that
  $d(u,z) \leq d(v,z)$ and let $d(u,z) = i \geq 1$. Assume first that
  $z \in I^{\circ}(u,v)$ (since $m<2k$, $I(u,v)$ is the path of $C_n$
  not traversing $x$).  Then $u \in I(z,x)$ and
  $F_{\pi}(z)=i\alpha+(m-i)\alpha+k+i=m\alpha+k+i > m\alpha +k
  =F_{\pi}(v)$. Suppose now that $u \in I(u,x) \setminus \{u\}$ and
  that $u \in I(z,v)$ (i.e., $d(z,v) = m+i \leq 2k-i$). Then,
  $F_\pi(z) = i\alpha + (m+i) \alpha + k-i = m\alpha + k + i(2\alpha
  -1) > m\alpha +k = F_{\pi}(v)$ since $i \geq 1$ and $\alpha >
  k$. Finally, suppose that $u \in I(u,x) \setminus \{u\}$ and that
  $x \in I(z,v)$ (i.e., $d(z,v) = 2k-i \leq m+i$). Then,
  $F_\pi(z) = i\alpha + (2k-i) \alpha + k-i \geq 2k\alpha$. Since
  $2k \geq m+1$ and $\alpha > k$, we have $(2k-m)\alpha > k$ and thus
  $F_\pi(z) \geq 2k\alpha > m\alpha + k = F_{\pi}(v)$.
\end{proof}

\subsection{Recognition of graphs with $G^p$ connected medians}
In view of Theorems \ref{wc-ltg} and \ref{th-cmed-p}, a graph has
$G^p$-connected medians if and only if for each pair of vertices $u,v$ such that $p+1\leq
d(u,v) \leq 2p$, the following system of linear inequalities is unsolvable in
$\pi$:
\[
\begin{array}{l}
D^{uv}\pi<0\mbox{ and } \pi\geq 0 \mbox{ with matrix}\\
D^{uv} = (d(v,w)d(u,x)+d(u,w)d(v,x)-d(u,v)d(w,x))_{w\in I^\circ(u,v), x\in V}.
\end{array}
\]
Since linear programming problems can be solved in polynomial time, we obtain
the following corollary, extending \cite[Corollary 1]{BaCh_median}:

\begin{corollary}
	The problem of deciding whether a graph $G$ has $G^p$-connected medians is
	solvable in polynomial time. In particular, computing $p(G)$ can be done in polynomial time.
\end{corollary}

In case of graphs with equilateral metric triangles, by Lemma \ref{equilateral-triangles} and
Proposition \ref{J(uv)}, the matrix $D^{uv}$ can be defined locally: instead of considering all vertices $x\in V$
it suffices to consider only the vertices $x\in J(u,v)\subseteq B_{2p}(I(u,v))$.

\subsection{Cartesian products, gated amalgams, and retracts}
The goal of this subsection is to prove the following result:

\begin{proposition}\label{gated-amalgams-products} For an integer $p\ge 1$, the class of
graphs with $G^p$-connected medians is closed under taking Cartesian products, gated amalgams,
and retracts.
\end{proposition}

\begin{proof} In this proof, given a profile $\pi$ on the vertices of a graph $G$, we will denote the median function
$F_{\pi}$ by $F_{G,\pi}$ and the median set by $\Med_{G}(\pi)$. For a subgraph $H$ of $G$, we set $\pi(H)=\sum_{v\in V(H)} \pi(v)$.

To prove the first assertion it suffices to show that the Cartesian product
$G=G_1\square G_2$ of two graphs $G_1$ and $G_2$ with $G^p$-connected medians
is also a graph with $G^p$-connected medians. Let $\pi$ be a profile on $G$.
For a vertex $v_1$ of $G_1$, let $\pi_1(v_1)$ be the sum of all
$\pi (v)$ such that $v=(v_1,v_2)$ for a vertex $v_2$ of $G_2$. Analogously define
the profile $\pi_2$ on the vertex set of $G_2$. Since $d_{G}(u,v)=d_{G_1}(u_1,v_1)+d_{G_2}(u_2,v_2)$ if $u=(u_1,u_2)$ and $v=(v_1,v_2)$,
we conclude that
$F_{G,\pi}(v)=F_{G_1,\pi_1}(v_1)+F_{G_2,\pi_2}(v_2)$
for any vertex $v=(v_1,v_2)$ of $G$. Consequently,
$\Med_{G}(\pi)=\Med_{G_1}(\pi_1)\square \Med_{G_2}(\pi_2)$.

We assert that $\Med_{G}(\pi)$ is connected in
$G^p=(G_1\square G_2)^p$.  Let $u=(u_1,u_2),v=(v_1,v_2)$ be two
vertices of $\Med_{G}(\pi)$ at distance at least $p+1$ in $G$. Suppose
first that $u_1 \neq v_1$ and $u_2 \neq v_2$. Then
$w = (u_1,v_2) \in I^{\circ}(u,v)$ and since
$u_1 \in \Med_{G_1}(\pi_1)$ and $v_2 \in \Med_{G_2}(\pi_2)$,
$w = (u_1,v_2) \in \Med_{G_1}(\pi_1)\square \Med_{G_2}(\pi_2)=
\Med_{G}(\pi)$. Suppose now that $u_2 = v_2$ (the case $u_1 = v_1$ is
similar). Since $d_G(u,v) \geq p+1$, we have
$d_{G_1}(u_1, v_1) \geq p+1$ and since $\Med_{G_1}(\pi_1)$ is
$G^p$-connected, there exists
$w_1 \in I^{\circ}(u,v)\cap \Med_{G_1}(\pi_1)$. Then,
$w = (w_1,u_2) \in I^{\circ}(u,v)$ and
$w = (w_1,u_2) \in \Med_{G_1}(\pi_1)\square \Med_{G_2}(\pi_2)=
\Med_{G}(\pi)$. This ends the proof for Cartesian products.

Let $G$ be a gated amalgam of two graphs $G_1$ and $G_2$ with
$G^p$-connected medians and let $\pi$ be a profile on $G$.
For a vertex $v$ of $G_1$ (respectively, of $G_2$) denote by $v'$ its
gate in $G_2$ (respectively, in $G_1$). Recall that $v'$ is the (necessarily unique) vertex
of $G_2$ such that $v'\in I(v,x)$ for any vertex of $G_2$.

\begin{claim}\label{gated1} For any vertex $v$ of $G_2\setminus G_1$,
$F_{G,\pi}(v)-F_{G,\pi}(v')\ge \pi(G_1)-\pi(G_2\setminus G_1)$. Consequently,
if $\pi(G_1)>\pi(G_2)$, then $\Med_{G}(\pi)\subseteq V(G_1)$ and if
$\pi(G_1)=\pi(G_2)$ and $v\in  \Med_{G}(\pi)$, then $v'\in \Med_{G}(\pi)$.
\end{claim}

\begin{proof}
The profile $\pi$ is the disjoint union
of the profiles $\pi'$ and $\pi''$, where $\pi'$ coincides with $\pi$ on $G_1$ and is 0 elsewhere
and $\pi''$ coincides with $\pi$ on $G_2\setminus G_1$ and is 0 elsewhere.
Then $F_{G,\pi}(u)=F_{G,\pi'}(u)+F_{G,\pi''}(u)$ for any vertex $u$ of $G$. Now, pick any $v$ in $G_2\setminus G_1$.
Since $v'$ is the gate of $v$ in $G_1$, for any $x_1\in G_1$, $d_G(v,x_1)=d_G(v,v')+d_G(v',x_1)$, whence
 \begin{align*}
F_{G,\pi'}(v)-F_{G,\pi'}(v')=\pi'(G_1)d(v,v')=\pi(G_1)d(v,v').
\end{align*}
By triangle inequality, for any $x_2\in G_2\setminus G_1$, $d(v,x_2)-d(v',x_2) \geq -d(v,v')$, whence
 \begin{align*}
 F_{G,\pi''}(v)-F_{G,\pi''}(v')\geq -\pi''(G_2\setminus G_1)d(v,v')=-\pi(G_2\setminus G_1)d(v,v').
 \end{align*}
 Since $F_{G,\pi}(v)=F_{G,\pi'}(v)+F_{G,\pi''}(v)$ and
 $F_{G,\pi}(v')=F_{G,\pi'}(v')+F_{G,\pi''}(v')$, we obtain that
 $F_{G,\pi}(v)-F_{G,\pi}(v')\ge (\pi(G_1)-\pi(G_2\setminus
 G_1))d(v,v')$.  If $\pi(G_1)>\pi(G_2)$, then
 $F_{G,\pi}(v)>F_{G,\pi(v')}$, whence $v \notin \Med_G(\pi)$.
 Consequently, $\Med_G(\pi) \subseteq V(G_1)$.  If
 $\pi(G_1)=\pi(G_2)$, then $F_{G,\pi(v')} \leq
 F_{G,\pi}(v)$. Therefore, if $v \in \Med_G(\pi)$, then
 $v' \in \Med_G(\pi)$.
\end{proof}

For each vertex $u$ of $G_1$, let $P(u)=\{x \in V(G): u=x' \text{ where } x' \text{ is the gate of } x \text{ in } G_1\}$ and call $P(u)$
the \emph{fiber} of $u$.  Since $G_1$ is gated and the gates are unique,  the fibers $P(u), u\in V(G_1)$ define a partition of the vertex-set of $G$.
For each vertex $u$ of $G_1$, let $\pi_1(u)=\sum_{v\in P(u)} \pi(v)$. Since $P(u)=\{u\}$ if $u\in V(G_1)\setminus V(G_2)$, in this case $\pi_1(u)=\pi(u)$.

\begin{claim}\label{gated2} If $\pi(G_1)\ge \pi(G_2)$, then
$\Med_{G}(\pi)\cap V(G_1)=\Med_{G_1}(\pi_1)$.
\end{claim}

\begin{proof}
Let $K:=\sum_{x \in V(G_2)} d(x,x') \pi(x)$, where $x'$ is the gate of $x$ in $G_1$.
We assert that for any vertex  $v$ of $G_1$, the equality $F_{G,\pi}(v)=F_{G_1,\pi_1}(v)+K$ holds.
Indeed, pick any vertex $x$ from the profile $\pi$. If $x\in V(G_1)\setminus V(G_2)$, then $\pi_1(x)=\pi(x)$ and
thus $\pi(x)d(v,x)=\pi_1(x)d(v,x)$. Now suppose that $x\in V(G_2)$. Then $x$ belongs to the fiber $P(x')$.
Therefore the vertex $x$ contributes with $\pi(x)$ to the weight $\pi_1(x')$ of $x'$  and
with $\pi(x)d(x',v)$ to $F_{G_1,\pi_1}(v)$. On the other hand, $x$ contributes with $\pi(x)d(x,v)$ to $F_{G,\pi}(v)$.
Since $d(x,v)=d(x,x')+d(x',v)$, the difference of the two contributions equals to $\pi(x)d(x,x')$. Summing over all $x\in V(G)$, we
obtain that the total difference $F_{G,\pi}(v)-F_{G,\pi_1}(v)$ is equal to $K$.  Consequently, any vertex of $G_1$ minimizing $F_{G,\pi}$ minimizes
also $F_{G_1,\pi_1}$ By Claim~\ref{gated1},
$\Med_G(\pi) \cap V(G_1)$ is not empty and thus
$\Med_{G}(\pi)\cap V(G_1)=\Med_{G_1}(\pi_1)$.
\end{proof}

Consider now two vertices $u,v$ of $\Med_{G}(\pi)$ at distance at
least $p+1$. We show that there exists a vertex $w$ belonging to
$I^{\circ}(u,v)\cap \Med_{G}(\pi)$.  First suppose that $u$ and $v$
belong to the same graph $G_1$ or $G_2$. If $u$ and $v$ belong both to $G_1$ and $G_2$, we can suppose that
$\pi(G_1)\ge \pi(G_2)$. Otherwise, we can suppose that $u,v$ belong to
$G_1$ and $u$ does not belong to $G_2$. Since $u,v\in \Med_{G}(\pi)$ and $u$ does not belong to $G_2$, from Claim \ref{gated1} we deduce that
$\pi(G_1)\ge \pi(G_2)$. In both cases,   from Claim \ref{gated2} we deduce that
$\Med_{G}(\pi)\cap V(G_1)=\Med_{G_1}(\pi_1)$. Since
$\Med_{G_1}(\pi_1)$ is $G^p$-connected, there exists a vertex
$w\in I^{\circ}(u,v)\cap \Med_{G_1}(\pi_1)$ which belongs to
$\Med_{G}(\pi)$ by Claim~\ref{gated2}. Now, suppose that $u$ belongs
to $G_1\setminus G_2$ and $v$ belongs to $G_2\setminus G_1$. By Claim
\ref{gated1}, $\pi(G_1)=\pi(G_2)$ and the gate $u'$ of $u$ in
$G_2$ and the gate of $v'$ of $v$ in $G_1$ both belong to
$\Med_{G}(\pi)$. Since $u'$ and $v'$ belong to $G_1\cap G_2$ and to $I(u,v)$,
necessarily $u',v'\in I^{\circ}(u,v)$, and we are done. This
establishes the result for gated amalgams.

Finally, suppose that a graph $G$ is a retract of a graph $G'$ with $G^p$-connected medians.
Let $V(G)\subseteq V(G')$ and $r:V(G')\rightarrow V(G)$ be the retraction map. Then $r$ is idempotent (i.e., $r(u)=u$ for any $u\in V(G)$) and nonexpansive (i.e., $d_{G'}(r(u),r(v))\le d_G(u,v)$ for any two vertices $u,v$ of $G'$).
Let $\pi$ be any profile with finite support on $V(G)$ and pick two vertices $u,v\in \Med_{G}(\pi)$ with $d_{G}(u,v)\ge p+1$. Since $r$ is nonexpansive, for any vertex $x$ of $G'$
we have $F_{G,\pi}(r(x))\le F_{G',\pi}(x)$. Since $r$ is idempotent and the support of $\pi$ is contained in $G$, for any vertex $x$ of $G$ we have  $F_{G,\pi}(r(x))=F_{G',\pi}(x)$.
Consequently, $\Med_{G}(\pi)\subseteq \Med_{G'}(\pi)$. Since $G'$ is a graph with $G^p$-connected medians and
$G$ is an isometric subgraph of $G$, there exists a vertex $w\in I^{\circ}_{G'}(u,v)\cap \Med_{G'}(\pi)$. Since $F_{G,\pi}(r(w))\le F_{G',\pi}(w)$, we conclude that $r(w)\in \Med_{G}(\pi)$. Since $w\in I^{\circ}_{G'}(u,v)$ and
$r(u)=u, r(v)=v$, we obtain that $r(w)\in I^{\circ}_G(u,v)$, consequently $r(w)\in \Med_{G}(\pi)\cap I^{\circ}(u,v)$ in $G$.
\end{proof}

\begin{corollary} \label{cor:amalgams-products} If a finite graph $G$ can be obtained from the graphs $G_1,\ldots,G_m$ via Cartesian products and gated amalgams,
then $p(G)=\max\{ p(G_i): i=1,\ldots,n\}$.
\end{corollary}

\begin{proof} The inequality $p(G)\le \max\{ p(G_i): i=1,\ldots,n\}$ follows from Proposition \ref{gated-amalgams-products}.
To prove the converse inequality  $p(G)\ge \max\{ p(G_i): i=1,\ldots,n\}$, notice that
$G$ contains a copy of any  $G_i$ as a gated subgraph. Indeed, any $G_i$ occurs in $G$ either as a gated subgraph or as a
factor in a Cartesian product $H=G_{j_1}\square G_{j_2}\square \ldots\square G_{j_k}$, say $G_i=G_{j_1}$. Then $H$ is a
gated subgraph of $G$.  Moreover any subgraph of $H$ of the form $G_{i}\square \{ v_{j_2}\} \square \ldots\square \{ v_{j_k}\}$, where
$v_{j_\ell}$ is a vertex of the factor $G_{j_\ell}$, is a gated subgraph of $H$ isomorphic to $G_i$.  Consequently, each $G_i$ occurs
as a gated subgraph of $G$, which we also denote by $G_i$. Let $p=p(G_i)$ and let $\pi$ be a weight function on $G_i$ such that the median set
$\Med_{G_i}(\pi)$ is $G^p$-connected but not $G^{p-1}$-connected. Let $\pi'$ be the weight function on $G$ defined in the following way:
$\pi'(v)=\pi(v)$ if $v$ is a vertex of $G_i$ and $\pi'(v)=0$ otherwise. Since $G_i$ is a gated subgraph of $G$, from the definition of
$\pi'$ we conclude that $\Med_G(\pi')=\Med_{G_i}(\pi)$. Consequently, $\Med_G(\pi')$ is $G^p$-connected but not $G^{p-1}$-connected,
whence $p(G)\ge p=p(G_i)$.
\end{proof}

\section{Classes of graphs with $G^2$ and $G^p$-connected medians}\label{classes}
In this section we show that several important classes of graphs have $G^2$-connected medians or have $G^p$-connected medians with bounded $p$.
Fig. \ref{fig-classes} provides an inclusion diagram containing these classes of graphs (except benzenoids, which are not included
in any class of graphs from the figure). In the resulting ``inclusion poset'' of classes, the classes
of graphs that have connected or $G^2$-connected medians are represented in blue. 

\bigskip
\begin{figure}[h]
	\begin{center}
		\includegraphics[scale=0.65]{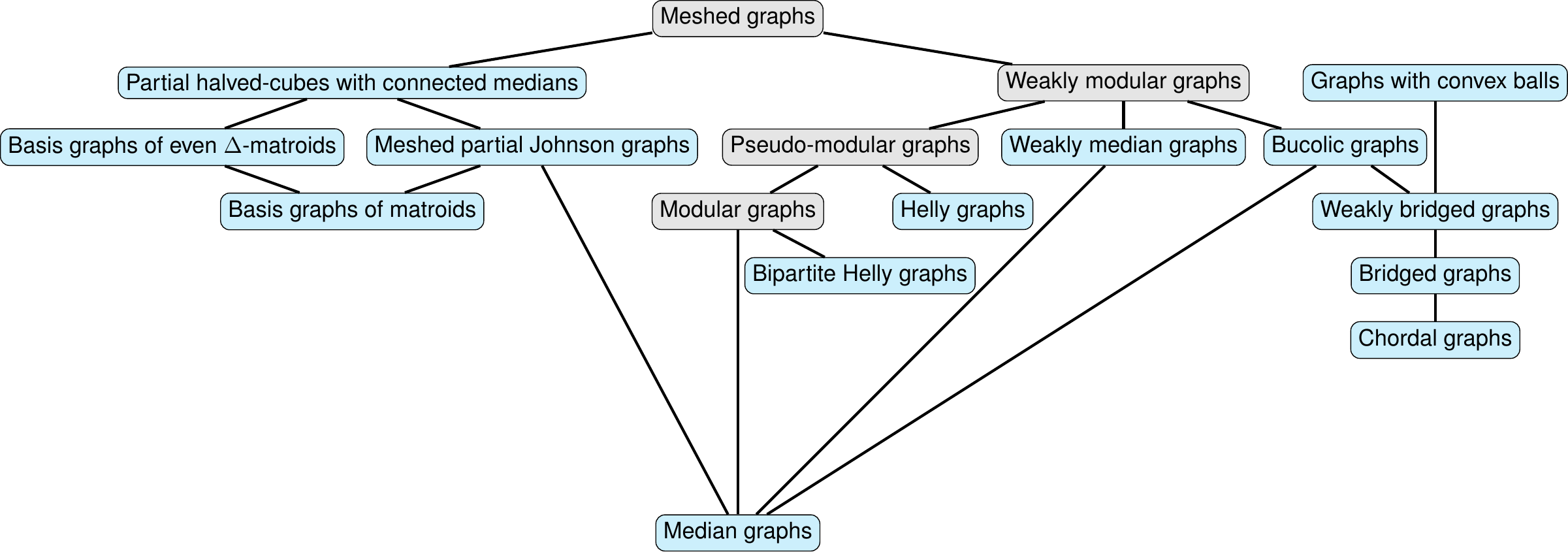}
	\end{center}
	\caption{An inclusion diagram of the graph classes occurring in this
	paper. The classes of graphs with
	connected or $G^2$-connected medians are given in blue.}\label{fig-classes}
\end{figure}

\subsection{Auxiliary results}
In this subsection, we present sufficient conditions, under which  a graph $G$ has $G^2$-connected medians. We specify  these conditions in case of weakly modular and modular graphs.
In subsequent subsections, we use them to prove that bridged and weakly bridged graphs, graphs with convex balls, and bipartite absolute retracts are graphs with $G^2$-connected medians.

For a 3-interval $I(u,v)$ of a graph $G$, consider the following condition:
\begin{itemize}
\item[(a)] there exist $x,y\in I^{\circ}(u,v)$, $x\sim u$ and $y\sim v$, such that $x$ is adjacent to all $z\in I(u,v)\cap N(v), z\ne y$
and $y$ is adjacent to all  $z'\in I(u,v)\cap N(u), z'\ne x$.
\end{itemize}
In condition (a), the vertices $x$ and $y$ may be adjacent and may be nonadjacent. 
For a 4-interval $I(u,v)$ of a graph $G$, consider the following conditions:
\begin{itemize}
\item[(b)] there exist $x,y\in I^{\circ}(u,v)$ with $d(u,x)=d(u,y)=2$ such that $x$ is adjacent to all  $z\in I(u,v)\cap N(u)$ and $y$
is adjacent to all  $z\in I(u,v)\cap N(v)$;
\item[(c)] there exist $x,y\in I^{\circ}(u,v)$, $x\sim u$ and $y\sim v$, such that  $x$ and $y$ are adjacent to all $z\in I^{\circ}(u,v)$
with $d(u,z)=d(v,z)=2$;
\end{itemize}

In condition (b), the vertices $x$ and $y$ may coincide, be adjacent or nonadjacent, however in bipartite graphs they cannot be adjacent.

\begin{lemma}\label{abcd} Let $I(u,v)$ be a 3-interval or a 4-interval of a graph $G$. If $I(u,v)$ satisfies one of the conditions (a),(b),(c), then for every weight function $\pi$ with $\supp(\pi)\subseteq I(u,v)$, the median function $F_{\pi}$ satisfies the lozenge condition $\Loz(u,v)$.
\end{lemma}

\begin{proof} To establish  $\Loz(u,v)$ we first prove that for any vertex $z\in I(u,v)$ we have
$d(x,z)+d(y,z)\le d(u,z)+d(v,z)$, where $x$ and $y$ are the vertices specified in conditions (a), (b), or (c). First, let $d(u,v)=3$. Then $d(u,z)+d(v,z)=3$. If $z$ coincides with $u$ or $v$, then
obviously $d(x,z)+d(y,z)=3$. Now, suppose
without loss of generality that $z\in I(u,v)\cap N(u)$. If $z=x$, then again $d(x,z)+d(y,z)\le 3$ because $d(y,u)=2$ and $u\sim x$. If $z\ne x$, then
$z\sim y$ and $d(x,z)\le 2$ because $u$ is adjacent to $x$ and $z$. Therefore, $d(x,z)+d(y,z)\le 3$ also in this case.

Now, let $d(u,v)=4$. Then $d(u,z)+d(v,z)=4$. Again, $d(x,z)+d(y,z)=4=d(u,z)+d(v,z)$ if $z\in \{ u,v\}$. First suppose that $z$ has distance 2 to $u$ and $v$.
If $I(u,v)$ satisfies the condition (b),  then $d(x,z)\le 2$ and $d(y,z)\le 2$ because
$x$ is adjacent to any common neighbor of $z$ and $u$ and $y$ is adjacent to any common neighbor of $z$ and $v$. Thus $d(x,z)+d(y,z)\le 4=d(u,z)+d(v,z)$ in this case.
If $I(u,v)$ satisfies condition (c), then $x\sim z$ and $y\sim z$ and thus $d(x,z)+d(y,z)\le 2$. 
Now, suppose that $z$ is adjacent to $u$ or $v$. If $I(u,v)$ satisfies condition (b), then $d(x,z)+d(y,z)\le 4$ because $z$ is adjacent to one of the vertices $x,y$ and has
distance at most 3 to the other one. If $I(u,v)$ satisfies condition (c), then $d(x,z)\le 2$ and $d(y,z)\le 2$ because $x$ and $y$ are adjacent to any neighbor of $z$ having distance 2
from $u$ and $v$. 
Now, let $\pi$ be any weight function $\pi$ with $\supp(\pi)\subseteq I(u,v)$. Applying the conditions (a),(b), or (c) to each we vertex $z$ of $\pi$ and summing up the inequalities $d(x,z)+d(y,z)\le d(u,z)+d(v,z)$
multiplied by $\pi(z)$, we conclude
that $F_{\pi}(x)+F_{\pi}(y)\le F_{\pi}(u)+F_{\pi}(u)$, establishing $\Loz(u,v)$.
\end{proof}

The following condition occurs in the characterization of graphs with convex balls in \cite{CCG,SoCh}:
\begin{itemize}
\item
\emph{Interval Neighborhood Condition (INC)}:  for any two nonadjacent vertices $u,v$, all neighbors of $u$ in $I(u,v)$ induce a complete subgraph.
\end{itemize}

For a pair of nonadjacent vertices $u,v$ of a graph $G$, we denote by $S(u,v)$ the set of all vertices $z$ of $G$ such that the triplet $z,u,v$ has a quasi-median $z'u'v'$, which
is a strongly equilateral metric triangle.

\begin{lemma}\label{abc_bis} Let $G$ be a graph satisfying the condition INC. Let $I(u,v)$ be a 3-interval of $G$ satisfying condition (a) or a 4-interval
satisfying one of the conditions (b) or (c) with vertices $x,y\in I^{\circ}(u,v)$. Then for any vertex $z\in S(u,v)$ we have  $d(z,x)+d(z,y)\le d(z,u)+d(z,v)$.
Furthermore, for any weight function $\pi$ with $\supp(\pi)\subseteq S(u,v)$, the median
function $F_{\pi}$ satisfies the lozenge condition $\Loz(u,v)$.
\end{lemma}

\begin{proof} To establish $\Loz(u,v)$ it suffices to show that for any vertex $z\in \supp(\pi)\subseteq S(u,v)$ we have
$d(x,z)+d(y,z)\le d(u,z)+d(v,z)$. Let $z'u'v'$ be a quasi-median of the triplet $z,u,v$ such that $z'u'v'$ is a strongly equilateral metric triangle.
Since $d(z,u)=d(z,z')+d(z',u), d(z,v)=d(z,z')+d(z',v)$ and $d(z,x)\le d(z,z')+d(z',x),d(z,y)\le d(z,z')+d(z',y)$, it suffices to show that
$d(z',x)+d(z',y)\le d(z',u)+d(z',v)$. We distinguish several cases depending of the size of the metric triangle $u'v'z'$, the
location of the vertices  $u'$ and $v'$, and which of the conditions (a),(b), or (c) is satisfied by $I(u,v)$.

\begin{case} \label{case1} $d(u,v)=3$.
\end{case}

Since $G$ satisfies INC and $I(u,v)$ satisfies condition (a), each of $x$ and $y$ is adjacent to all vertices
of $I(u,v)\setminus \{ u,v\}$, except that $x$ and $y$ are not necessarily adjacent. If $u'v'z'$ has size 0, i.e., $u'=v'=z'$, then
$z'\in I(u,v)$ and the inequality $d(z',x)+d(z',y)\le d(z',u)+d(z',v)$ follows from Lemma \ref{abcd}. If $u'v'z'$ has size 3, then $u'=u$ and $v'=v$, thus
$d(u,z')+d(v,z')=3+3=6.$ Since $z'uv$ is a strongly equilateral metric triangle of size 3 and $x,y\in I(u,v)$, thus $d(x,z')=d(y,z')=3$, yielding
$d(x,z')+d(y,z')=6$.

\begin{subcase} The metric triangle $z'u'v'$ has size 2.
\end{subcase}

Then $u'=u$ or $v'=v$, say $u'=u$. In this case $v'\sim v$ and $d(z',u)+d(z',v)=2+3=5$. By condition (a), either $v'=y$ or $x\sim v'$. In the second case,
since $x\sim u,v'$, we have $x\in I(u,v')$.  Since $z'uv'$ is strongly equilateral, we conclude that $d(z',x)=2$. Since by INC $y$ coincides or is adjacent
to $v'$, $d(z',y)\le 3$, yielding $d(z',x)+d(z',y)\le 5$. If $y=v'$, then $d(z',y)=2$. Since $x$ is adjacent to $u$, $d(z',x)\le 3$ and again we deduce that
$d(z',x)+d(z',y)\le 3+2=5$.

\begin{subcase} The metric triangle $z'u'v'$ has size 1.
\end{subcase}

Then $d(z',u)+d(z',v)=4$. First, let $u'=u$ or $v'=v$, say $u'=u$.  By INC and
(a), if $x\ne v'$, then
$x$ and $y$ are adjacent to $v'$, whence $d(z',x)+d(z',y)\le 2+2=4$. If $x=v'$, the $d(z',x)=1$. Since by INC $y$ coincides or is adjacent to any common
neighbor of $v'$ and $v$, $d(z',y)\le 3$. Consequently, $d(z',x)+d(z',y)\le 1+3=4$. If $u'\ne u$ and $v'\ne v$, then $u\sim u'$ and $v\sim v'$.
By INC, $x$ coincides  or is adjacent to $u'$ and $y$ coincides or is adjacent to $v'$, whence $d(z',x)+d(z',y)\le 2+2=4$. This concludes the proof of Case \ref{case1}.

\begin{case} \label{case2} $d(u,v)=4$.
\end{case}

If $u'v'z'$ has size 0, i.e., $u'=v'=z'$, then
$z'\in I(u,v)$ and the inequality $d(z',x)+d(z',y)\le d(z',u)+d(z',v)$ follows from Lemma \ref{abcd}. If $u'v'z'$ has size 4, then $u'=u$ and $v'=v$, thus
$d(u,z')+d(v,z')=4+4=8.$ Since $z'uv$ is a strongly equilateral metric triangle of size 4 and $x,y\in I(u,v)$, we conclude that $d(x,z')=d(y,z')=4$, yielding
$d(x,z')+d(y,z')=8$.

\begin{subcase} The metric triangle $u'v'z'$ has size 3.
\end{subcase}

Then $u'=u$ or $v'=v$, say $u'=u$. In this case $v'\sim v$ and $d(z',u)+d(z',v)=3+4=7$. If $I(u,v)$ satisfies condition (b), then
$x$ is adjacent to any $w\in I^{\circ}(u,v')$ adjacent to $u$ and $y$ is adjacent to $v'$.
Consequently, $y\in I(u,v')$. Since $z'uv'$ is strongly equilateral, we deduce
that $d(z',y)=3$ and $d(z',x)\le d(z',w)+1=4$.
Consequently,  $d(z',x)+d(z',y)\le 3+4=7$. If $I(u,v)$ satisfies condition (c), then $x$ and $y$ are adjacent to any neighbor $w$ of $v'$ in $I(u,v')$. This implies that
$d(z',y)\le d(z',w)+1=4$ and that $x\in I(u,v')$. Since $z'uv'$ is strongly equilateral, this implies
that $d(x,z')=3$. Consequently, $d(z',x)+d(z',y)\le 3+4=7$.  
\begin{subcase} The metric triangle  $u'v'z'$ has size 2.
\end{subcase}

First suppose that $u'=u$ or $v'=v$, say $u'=u$. Then $d(v',v)=2$ and $d(z',u)+d(z',v)=2+4=6$. If $I(u,v)$ satisfies condition (b), then $x$ is adjacent
to any $w\in I^{\circ}(u,v')$. Since $z'uv'$ is strongly equilateral,
$d(z',w)=2$. Consequently, $d(z',x)\le 3$. On the other hand, $y$ is adjacent
to any common neighbor $t$ of $v'$ and $v$.
Thus $v',y\in I(t,u)$ and by INC we conclude that $y\sim v'$. Consequently,
$d(z',y)\le 3$ and we obtain that $d(z',x)+d(z',y)\le 3+3=6$.  If $I(u,v)$
satisfies condition (c), then
$x$ and $y$ are adjacent to $v'$ and again $d(z',x)+d(z',y)\le 3+3=6$. 
Now suppose that $u'\ne u$ and $v'\ne v$. This implies that $u'\sim u$ and $v'\sim v$ and that $d(z',u)+d(z',v)=3+3=6$. If $I(u,v)$ satisfies condition (b), then $x\sim u'$ and $y\sim v'$, yielding
$d(z',x)+d(z',y)\le 3+3=6$. If $I(u,v)$ satisfies condition (c), then $x$ and $y$ are adjacent to any
$w\in I^{\circ}(u',v')$. Again, since  $z'u'v'$ is strongly equilateral, $d(z',w)=2$. Consequently, $d(z',x)+d(z',y)\le 3+3=6$. 
\begin{subcase} The metric triangle  $u'v'z'$ has size 1.
\end{subcase}

First suppose that $u'=u$ or $v'=v$, say $u'=u$. Then $d(v',v)=3$ and $d(z',u)+d(z',v)=1+4=5$. If $I(u,v)$ satisfies condition (b), then $x$ is adjacent to $v'$ and $y$ is adjacent to any neighbor $t$ of $v$ in $I(v,v')$.
Consequently, if $w\in I^{\circ}(t,v')$, then $y,w\in I(t,u)$. By INC, $y\sim w$ and thus $d(z',y)\le d(z',w)+1=3$. Consequently, $d(z',x)+d(z',y)\le 2+3=5$. If $I(u,v)$ satisfies condition (c),
then $y$ is adjacent to any neighbor $t$ of $v'$ in $I(v',v)$, whence $d(z',y)\le d(z',t)+1=3$. Since $x,v'\in I(u,v)\cap N(u)$, by INC $x\sim v'$ or $x=v'$, thus $d(z',x)\le 2$. Consequently, $d(z',x)+d(z',v)\le 2+3=5$.

Now suppose that $u'\sim u$ and $d(v',v)=2$ (the case $v'\sim v$ and $d(u',u)=2$ is analogous). Then $d(u,z')=2,d(v,z')=3$, thus $d(z',u)+d(z',v)=5$. If $I(u,v)$ satisfies condition (b), then $x$ is adjacent to $u'$ and $y$ is adjacent to any common neighbor $t$ of $v'$ and $v$.
Consequently, $d(z',x)\le 2$ and $d(z',y)\le 3$, yielding $d(z',x)+d(z',y)\le 2+3=5$. If $I(u,v)$ satisfies condition (c), then $x$ and $y$ are adjacent to $v'$, thus $d(z',x)\le 2$ and $d(z',y)\le 2$, yielding $d(z',x)+d(z',y)\le 4$.
This conclude the analysis of Case \ref{case2}  and the proof of the lemma.
\end{proof}

\subsection{Meshed, weakly modular, and modular graphs}
Meshed graphs have equilateral metric triangles. In this case, for any two vertices $u,v$ such that $3\le d(u,v)\le 4$, the set $J(u,v)$ is a subset of the 3- or 4-neighborhood of the interval
$I(u,v)$. Therefore, in case of meshed and weakly modular graphs Proposition \ref{J(uv)} can be restated in the following way:

\begin{proposition}\label{J(uv)meshed}
A meshed  graph $G$ is a graph with $G^2$-connected medians if and only if for any vertices $u,v$
with $d(u,v)=k$ with $k\in \{ 3,4\}$ and every weight function $\pi$ with $\supp(\pi)\subseteq B_k(I(u,v))$, the median
function $F_{\pi}$ satisfies the condition $\WP(u,v)$.
\end{proposition}

Since weakly modular graphs are exactly the graphs in which the metric triangles are strongly equilateral, from Lemma \ref{abc_bis}  and  Theorems \ref{wp-ltg} and \ref{th-cmed-p},
we obtain the following sufficient condition for $G^2$-connectedness of medians:

\begin{proposition}\label{weaklymodular-abc} If $G$ is a weakly modular graph satisfying INC and in which each 3-interval satisfies the condition (a) and
each 4-interval satisfies  one of the conditions (b) or (c), then $G$ is a graph with  $G^2$-connected medians.
\end{proposition}

In case of modular graphs, for any pair of vertices $u,v$ the set $J(u,v)$ coincides with $I(u,v)$. Therefore, in this case
Proposition \ref{J(uv)} can be restated in the following way:

\begin{proposition}\label{J(uv)modular}
A modular graph $G$ is a graph with $G^2$-connected medians if and only if for any vertices $u,v$
with $3\leq d(u,v)\le 4$ and every weight function $\pi$ with $\supp(\pi)\subseteq I(u,v)$, the median
function $F_{\pi}$ satisfies the condition $\WP(u,v)$.
\end{proposition}

From Proposition \ref{J(uv)modular} and Lemma \ref{abcd} we obtain the following sufficient condition for $G^2$-connectedness of medians of modular graphs:

\begin{proposition}\label{modular-abcd} If $G$ is a modular graph in which each 3-interval satisfies the condition (a) and  each 4-interval satisfies  one of the
conditions (b) or (c), then $G$ is a graph with  $G^2$-connected medians.
\end{proposition}

Bandelt, van de Vel, and Verheul \cite[Fig. 10]{BVV} showed that  the 4-intervals of modular graphs are not necessarily modular graphs. Therefore,
Proposition \ref{J(uv)modular} cannot be restated by requiring that 3- and 4-intervals are graphs with $G^2$-connected medians, but we do not have
an example that this is not true.

Now we show that there exist modular graphs in which medians are not $G^2$-connected.  Let $\PG(2,q)$ be the projective plane of order $q$ (with $q \geq 2$). This is a finite point-line geometry in which any point belongs to $q+1$ lines, any line contains $q+1$ points,
any two points belong to a common line, and the intersection of any two lines is a point. It is well-known that $\PG(2,q)$ contains exactly $q^2+q+1$ points and
$q^2+q+1$ lines.  Let $G_q$ be the graph whose vertices are the points and the lines of $\PG(2,q)$ and two additional vertices $u$ and $v$. There is an edge in $G_q$
from $u$ to every point, an edge from $v$ to every line, and an edge between a point and a line if and only if the point belongs to the line. The graph $G_q$ is modular
as the cover graph of a modular lattice (in fact, the face-lattice of any projective space is a complemented modular lattice \cite{Bir}).

\begin{proposition}\label{modular-->nonG2} The medians of the modular graph $G_q$ are not $G^2$-connected, i.e., $p(G_q)\ge 3$.
\end{proposition}

\begin{proof} Consider the profile $\pi$ on $V(G_q)$ such that
  $\pi(x)=1$ for any $x\in V(G_q)$. The vertex $u$ is adjacent to
  $q^2+q+1$ points of $\PG(q)$, has distance 2 to all $q^2+q+1$ lines
  of $\PG(q)$, and distance 3 to $v$.  Therefore
  $F_{\pi}(u)=(q^2+q+1)+2(q^2+q+1)+3=3q^2+3q+6$. Analogously,
  $F_{\pi}(v)=3q^2+3q+6$.  On the other hand, any vertex $z\ne u,v$,
  say a point, is adjacent to $u$, to $q+1$ lines containing the
  point, is at distance 2 to all the other $q^2+q$ points (because any
  two points belong to a line) and to $v$, and is at distance 3 to the
  $q^2$ lines not containing $z$. Consequently,
  $F_{\pi}(z) = 1+(q+1) + 2 (q^2 + q + 1) + 3 q^2 = 5q^2 + 3q +4 >
  3q^2+3q+6$ since $q \geq 2$.  This shows that
  $\Med_{\pi}(G_q)=\{ u,v\}$. Since $d(u,v)=3$, $p(G_q)\ge 3$.
\end{proof}

As we mentioned above, the covering graph $G_{d,q}$ of the face-lattice of the finite projective space $\PG(d,q)$ of dimension $d$ and order $q$ is modular. We \emph{conjecture} that
$p(G_{d,q})>d$.

\subsection{Bridged and weakly bridged graphs}
Wittenberg \cite{Wi} presented examples of chordal graphs in which the median sets are not connected. It is shown in \cite{BaCh_median} that chordal graphs with connected medians cannot be
characterized using sets $S$ (in Theorem \ref{th-cmed}) of bounded size. In this subsection, we prove that chordal graphs and, much more generally, weakly bridged graphs are graphs with $G^2$-connected medians.

A graph $G$ is called {\it bridged} \cite{FaJa,SoCh} if it does not contain any isometric cycle
of length greater than $3$. Alternatively, a graph $G$ is bridged if and only if the balls
$B_r(A)=\{ v\in V: d(v,A)\le r\}$ around convex sets $A$ of $G$ are
convex.  Bridged graphs are exactly weakly modular graphs that do not
contain induced $4$-- and $5$--cycles (and therefore do not contain $4$-- and
$5$--wheels) \cite{Ch_metric}. Bridged graphs represent a far-reaching generalization
of chordal graphs; a  graph $G$ is \emph{chordal} if any induced cycle has length 3.
Notice also that bridged graphs are universal
in the following sense: any graph not containing induced 4- and 5-cycles is an induced
subgraph of a bridged graph. Together with median and Helly graphs, bridged graphs
represent one of the most important classes of graphs in metric graph theory \cite{BaCh_survey}.
For numerous applications of bridged graphs (alias systolic complexes) in geometric group theory, see
the paper \cite{ChOs} and the references therein.  A graph $G$ is {\it weakly bridged} \cite{ChOs} if $G$ is a
weakly modular graph with convex balls $B_r(x).$ It was shown in \cite{ChOs} that a weakly
modular graph $G$ is weakly bridged if and only if $G$ does not contain induced $4$-cycles.
Since bridged graphs are weakly modular and have convex balls, bridged graphs are weakly bridged.
Moreover, since weakly bridged graphs are weakly modular, they have strongly equilateral metric triangles.
We summarize all these results as follows:

\begin{lemma}[\cite{ChOs}]\label{convex-balls-chordal}  Weakly bridged graphs are exactly the weakly modular graphs with convex balls.
\end{lemma}

\begin{lemma}\label{neighbors-interval} Weakly bridged graphs satisfy INC.
\end{lemma}

\begin{proof} Let $u,v\in V$ and $k=d(u,v)-1$. If two vertices $x,y\in I(u,v)\cap N(u)$ are not adjacent, then $u\in I(x,y)$. Since $x,y\in B_k(v)$ and $u\notin B_k(v)$, this contradicts Lemma \ref{convex-balls-chordal}.
\end{proof}

\begin{lemma}\label{clique-chordal} If $C$ is a clique of a weakly bridged graph $G$ and all vertices of $C$ have the same distance $k$ to a vertex $u$, then there exists
a vertex $x$ at distance $k-1$ from $u$ and adjacent to all vertices of $C$.
\end{lemma}

\begin{proof} Let $C'$ be a maximal subset of $C$ such that all vertices of $C'$ have a common neighbor $x$ at distance $k-1$ from $u$. Since weakly bridged graphs satisfy the triangle condition,
$|C'|\ge 2$. We assert that $C'=C$. Let $w\in C\setminus C'$. Then $w\nsim x$.
Pick any $w'\in C'$ and let $y$ be a common neighbor of $w$ and $w'$ at
distance $k-1$ from $u$ (the existence of $y$ follows from triangle condition).
Since $w'$ is adjacent to $x$ and $y$, by the convexity of $B_{k-1}(u)$, $x$ and $y$ are adjacent. Let $w''$ be any other vertex of $C'$. Then $w'',x,y,w$ define a 4-cycle of $G$. Since $G$ is weakly bridged, this
$4$-cycle cannot be induced. Since $w\nsim x$, we conclude that $w''\sim y$. Therefore $y$ is adjacent to $w\in C\setminus C'$  and to all vertices of $C'$, contrary to the maximality of $C'$.
\end{proof}

\begin{theorem}\label{bridged->G2} Any weakly bridged graph $G$ has $G^2$-connected medians.
\end{theorem}

\begin{proof}  To establish the result, we will prove that weakly bridged graphs  satisfy the conditions of Proposition \ref{weaklymodular-abc}. By Lemmata \ref{convex-balls-chordal} and \ref{neighbors-interval},
weakly bridged graphs are weakly modular graphs satisfying INC. We will show that any 3-interval $I(u,v)$ satisfies condition (a) and that any 4-interval $I(u,v)$ satisfies condition (b).
If $d(u,v)=3$, then by Lemmata \ref{neighbors-interval} and \ref{clique-chordal} there exists a vertex $x$ at distance $d(u,v)-2=1$ from $u$ and adjacent to
all vertices of the clique $I(u,v)\cap N(v)$ and a vertex $y$ at distance $d(u,v)-2=1$ from $v$ and adjacent to all vertices of the clique $I(u,v)\cap N(v)$. Analogously, if $d(u,v)=4$, then
by Lemmata \ref{neighbors-interval} and \ref{clique-chordal} there exists
a vertex $x$ at distance $d(u,v)-2=2$ from $u$ and adjacent to
all vertices of $I(u,v)\cap N(u)$ and a vertex $y$ at distance $d(u,v)-2=2$ from $v$ and adjacent to
all vertices of $I(u,v)\cap N(v)$. If $d(u,v)=3$, then $x$ is adjacent to $u$ and $y$ is adjacent to $v$, establishing (a). If $d(u,v)=4$, then $x$ and $y$ both have
distance 2 to $u$ and to $v$, establishing (b). \end{proof}

Since chordal graphs are bridged, we immediately obtain the following result:

\begin{corollary}\label{chordal->G2} Any chordal graph $G$ has $G^2$-connected medians.
\end{corollary}

\subsection{Graphs with convex balls}

A graph $G$ is a {\it graph with convex balls} (a \emph{CB-graph} for short) if any ball $B_r(v)$ of $G$ is convex. It was shown
in \cite{FaJa,SoCh} that $G$ is a CB-graph if and only if all isometric cycles of $G$ have length 3 or 5 and $G$ satisfies the condition INC.
Bridged and weakly bridged graphs are CB-graphs. CB-graphs are not weakly modular: the 5-cycle $C_5$ is a CB-graph but is not weakly modular.
CB-graphs were studied in details in the recent paper \cite{CCG}. In particular, similar to the fact that bridged graphs are the weakly modular
graphs without induced 4-cycles, it was shown in \cite{CCG} that CB-graphs are exactly the graphs satisfying the condition INC and the following
condition, similar to triangle condition:

\begin{itemize}
\item
\emph{Triangle-Pentagon Condition (TPC)}:  for any three vertices $v,x,y$ such that $d(v,x)=d(v,y)=k\ge 2$ and $x\sim y$ either there exists a vertex $z\sim x,y$ with $d(v,z)=k-1$
or there exist vertices $z,w,w'$ such that   $xwzw'y$ is a pentagon of $G$ and $d(v,z)=k-2$.
\end{itemize}

The paper \cite{CCG} also proves a local-to-global characterization of CB-graphs by showing that the CB-graphs are exactly the graphs whose triangle-pentagon complexes are
simply connected and in which small balls (i.e., balls of radius at most 3) are convex. We will use the following result of \cite{CCG}, giving
the structure of metric triangles of CB-graphs:

\begin{proposition}[\cite{CCG}]\label{metrictriangle}   Any metric triangle of a CB-graph is strongly equilateral or has two sides of size 2 and one side of size 1.
\end{proposition}

Let $xyz$ be a metric triangle with $d(x,y)=d(x,z)=2$ and $d(y,z)=1$. Then taking any $s\sim x,y$ and any  $t\sim x,z$, we will get an induced 5-cycle. Indeed, $s$ and $t$ cannot coincide because $xyz$ is a metric triangle. For the same reason, $s\nsim z$ and $t\nsim y$. If $s\sim t$, then
by INC, $s$ and $z$ or $t$ and $y$ also must be adjacent, a contradiction. Therefore $(x,t,z,y,s)$ is an induced 5-cycle and we will call such a metric triangle $xyz$ a \emph{pentagon}.

The following analog of Lemma \ref{clique-chordal} was proved in \cite{CCG}:

\begin{lemma}[\cite{CCG}]\label{gate}  If $G$ is a CB-graph and $d(u,v)=k\ge 2$, then there exists a vertex $x$ at distance $k-2$ from $u$ and adjacent
to all vertices of $I(v,u)\cap N(v)$.
\end{lemma}

We continue with the following result, which generalizes Theorem \ref{bridged->G2}:

\begin{theorem}\label{CB->G2}   Any CB-graph $G$ has $G^2$-connected medians.
\end{theorem}

\begin{proof} First  we show that any 3-interval $I(u,v)$ satisfies condition (a) and that any 4-interval $I(u,v)$ satisfies condition (b).
If $d(u,v)=3$, then by Lemma \ref{gate}  there exists a vertex $x$ at distance $d(u,v)-2=1$ from $u$ and adjacent to
all vertices of the clique $I(u,v)\cap N(v)$ and a vertex $y$ at distance $d(u,v)-2=1$ from $v$ and adjacent to all vertices
of the clique $I(u,v)\cap N(u)$. If $d(u,v)=4$, then by Lemma \ref{gate}  there
exists
a vertex $x$ at distance $d(u,v)-2=2$ from $u$ and adjacent to
all vertices of $I(u,v)\cap N(u)$ and a vertex $y$ at distance $d(u,v)-2=2$ from $v$ and adjacent to
all vertices of $I(u,v)\cap N(v)$. If $d(u,v)=3$, then $x$ is adjacent to $u$ and $y$ is adjacent to $v$, establishing (a). If $d(u,v)=4$, then $x$ and $y$ both have
distance 2 to $u$ and to $v$, establishing (b).

Let $\pi$ be any weight function on $G$ and let $u,v$ be any pair of vertices with $3\le d(u,v)\le 4$. We assert that the median function $F_{\pi}$ satisfies the lozenge
condition $\Loz(u,v)$ with the vertices $x,y$ defined above. For this it suffices to show that $d(z,x)+d(z,y)\le d(z,u)+d(z,v)$ for any vertex $z\in \supp(\pi)$. If
$z\in S(u,v)$ (i.e., the triplet $z,u,v$ has a quasi-median $z'u'v'$ which is strongly equilateral), then the desired inequality follows from Lemma \ref{abc_bis}
(recall that $G$ satisfies INC). Therefore, it remains to consider the case
when any quasi-median $z'u'v'$ of $z,u,v$ is a pentagon. As in the proof of
Lemma \ref{abc_bis}, it suffices to show that $d(z',x)+d(z',y)\le
d(z',u)+d(z',v)$ holds. We distinguish two cases:

\begin{case} \label{case1CB} $d(u',v')=1$.
\end{case}
Then $d(z',u')=d(z',v')=2$. Independently of the position of the edge $u'v'$ in $I(u,v)$, we
have $d(z',u)+d(z',v)=6$ if $d(u,v)=3$ and $d(z',u)+d(z',v)=7$ if $d(u,v)=4$.

\begin{subcase}  $d(u,v)=3$.
\end{subcase}
If $u'=u$ or $v'=v$, say $u'=u$, then by condition (a)  $y$ is adjacent to $v'$ and $x$ is adjacent to $u$, thus $d(z',x)+d(z',y)\le 3+3=6$.
Otherwise, $u'\sim u$ and $v'\sim v$. Then $x$ is adjacent to $v'$ and $y$ is adjacent to $u'$ and again we obtain that $d(z',x)+d(z',y)\le 3+3=6$.

\begin{subcase}  $d(u,v)=4$.
\end{subcase}

Again, first let $u'=u$ (the case $v'=v$ is similar). By condition
(b), $x$ is adjacent to $v'$ and $d(z',x)\leq 3$. Since
$d(z',y) \leq d(z',u)+d(u,y) = 4$, we have $d(z',x)+d(z',y)\le 7$.
Now suppose that $u'$ is adjacent to $u$ (the case $v'\sim v$ is
similar). By condition (b), $x$ is adjacent to $u'$ and
$d(z',x) \le 3$. By condition (b), $y$ is adjacent to any common
neighbor $w$ of $v$ and $v'$ and thus $d(z',y) \le d(z',v')+2 = 4$ and
$d(z',x)+d(z',y)\le 7$.  This concludes the analysis of Case
\ref{case1CB}.

\begin{case} \label{case2CB}  $d(u',v')=2$.
\end{case}
Assume without loss of generality that $d(z',u')=1$ (the case $d(z',v')=1$ is similar). Independently of the position of $u'$ and $v'$ in $I(u,v)$,
we have $d(z',u)+d(z',v)=4$ if $d(u,v)=3$ and $d(z',u)+d(z',v)=5$ if $d(u,v)=4$. Suppose that $s$ is a common neighbor of $z'$ and $v'$ and $t$ is a common neighbor of $u'$ and $v'$.
The vertices $z',s,v',t,u'$ induce a pentagon.

\begin{subcase} $d(u,v)=3$.
\end{subcase}

If $u'=u$, then $y$ is adjacent to $t$ by condition (a) and to $v'$ by INC. If $d(z',y)=3$, then necessarily $u,s$ are neighbors of $z'$ in $I(z',y)$. By INC, $u\sim s$, which is impossible.
Thus $d(z',y)\le 2$. Since $x$ is adjacent to $u$, $d(z',x)\le 2$. Consequently, $d(z',x)+d(z',y)\le 4$. Now suppose that $u'\sim u$. In this case by condition (a) $y$ is adjacent to $u'$, whence $d(z',y)\le 2$.
By INC, $x$ is adjacent or coincides with $u'$, whence $d(z',x)\le 2$. Consequently, $d(z',x)+d(z',y)\le 4$.

\begin{subcase}  $d(u,v)=4$.
\end{subcase}
Again, first let $u'=u$.  By condition (b),  $x$ is adjacent to $t$ (recall that $t$ is a common neighbor of $u=u'$ and $v'$).
Since $v',x\in I(t,v)$, by INC the vertices $v'$ and $x$ are adjacent.
We assert that $d(z',x)\le 2$. Indeed, if $d(z',x)=3$, then $u,s\in I(z',x)$ and by INC we conclude that $u\sim s$, which is impossible because the vertices $z',s,v',t,u$ induce a pentagon. Thus $d(z',x)\le 2$.
Let $w$ be any common neighbor of $v'$ and $v$. By condition (b), $y\sim w$. Since $y,v'\in I(w,u)$, by INC $y\sim v'$, thus $d(z',y)\le 3$.  Consequently, $d(z',x)+d(z',y)\le 2+3=5$.

Now, let $u'\sim u$ and $v'\sim v$. By condition (b), $x$ is adjacent to $u'$
and $y$ is adjacent to $v'$. Consequently, $d(z',x)+d(z',y)\le 2+3=5$.
Finally, suppose that $v'=v$. By condition (b), $y\sim t$. Since $u',y\in I(t,u)$, by INC $u'\sim y$, thus $d(z',y)\le 2$. By condition (b), $x$ is adjacent to
any common neighbor $w$ of $u$ and $u'$. Hence $d(z',x)\le 3$, yielding $d(z',x)+d(z',y)\le 3+2=5.$ This concludes the analysis of Case \ref{case2CB}  and the proof of the theorem.
\end{proof}

\begin{remark} For an integer $i\ge 0$, a graph $G$ is a graph with an \emph{$\alpha_i$-metric} \cite{Ch_alpha,YuCh} if for any four vertices  $u,v,x,y$ such that $x\in I(u,y)$ and $y\in I(v,x)$, the  inequality
$d(u,v)\ge d(u,x)+d(x,y)+d(y,v)-i$ holds. Chordal graphs have $\alpha_1$-metrics \cite{Ch_alpha}. Graphs with $\alpha_1$-metric have been characterized in \cite{YuCh} as the graphs with convex balls not containing one
subgraph on 9 vertices as an isometric subgraph. Dragan and Ducoffe \cite{DrDu} informed us that they recently showed that the median functions in graphs with $\alpha_1$-metrics satisfy a property which is
equivalent to $G^2$-connectivity. Since graphs with $\alpha_1$-metrics have convex balls, our Theorem \ref{CB->G2} covers this result.
\end{remark}

\subsection{Bucolic and hypercellular graphs}
\emph{Bucolic graphs} are the retracts of the (weak) Cartesian
products of weakly bridged graphs \cite{BrChChGoOs}. They have been
characterized in \cite{BrChChGoOs} as the weakly modular graphs not
containing $K_{2,3},W_4$, and $W^-_4$ as induced subgraphs (this can
be also used as their definition). For numerous properties and topological
characterizations of bucolic graphs, see the paper \cite{BrChChGoOs}.
In this paper, we will use the decomposition theorem of bucolic graphs
into primes.

A graph $G$ is said to be {\it elementary} if the only proper gated
subgraphs of $G$ are singletons.  A graph with at least two vertices
is said to be \emph{prime} if it is neither a Cartesian product nor a
gated amalgam of smaller graphs. The prime gated subgraphs of a graph
$G$ are called the {\it primes} of $G$.  For example, the primes of
median graphs are $K_2$ and the primes of quasi-median graphs are
cliques. A graph $G$ is called {\it pre-median} \cite{Cha1,Cha2} if
$G$ is a weakly modular graph without induced $K_{2,3}$ and $W^-_4$.
Chastand \cite{Cha1,Cha2} proved that the prime pre-median graphs are
exactly the elementary pre-median graphs and he proved that any finite
pre-median graph can be obtained by gated amalgams from Cartesian
products of its prime subgraphs. Answering a question from
\cite{Cha1,Cha2} it was shown in \cite{CCHO} that the prime pre-median
graphs are exactly the pre-median graphs with simply connected clique
complexes and are exactly the pre-median graphs in which any $C_4$ is
included in a $W_4$ or a $M_4$ ($M_4$ is the graph consisting of an
induced $4$--cycle $(x_1,x_2,x_3,x_4)$ and four pairwise adjacent
vertices $a_1,a_2,a_3,a_4$ such that
$a_1\sim x_1,x_2; a_2\sim x_2,x_3; a_3\sim x_3,x_4; a_4\sim x_4,x_1$
and
$a_1\nsim x_3,x_4; a_2\nsim x_1,x_4; a_3\nsim x_1,x_2; a_4\nsim
x_2,x_3$).  Bucolic graphs are pre-median graphs and their primes are
exactly the weakly bridged graphs.  For finite bucolic graphs, this
can be restated in the following
way: 
\begin{theorem}[\cite{BrChChGoOs}]\label{decomposition_bucolic}  Finite bucolic graphs are exactly
the graphs which can be obtained from weakly bridged graphs via Cartesian products and gated amalgamations.
\end{theorem}

From Theorem \ref{decomposition_bucolic},
Proposition \ref{gated-amalgams-products},  and Theorem \ref{bridged->G2},  we obtain that any
finite bucolic graph $G$ has $G^2$-connected medians. We extend this
result to all locally-finite bucolic graphs (recall that a graph $G$ is \emph{locally-finite} if all vertices have finite degrees):

\begin{proposition}\label{bucolic} Any locally-finite bucolic graph $G$ has $G^2$-connected medians.
\end{proposition}

\begin{proof} Let $\pi$ be a profile with finite support $\supp(\pi)$ and let $u,v$ be two median vertices of
the median function $F_{\pi}$ with $d(u,v)\ge 3$. Let $G'$ be the convex hull of the finite set $\supp(\pi)\cup \{ u,v\}$.
By \cite[Proposition 2]{BrChChGoOs}, $G'$ is a finite bucolic graph. Since $G'$ is an isometric subgraph of $G$
containing $\supp(\pi)$, on the vertices of $V(G')$ the function $F_{\pi}$ has the same values with respect to
the distances of $G$ and of $G'$. Since $u$ and $v$ are median vertices of $F_{\pi}$ in $G'$, there exists a
median vertex $w\in I^{\circ}(u,v)\subset V(G')\subseteq V(G)$. Then clearly $w$ is also a median vertex of
$F_{\pi}$ in $G$ and we are done.
\end{proof}

A graph $G=(V,E)$ is called \emph{hypercellular}~\cite{ChKnMa19} if
$G$ can be isometrically embedded into a hypercube and $G$ does not
contain $Q^-_3$ as a partial cube minor ($Q^-_3$ is the 3-cube $Q_3$
minus one vertex).  Recall that a graph $H$ is called a \emph{partial cube minor}
of $G$ if $G$ contains a finite convex subgraph $G'$ which can be
transformed into $H$ by successively contracting some classes of
parallel edges of $G'$.  Hypercellular graphs
generalize \emph{bipartite cellular graphs}, which are the graphs in which all isometric cycles
are gated \cite{BaCh_cellular}. Hypercellular graphs are not weakly modular but  they represent
a far-reaching generalization of median graphs. For numerous properties and characterizations of
hypercellular graphs, see the paper
\cite{ChKnMa19}.  We will use the following characterization of finite hyperellular graphs:

\begin{theorem}[\cite{ChKnMa19}]\label{decomposition_hypercellular}  Finite hypercellular graphs are exactly
the graphs which can be obtained from even cycles via Cartesian products and gated amalgamations.
\end{theorem}

From Theorem \ref{decomposition_hypercellular} and  Corollary \ref{cor:amalgams-products} we conclude that if $G$ is a finite hypercellular graph,
then $p(G)$ is equal to the maximum of all $p(C)$ over all gated even cycles of $G$. Since, hypercellular graphs are isometrically embeddable into hypercubes,
convex hulls of finite sets are finite \cite{ChKnMa19}. Applying the same argument as in the proof of  Proposition \ref{bucolic}, we obtain the following result:

\begin{proposition} Let $G=(V,E)$ be a hypercellular graph and let ${\mathcal C}$ be the collection of all  gated cycles of $G$. Then $p(G)=\sup \{ p(C): C\in {\mathcal C}\}$.
\end{proposition}

\subsection{Bipartite absolute retracts} Median graphs are exactly the modular graphs with
connected medians \cite{SoCh_Weber}. Absolute retracts (alias Helly graphs) is
another important class of weakly modular graphs with connected medians \cite{BaCh_median}.
In this subsection,  we prove that the bipartite analogs of absolute retracts are graphs with $G^2$-connected medians.  A bipartite graph $G$ is an \emph{absolute retract} (in the category
of bipartite graphs) if for every
isometric embedding $\iota$ of $G$ into a bipartite graph $H$ there
exists a retraction from $H$ to $\iota(G)$.
Bipartite absolute retracts are bipartite analogs of Helly graphs
and constitute an important subclass of modular graphs. The graph $B_n$ is the bipartite complete graph $K_{n,n}$ minus a
perfect matching, i.e., the bipartition of $B_n$ is defined by
$a_1,\ldots,a_n$ and $b_1,\ldots,b_n$ and $a_i$ is adjacent to $b_j$
if and only if $i\ne j$. The graph $\widehat{B}_n$ is obtained from
$B_n$ by adding two adjacent new vertices $a$ and $b$ such that $a$ is
adjacent to all vertices $b_1,\ldots,b_n$ and $b$ is adjacent to all
vertices $a_1,\ldots,a_n$. The bipartite absolute
retracts have been characterized by Bandelt,
D\"ahlmann, and Sch\"utte~\cite{BaDaSch}:

\begin{theorem}[\cite{BaDaSch}]\label{absolute-retracts}
	For a bipartite graph $G$, the following conditions are equivalent:
	
	\begin{enumerate}[(1)]
		\item\label{th-bar-1} $G$ is an absolute retract;
\item\label{th-bar-3} $G$ is a modular graph such that every induced
		$B_n$ ($n\ge 4$) extends to $\widehat{B}_n$ in $G$;
		\item\label{th-bar-4} $G$ satisfies the following interval
		condition: for any vertices $u$ and $v$ with $d(u,v)\ge 3$, the
		neighbors of $v$ in $I(u,v)$ have a second common neighbor $x$ in
		$I(u,v)$.
	\end{enumerate}
\end{theorem}

\begin{proposition}\label{bar->G2} Any bipartite absolute retract $G$ has $G^2$-connected medians.
\end{proposition}

\begin{proof} By Proposition \ref{modular-abcd}, it suffices to show that each 3-interval satisfies condition (a) and each 4-interval satisfies condition (b). Pick any two vertices $u,v$ of $G$ with $3\le d(u,v)\le 4$.
By Theorem \ref{absolute-retracts}(\ref{th-bar-4}), there exists a vertex $x$ at distance $d(u,v)-2$ from $v$ and adjacent to
all neighbors of $u$ in $I(u,v)$. Analogously, there exists a vertex $y$ at distance $d(u,v)-2$ from $u$ and adjacent to
all neighbors of $v$ in $I(u,v)$. If $d(u,v)=3$, then $x$ is adjacent to $v$ and $y$ is adjacent to $u$, establishing that $G$ satisfies the stronger version of
condition (a) where $x$ and $y$ are adjacent.  If $d(u,v)=4$, then $x$ and $y$ both have distance 2 to $u$ and to $v$, thus $I(u,v)$ satisfies condition (b).
This concludes the proof.
\end{proof}

The class of bipartite absolute retracts is closed by taking gated amalgams and retractions but it is not closed by taking Cartesian products.
We will say that a graph $G$ is a {\it weak bipartite absolute retract} if $G$ can be obtained from a set of bipartite absolute retracts using Cartesian products and gated amalgams.
 Since modular graphs are closed by taking Cartesian products and gated amalgams,
weak bipartite absolute retracts are modular. Median graphs are not bipartite absolute retracts but are weak bipartite absolute retracts.
From Propositions \ref{gated-amalgams-products} and \ref{bar->G2} we immediately obtain:

\begin{proposition}\label{weak_bar} Any weak bipartite absolute retract $G$ has $G^2$-connected medians.
\end{proposition}

\subsection{Basis graphs} Bandelt and Chepoi \cite[Corollary 4]{BaCh_median} proved that basis graphs of matroids
have connected medians and used this result to design a simple greedy algorithm to compute the median of a profile
in such graphs. The proof follows from the fact that basis graphs of matroids are meshed, satisfy the positioning
condition (PC), and are thick. In this subsection, we revisit  this result of \cite{BaCh_median} and extend it
in several directions.

\subsubsection{Basis graphs of matroids and even $\Delta$-matroids} A {\it matroid} on a finite set $I$ is a collection
$\mathcal B$ of subsets of $I,$ called {\it bases,}  which satisfy the
following exchange property: for all $A,B\in {\mathcal B}$ and $a\in A\setminus B$
there exists $b\in B\setminus A$ such that $A\setminus
\{ a\}\cup \{ b\}\in {\mathcal B}$.  It is
well-known that all the bases of a matroid have the same
cardinality. The {\it basis graph} $G=G({\mathcal B})$ of a matroid $\mathcal B$ is the
graph whose vertices are the bases of $\mathcal B$ and edges are the
pairs  $A,B$ of bases differing by a single exchange (i.e., $\vert
A\triangle B\vert=2$).

A {\it $\triangle$--matroid} is a collection $\mathcal B$
of subsets of a finite set $I,$  called {\it bases} (not
necessarily equicardinal) satisfying the  symmetric exchange
property: for any  $A,B\in {\mathcal B}$ and $a\in A\triangle B,$ there
exists $b\in B\triangle A$ such that $A\triangle\{ a,b\}\in {\mathcal
B}.$ A $\triangle$--matroid whose bases all have the same cardinality
modulo 2 is called an {\it even $\triangle$--matroid}. The {\it basis
graph} $G=G({\mathcal B})$ of an even $\triangle$--matroid $\mathcal B$ is the
graph whose vertices are the bases of $\mathcal B$ and edges are the
pairs $A,B$ of bases differing by a single exchange, i.e., $\vert
A\triangle B\vert=2$ (the basis graphs of arbitrary collections
of subsets of even size of $I$ can be defined in a similar way).

Basis graphs of matroids and even $\triangle$--matroids
constitute a subclass of  meshed graphs  \cite{Che_bas}. Basis graphs
of matroids of rank $k$ on a set of size $n$ are isometric subgraphs
of the Johnson graph $J(n,k)$ and basis graphs of even $\triangle$--matroids on a set of size $n$
are isometric subgraphs of the half-cube $\frac{1}{2}H_n$.
Johnson graphs $J(n,k)$ are the basis graphs of uniform matroids of rank $k$, analogously
the half-cubes $\frac{1}{2}H_n$ are the basis graphs of uniform even $\triangle$--matroids.
We recall now the  characterization of basis graphs of matroids and $\triangle$--matroids.  For this purpose,
we introduce the following conditions:
\begin{itemize}
\item
{\it Positioning condition} (PC): for each vertex $u$ and each
square $v_1v_2v_3v_4$ of $G$ the equality $d(u,v_1)+d(u,v_3)=d(u,v_2)+d(u,v_4)$
holds.
\item
{\it 2-Interval condition} (IC$m$):  each 2-interval $I(u,v)$ is an induced subgraph of the $m$--hyperoctahedron $K_{m\times 2}$.
\end{itemize}

It is known that basis graphs of matroids and even $\triangle$--matroids  satisfy
(PC) \cite{Che_bas,Mau}. It is well-known \cite{Mau}
that the 2-intervals of basis graphs of matroids are either squares, or pyramids, or 3-octahedra, thus basis graphs of matroids
are thick and satisfy the 2-interval condition (IC3); analogously, the 2-intervals of even $\triangle$--matroids are thick and satisfy  (IC4).
Notice also that (PC) together with (IC3) or (IC4) imply the meshedness of basis graphs.
Maurer \cite{Mau}  presented a full characterization of finite graphs which
are basis graphs of matroids. Recently, answering a question of \cite{Mau}, this characterization
was refined in \cite{ChChOs_matroid}.  Extending Maurer's result, a characterization of basis graphs of even
$\triangle$--matroids was given in \cite{Che_bas}.   These characterizations
can be formulated in the following way:

\begin{theorem}[\cite{ChChOs_matroid,Che_bas,Mau}]\label{Th_Mau}
A finite graph $G$ is the basis graph of a
matroid if and only if $G$ is a connected thick graph
satisfying  (IC3) and (PC). A finite graph $G$ is a basis
graph of an even $\triangle$--matroid if and only if $G$ is a connected thick graph
satisfying (IC4), (PC),  and no neighborhood $N(v)$  contains an
induced 5-wheel $W_5$ or 6-wheel $W_6$.
\end{theorem}
For a characterization of isometric subgraphs of hypercubes, see \cite{Dj} and for a characterization of isometric
subgraphs of Johnson graphs, see \cite{Ch_Johnson}. Since the hypercubes can be isometrically embedded in Johnson graphs,
the second result generalizes the first result. Notice also that the question of the characterization of isometric subgraphs of
halved-cubes is still open.

\subsubsection{Isometric subgraphs of Johnson graphs and halved-cubes with connected medians}
In this subsection, we characterize the isometric subgraphs of Johnson graphs and of halved-cubes having connected medians.
In analogy with isometric subgraphs of hypercubes,
which are usually called \emph{partial cubes}, we call \emph{partial Johnson graphs} and \emph{partial halved-cubes} the
isometric subgraphs of Johnson graphs and halved-cubes, respectively. We start by extending the proof of \cite{BaCh_median} that basis graphs of matroids have connected medians to basis graphs
of even $\triangle$-matroid.

\begin{proposition}\label{thick} Let $G=(V,E)$ be a graph satisfying the positioning condition (PC) and let $\pi$ be any profile
on $V$. For any $u,v\in V$ with $d(u,v)=2$ included in a square $(u,s,v,t)$ the following version of $\Loz(u,v)$
holds: $F_{\pi}(s)+F_{\pi}(t)=F_{\pi}(u)+F_{\pi}(v)$. Consequently, if
$G$ satisfies (PC) and is thick (in particular, if $G$ is the basis graph
of an even $\triangle$-matroid), then $G$ has connected medians.
\end{proposition}

\begin{proof} Let $u,v$ be two vertices of $G$ with $d(u,v)=2$ included in a square $(u,s,v,t)$. By (PC), $d(s,x)+d(t,x)=d(u,x)+d(v,x)$ for any vertex $x$.
Taking the sum of such equalities multiplied by the weights of the vertices $x$ from $\supp(\pi)$, we obtain the required equality   $F_{\pi}(s)+F_{\pi}(t)=F_{\pi}(u)+F_{\pi}(v)$.
If $G$ is thick, then any pair of vertices $u,v$ with $d(u,v)=2$ is included in a square $(u,s,v,t)$ satisfying the previous equality, thus the median function $F_{\pi}$
is weakly convex. Finally, basis graphs of even $\triangle$-matroids are thick and satisfy the positioning condition (PC) \cite{Che_bas}.
\end{proof}

Median graphs are exactly the meshed isometric subgraphs of hypercubes and are exactly the
isometric subgraphs of hypercubes with connected medians. We show that this kind of characterization
can be extended to isometric subgraphs of Johnson graphs.

\begin{figure}[t]
{\includegraphics[scale=0.7]{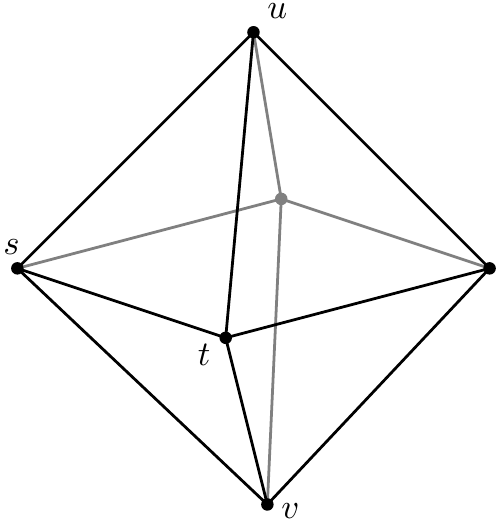}}
\caption{A 3-octahedron}\label{3octahedron}
\end{figure}

\begin{theorem}\label{meshed_Johnson} A partial Johnson graph  $G$  has connected medians if and only if
$G$ is meshed.
\end{theorem}

\begin{proof} One direction follows from the fact that all graphs with
connected medians are meshed (Remark \ref{connected->meshed}).
Conversely, let $G$ be a meshed partial Johnson
graph and $\pi$ be a profile on $G$. We show  that any pair of
vertices $u,v$ at distance 2
satisfies the lozenge condition $\Loz(u,v)$. Johnson graphs are
basis graphs of matroids, thus they satisfy the positioning condition (PC) \cite{Mau}.
Since $G$ is an isometric subgraph of a Johnson graph,  $G$ also satisfies
(PC). Since Johnson graphs satisfy the interval condition
(IC3), $G$ also satisfies (IC3). Thus, the interval $I(u,v)$ is included in a 3-octahedron; see Fig.~\ref{3octahedron}.
If the pair $u,v$ is included in a square $(u,s,v,t)$, then
$F_{\pi}(s)+F_{\pi}(t)=F_{\pi}(u)+F_{\pi}(v)$ by the first assertion
of Proposition \ref{thick} and we are done. Thus suppose that the
pair $u,v$ is not included in a square. Since $I(u,v)$ is a square-free subgraph of a 3-octahedron,
$I^{\circ}(u,v)$ consists
of a single vertex $s$ or of two adjacent vertices $s,t$. First,
let $I^{\circ}(u,v)=\{ s\}$.  Since $G$ is meshed, for any vertex
$x$, the radius function $r_x$ is weakly convex. This implies that
$d(s,x)\le \frac{1}{2}(d(u,x)+d(v,x))$ and therefore that
$F_{\pi}(s)\le \frac{1}{2}(F_{\pi}(u)+F_{\pi}(v))$.

Now suppose that $I^{\circ}(u,v)=\{ s,t\}$, where $s\sim t$. Pick any $x$ from the profile $\pi$.
We assert that $d(s,x)+d(t,x)\le d(u,x)+d(v,x)$. Suppose without loss of generality that
$d(u,x)\le d(v,x)$. The required inequality is obviously true if $d(v,x)=d(u,x)+2$.
Now, let $d(v,x)=d(u,x)+1$. By meshedness of $G$, we then have $\min\{ d(s,x),d(t,x)\}\le d(u,x)$,  say
$d(s,x)\le d(u,x)$. Since $s$ and $t$ are adjacent,   $d(t,x)\le d(s,x)+1\le d(u,x)+1=d(v,x)$,
yielding the required inequality $d(s,x)+d(t,x)\le d(u,x)+d(v,x)$.

Finally, let $d(u,x)=d(v,x)=\ell$. By meshedness of $G$, $\min\{ d(s,x),d(t,x)\}\le \ell$,
say $d(s,x)\le \ell$. Since $s\sim t$, if $d(s,x)<\ell$, then $d(t,x)\le \ell$ and
we  are done. Therefore, further
suppose that $d(s,x)=\ell$ and $d(t,x)=\ell+1$. Let $x'u'v'$ be a quasi-median of the triplet
$x,u,v$ in $G$. Since $G$ is meshed, $x'u'v'$ is an equilateral metric triangle.
Since $d(x,u)=d(x,v)$ and $d(u,v)=2$, either
$x'=u'=v'$ or $u'=u,v'=v$ and $x'uv$ is a metric triangle of size 2.
If $x'=u'=v'$ we conclude that $x'\in I^{\circ}(u,v)=\{ s,t\}$, yielding $x'=s$.
Since  $d(x,x')=\ell-1$ and $d(x,s)=\ell$, this is impossible.
Therefore, $x'uv$ is a metric triangle of size 2 and $d(x',x)=\ell-2$.
Since $G$ is meshed, $x'$ has distance 2 to $s$ or to $t$. If $d(x',t)=2$,
then $d(t,x)\le 2+d(x',x)=\ell$, which is impossible.
Therefore, $d(x',s)=2$ and $d(x',t)=3$.

Since metric triangles of $G$ are equilateral, $G$ satisfies the
triangle condition (TC). Applying (TC) to the triplet $x',s,v$, we
conclude that there exists a vertex $z\sim x',s,v$. If $z \sim u$, then
$z \in I^{\circ}(u,v)=\{ s,t\}$, contrary to the assumption that
$d(x',s)=2$ and $d(x',t)=3$.  Since $d(u,x')=d(u,z)=2$ and $x'\sim z$,
by (TC) there exists a vertex $w\sim u,x',z$ and $w \notin
\{s,t\}$. Since $w\sim u$ and $d(w,v)\le 2$ and since $G$ is meshed,
$w$ is adjacent to a vertex in $I^{\circ}(u,v)=\{s,t\}$. Since
$d(t,x')=3$, necessarily $w$ is adjacent to $s$ and is at distance 2
from $v$.  Then we obtain an induced $W_5$ defined by the vertices
$w,z,v,t,u,s$ and centered at $s$, which is impossible since $G$ is an
isometric subgraph of a Johnson graph (in fact, the neighborhood of
any vertex must be the line graph of a bipartite graph
\cite{Ch_Johnson}).
\end{proof}

\begin{figure}[t]
{\includegraphics[scale=0.6,page=4]{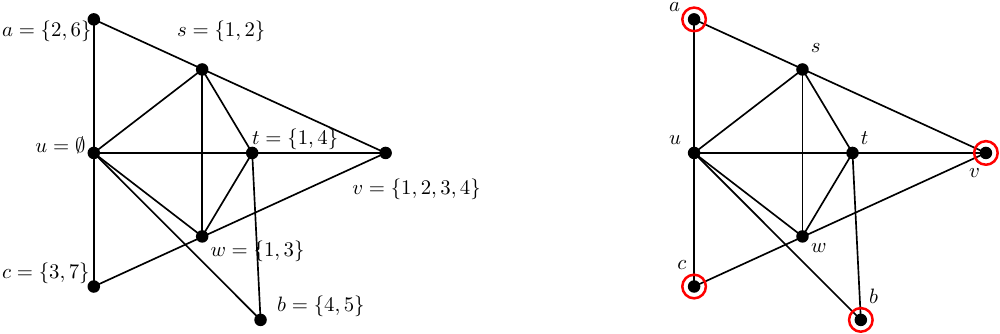}}
\caption{Left: a $\beta$-configuration $R_0$ and its isometric embedding into the halved-cube $\frac{1}{2}H_7$. Right: a profile $\pi=\{ a,b,c,v\}$  on $R_0$ for which
all values of $F_{\pi}$ are 6, only $F_{\pi}(u)=5$.}\label{beta-config}
\end{figure}

For partial halved-cubes, connectedness of medians cannot be longer characterized by  meshedness.
In Fig. \ref{beta-config} left, we present a weakly modular (in fact, a chordal) isometric subgraph $R_0$ of the halved-cube $\frac{1}{2}H_7$ and a profile for which not all local medians are global.
Right, a profile $\pi$ (which is 1 on the vertices $a,b,c,v$ and 0 elsewhere) for which $\Med(\pi)=\{ v\}$ and $\lMed(\pi)=\{ u,v\}$ is given.
Other examples of weakly modular partial halved-cubes with non-connected medians can be
obtained from $R_0$ by adding the edge $ac$ (in this case, $a=\{ 2,6\}$ and $c=\{ 3,6\}$)
or by adding the edges $ac, bc$, and $ab$ (in this case, $a=\{ 2,5\}, b=\{ 4,5\}$, and $c=\{ 3,5\}$). In both cases, $v$ is still a local median which is not a global median.
Now, we present a general method of constructing such examples. Let $u,v,s,t,w$ be five vertices as in the graph $R_0$, i.e., $d(u,v)=2, I^{\circ}(u,v)=\{ s,t,w\}$ and the vertices $s,t,w$ are pairwise
adjacent. We add three new vertices $a,b,c$, such that $a$ is adjacent to $s$ and  one of the vertices $u$ or $v$, the vertex $b$ is adjacent to $t$ and one of the vertices $u$ or $v$, and the vertex $c$ is adjacent to $w$ and one of the
vertices $u$ or $v$. There maybe also some edges between $a$, $b$, and $c$. We call such a graph a \emph{$\beta$-configuration} (since it violates condition $(\beta)$ of Theorem \ref{th-cmed}).

\begin{proposition}\label{prop:wm_hc} A weakly modular partial halved-cube $G$
has connected medians if and only if whenever $d(u,v)=2$ and $I^{\circ}(u,v)$ consists of three pairwise adjacent
vertices $s,t,w$,  the interval $I(u,v)$ is not included in a $\beta$-configuration.
\end{proposition}

\begin{proof} Let $G$ be a  weakly modular isometric subgraph $G$ of a halved-cube. Let $u,v$
be any two vertices of $G$ with $d(u,v)=2$. We asserts that the pair $u,v$ satisfies the conditions $(\alpha)$ and $(\beta)$
of Theorem \ref{th-cmed} if and only if $I(u,v)$ is not included in a $\beta$-configuration.

First suppose that the pair $u,v$ is included in a square $(u,s',v,t')$. Then clearly $I(u,v)$ cannot be included in a $\beta$-configuration.
Conversely, we show that in this case the pair $u,v$ satisfies the conditions $(\alpha)$ and $(\beta)$ of Theorem \ref{th-cmed}.
As  the set $S$ we can take $\{ s',t'\}$ and set $\eta(s')=\eta(t')=1$.
Since halved-cubes are basis graphs of even $\triangle$-matroids,
they satisfy the positioning condition (PC). Since $G$ is an isometric subgraph of a halved-cube,
$G$ also satisfies (PC).  Therefore,  $d(s',x)+d(t',x)=d(u,x)+d(v,x)$ for any vertex $x$ of $G$, implying  the condition $(\alpha)$ of
Theorem \ref{th-cmed}. To verify the condition $(\beta)$ of the same theorem, pick any vertex $x$ of
$J^{\circ}(u,v)$. Since $G$ satisfies the triangle condition, from the definition of $J^{\circ}(u,v)$ it follows that $x$ is adjacent to
one of the vertices $u,v$ and has distance 2 to the second vertex, i.e.,
$d(x,u)+d(x,v)=3$ holds. By (PC), we obtain that $d(x,s')+d(x,t')=3$. Thus $x$ is adjacent to exactly one of the vertices $s',t'$, establishing
$(\beta)$.

Now, suppose that the pair $u,v$ is not included in a square. Since any halved-cube satisfies
the interval condition (IC4), $G$ also satisfies (IC4). This implies that $I^{\circ}(u,v)$
consists of at most three pairwise adjacent vertices $\{ s,t,w\}$.
In all next cases, we set $S=I^{\circ}(u,v)$. We assert that $S$ satisfies the
condition $(\alpha)$ of Theorem \ref{th-cmed}, where the companion of each vertex of $S$ is
the vertex itself. Indeed, pick any $x\in M(u,v)$. If $d(u,x)=d(v,x)=2$, since $x\in J(u,v)$,
we conclude that the triplet $xuv$ is a metric triangle of size 2.
By the characterization of metric triangles  of weakly modular graphs \cite{Ch_metric},
we conclude that $d(z,x) = 2$ and
thus $d(z,x)+d(z,x)=d(u,x)+d(v,x)$ for any $z\in I^{\circ}(u,v)$. If $d(u,x)=d(v,x)=1$, then
$x\in I^{\circ}(u,v)$. Since all vertices of $I^{\circ}(u,v)$ are pairwise adjacent, we
obtain $d(z,x)+d(z,x)\le d(u,x)+d(v,x)$ for any $z\in I^{\circ}(u,v)$ also in this case.

Let $x\in J^{\circ}(u,v)=J(u,v)\setminus M(u,v)$.  Then either $x\sim u$ and $d(x,v)=2$ or $d(x,u)=2$ and $x\sim v$.
Since $G$ has equilateral triangles, any quasi-median of the triplet $x,u,v$ has the form $xuv'$ in the first
case and $xu'v$ in the second case, where $u'$ and $v'$ are vertices of $I^{\circ}(u,v)$.
If $x\in J^{\circ}(u,v)$ has a unique neighbor in $S=I^{\circ}(u,v)$, then we call it a \emph{personal neighbor} of $x$.

Now, we define a weight function $\eta$ on $S$ such that $\eta$ satisfies the condition $(\beta)$ of Theorem \ref{th-cmed}
if and only if the pair $u,v$ is not included in a $\beta$-configuration.  If $S=I^{\circ}(u,v)=\{ s\}$, we set $\eta(s)=1$.
Then $s$ is the personal neighbor of any vertex  $x\in J^{\circ}(u,v)$ and we are done.  If $S=I^{\circ}(u,v)=\{ s,t\}$,
then we set $\eta(s)=\eta(t)=1$ if $s$ and $t$ are personal neighbors of some vertices of
$J^{\circ}(u,v)$ and $\eta(s)=1, \eta(t)=0$ if $t$ is not a personal neighbor of
any vertex of $J^{\circ}(u,v)$. Then any vertex  $x\in J^{\circ}(u,v)$ is adjacent to at least one of the vertices $s$ and $t$ in the
first case and to vertex $s$ in the second case, thus again the condition $(\beta)$ is satisfied. Since $I(u,v)$ cannot be included in
a $\beta$-configuration, we are also done in this case.

Finally, let $S=I^{\circ}(u,v)=\{ s,t,w\}$. If each of the vertices $s,t,w$ is a
personal neighbor of some vertex of $J(u,v)\setminus M(u,v)$, then $u,v,s,t,w$ together with these three vertices  $a,b,c$ of $J(u,v)\setminus M(u,v)$
will define a forbidden $\beta$-configuration. In this case, no function $\eta$ on $S$ will satisfy the condition $(\beta)$.
Indeed, $(\beta)$ will be violated  by that vertex among $a,b,c$ whose personal neighbor realizes
$\min \{ \eta(s),\eta(t),\eta(w)\}$).  Thus $G$ does not have connected medians. Conversely, if $u,v,s,t,w$ is not included in a forbidden $\beta$-configuration,
then at least one of the vertices from $S$, say $w$, is not a personal neighbor of any vertex of $J^{\circ}(u,v)$. Then we set $\eta(s)=\eta(t)=1$ and $\eta(w)=0$.
Since any vertex $x$ of $J^{\circ}(u,v)$ is adjacent to at least one of the vertices $s,t$, we conclude that the pair $u,v$ satisfies the condition $(\beta)$.
\end{proof}

In case of meshed partial halved-cubes, the condition $(\alpha)$ of Theorem \ref{th-cmed} is no longer satisfied.
An \emph{$\alpha$-configuration} containing a 2-interval $I(u,v)$ is of one of the following three types:
\begin{itemize}
\item[\emph{Type 1}:] two or three pairwise adjacent vertices $s,t$ or $s,t,w$ of $I^{\circ}(u,v)$, a vertex $a$ such that $d(a,u)=d(a,v)=d(a,s)=d(a,w)=2, d(a,t)=3$
and a vertex $b$ such that $b\sim t, b\nsim s,w$, and $b$ is adjacent to $u$ or to $v$.
\item[\emph{Type 2}:] three pairwise adjacent vertices $s,t,w$ of  $I^{\circ}(u,v)$, two vertices $a_1,a_2$ such that $d(a_1,u)=d(a_1,v)=d(a_1,s)=d(a_1,w)=2, d(a_1,t)=3$, $d(a_2,u)=d(a_2,v)=d(a_2,s)=d(a_2,t)=2, d(a_2,w)=3$
and a vertex $b$ such that $b\sim t,w$, $b\nsim s$, and $b$ is adjacent to $u$ or  to $v$.
\item[\emph{Type 3}:] three pairwise adjacent vertices $s,t,w$ of  $I^{\circ}(u,v)$, and three vertices $a_1,a_2,a_3$ such that $d(a_1,u)=d(a_1,v)=d(a_1,s)=d(a_1,w)=2, d(a_1,t)=3$, $d(a_2,u)=d(a_2,v)=d(a_2,s)=d(a_2,t)=2, d(a_2,w)=3$, and $d(a_3,u)=d(a_3,v)=d(a_3,t)=d(a_3,w)=2, d(a_1,s)=3$.
\end{itemize}
In Fig. \ref{alpha-config} we present isometric subgraphs of halved-cubes realizing $\alpha$-configurations of Type 1. Notice that there exists a finite number of $(\alpha)$- and $(\beta)$-configurations
and their realizations as isometric subgraphs of halved-cubes. However we do not know if all $\alpha$-configurations of Types 2 and 3 can be realized in isometric subgraphs of halved-cubes with the condition $I^{\circ}(u,v)=\{ s,t,w\}$.

\begin{figure}[t]
{\includegraphics[scale=0.6,page=4]{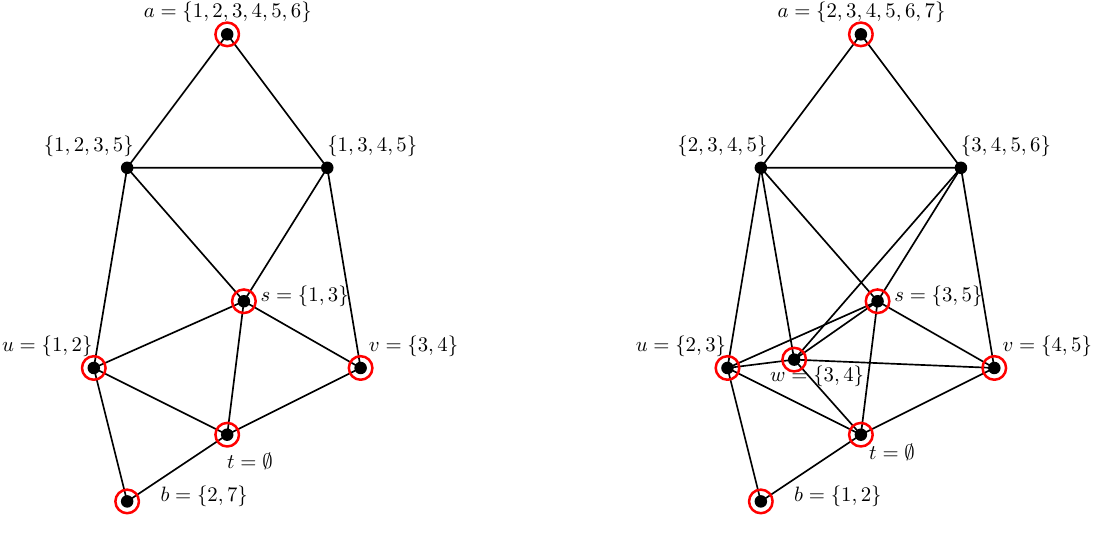}}
\caption{Two $\alpha$-configurations of Type 1 and their realizations as partial halved-cubes.}\label{alpha-config}
\end{figure}

\begin{theorem}\label{prop:meshed_hc} A partial halved-cube $G$ has connected medians if and only if $G$ is meshed and  whenever $d(u,v)=2$ and $I^{\circ}(u,v)$ consists of two or three
pairwise adjacent vertices,  then  $I(u,v)$ is not included in an $\alpha$- or $\beta$-configuration.
\end{theorem}

\begin{proof} The case when $u,v$ is included in a square is the same as in Proposition \ref{prop:wm_hc}. Now suppose that $I^{\circ}(u,v)$ consists of at most three vertices, which are pairwise adjacent.
Again, the proof of condition $(\beta)$ of Theorem \ref{th-cmed} is the same as in Proposition \ref{prop:wm_hc}. Therefore, it remains to show that if a meshed partial halved-cube $G$  satisfies the
conditions $(\alpha)$ and $(\beta)$, then
$G$ does not contains an  $\alpha$-configuration and that if $G$ does not contain   $\alpha$- and $\beta$-configurations, then $G$ satisfies condition $(\alpha)$.

Let $S=I^{\circ}(u,v)$ and $\eta$ be the weight function on $S$ satisfying the conditions $(\alpha)$ and $(\beta)$ of Theorem \ref{th-cmed}. Suppose by way of contradiction that  $G$ contains an $\alpha$-configuration.
If this configuration is of Type 3, then any vertex of $I^{\circ}(u,v)$ is at distance 3 from some vertex of $M(u,v)$ having distance 2 to $u,v$ and to all other vertices of $I^{\circ}(u,v)$, showing that condition $(\alpha)$ is violated.
If $I(u,v)$ is included in a configuration of Type 2, then condition $(\alpha)$ is violated for vertices $t$ and $w$ of $I^{\circ}(u,v)$. Thus $S$ may include only the vertex $s$. But the vertex $s$ is not adjacent to the vertex $b$ from $J^{\circ}(u,v)$,
thus condition $(\beta)$ cannot be satisfied. Finally suppose that $I(u,v)$ is included in a configuration of Type 1.
Then $a\in M(u,v)$ and $b\in J^{\circ}(u,v)$. Since $d(a,t)=3$ and $a$ has distance 2 to all other vertices of $I^{\circ}(u,v)$, the vertex $t$ cannot be included in $S$ (otherwise, for $t$ and any its companion in $S$
the condition $(\alpha)$ will be violated). But $t$ is the personal neighbor in $I^{\circ}(u,v)$ of the vertex $b$ from $J^{\circ}(u,v)$ and since $t\notin S$, the condition $(\beta)$ is violated. This shows that $G$ does not
contain $\alpha$-configurations.

Conversely, suppose that for any pair of vertices $u,v$ at distance 2, the interval $I(u,v)$  is not included in an $\alpha$-configuration and satisfies condition $(\beta)$. If all vertices of $I^{\circ}(u,v)$
have distance 2 to all vertices $x$ of $M(u,v)$ with $d(u,x)=d(v,x)=2$, then condition $(\alpha)$ will be satisfied with $S=I^{\circ}(u,v)$ and the weight function $\eta$ defined as in the proof of Proposition \ref{prop:wm_hc}.
In particular, this holds if $|I^{\circ}(u,v)|=1$. Thus suppose that $2\le |I^{\circ}(u,v)|\le 3$ and that there exists a vertex $t\in I^{\circ}(u,v)$ at distance 3 from a vertex $a_1\in M(u,v)$.
Since intervals of halved-cubes are convex \cite[(3.4)]{Che_bas} and $G$ is an isometric subgraph of a halved-cube, the intervals of $G$ are convex. Since $u,v\in I(t,a_1)$, we conclude that all
vertices of $I^{\circ}(u,v)$ different from $t$  also belong to $I(t,a_1)$. Since they are all adjacent to $t$, they all have distance 2 to $a_1$. If $t$ is the personal neighbor of a vertex $b\in J^{\circ}(u,v)$,
then the set $I(u,v)\cup \{ a_1,b\}$ defines a forbidden $\alpha$-configuration of Type 1. Therefore, $t$ is not a personal neighbor of any vertex of $J^{\circ}(u,v)$.

First, let $I^{\circ}(u,v)=\{ s,t\}$. If $d(s,x)=2$ for any vertex $x\in M(u,v)$ with $d(x,u)=d(x,v)=2$, then we can set $S=\{ s\}$ and $\eta(s)=1$ and the condition $(\alpha)$ will be satisfied. Since
any vertex of $J^{\circ}(u,v)$ is adjacent to $s$, the condition $(\beta)$ also is satisfied. Therefore, suppose that $d(s,x)=3$ for a vertex $x\in M(u,v)$. Since $G$ is meshed, $d(x,t)=2$. By (TC), there exists a
vertex $b$ adjacent to $x,u,$ and $t$. Since $b\in J^{\circ}(u,v)$ and $t$ is not a personal neighbor of $b$, $b$ must be adjacent to $s$, contrary to the assumption that $d(s,x)=3$.

Finally, let $I^{\circ}(u,v)=\{ s,t,w\}$. If $d(s,x)=d(w,x)=2$ for any vertex $x\in M(u,v)$ with $d(x,u)=d(x,v)=2$, then we can set $S=\{ s,w\}$ and $\eta(s)=\eta(w)=1$ and the condition $(\alpha)$ will be satisfied. Since
any vertex of $J^{\circ}(u,v)$ is adjacent to at least one of the vertices $s$ or $w$, the condition $(\beta)$ also is satisfied. Therefore, suppose that $d(w,a_2)=3$ for a vertex $a_2\in M(u,v)$. Then from convexity of intervals
of $G$ we conclude that $d(a_2,s)=d(a_2,t)=2$. Also, as in the case of $t$, we conclude that $w$ is not a personal neighbor of any vertex of $J^{\circ}(u,v)$ (otherwise, we get a forbidden $\alpha$-configuration of Type 1).
Now, if there exists a vertex $b\in J^{\circ}(u,v)$ adjacent to $t$ and $w$ and not adjacent to $s$, then we will get a forbidden $\alpha$-configuration of Type 2. Therefore all vertices of $J^{\circ}(u,v)$ are adjacent to $s$. If $d(s,x)=2$ for any
vertex $x\in M(u,v)$ with $d(x,u)=d(x,v)=2$, then we can set $S=\{ s\}$ and $\eta(s)=1$ and conditions $(\alpha)$ and $(\beta)$ of Theorem \ref{th-cmed} will be satisfied. Therefore, there exists a vertex $a_3\in M(u,v)$ with
$d(s,a_3)=3$. Then again, we conclude that $d(t,a_3)=d(w,a_3)=2$ and we obtain a forbidden $\alpha$-configuration of Type 3. This final contradiction shows that if $G$ does not contains $\alpha$- and $\beta$-configurations, then
$G$ satisfies the conditions $(\alpha)$ and $(\beta)$ of Theorem \ref{th-cmed}, thus $G$ is a graph with connected medians.
\end{proof}

\begin{corollary}\label{meshed_hc-median} A triangle-free partial halved-cube  $G$  has connected medians if and only if $G$ is median.
\end{corollary}

\begin{proof} The graph $G$ satisfies the conditions of Theorem \ref{prop:meshed_hc}, in particular $G$ is meshed.  Modular graphs are exactly the triangle-free meshed graphs
and are exactly the triangle-free weakly modular graphs \cite{Ch_metric}. On the other hand, median graphs are exactly the
modular graphs not containing induced $K_{2,3}$ \cite{Mu}. Finally,  halved-cubes do not contain induced $K_{2,3}$. Indeed, suppose by way of contradiction that $\frac{1}{2}H(X)$
contains an induced $K_{2,3}$ and suppose without loss of generality that a vertex $u$ of degree 3 of $K_{2,3}$ is encoded as $\varnothing$. Since the other vertex $v$ of degree 3 of $K_{2,3}$ has
distance 2 to $u$, $v$ must be encoded by a subset of size 4 of  $\frac{1}{2}H(X)$, say by the set $\{ 1,2,3,4\}$. Then the three common neighbors of $u$ and $v$ in $K_{2,3}$ must be encoded
by pairwise disjoint subsets of size 2 of $\{ 1,2,3,4\}$. Obviously, this is impossible.
\end{proof}

\subsection{Benzenoid systems}
A \emph{hexagonal grid} is a grid formed by a tessellation of the plane ${\mathbb R}^2$
into regular hexagons.  Let $Z$ be a cycle of the hexagonal grid. A \emph{benzenoid system} is a subgraph of the
hexagonal grid induced by the vertices located in the closed region of the plane
bounded by $Z$.  Equivalently, a benzenoid system is a finite connected plane graph in which every interior face is a regular
hexagon of side length 1. Benzenoid systems play an important role in chemical graph theory  \cite{GuCy}.

It was shown in \cite{Ch_benzen} that any benzenoid system $G=(V,E)$ can be isometrically embedded into the Cartesian product $T_1\square T_2\square T_3$
of three trees $T_1,T_2,T_3$, which are defined in the following way.
Let $E_1, E_2,$ and $E_3$ be the edges of $G$  on the three directions
of the hexagonal grid and let $G_i=(V,E\setminus E_i)$ be the graph obtained from $G$ by removing the edges of $E_i, i=1,2,3$. The tree $T_i$
has the connected components of $G_i$ as its set of vertices and two such connected components $P$ and $P'$ are adjacent in $T_i$ if and only if
there is an edge $uv\in E_i$ with one end in $P$ and another end in $P'$. The isometric embedding $\varphi: V\rightarrow T_1\square T_2\square T_3$
maps any vertex $v$ of $G$ to a triplet $(v_1,v_2,v_3)$, where $v_i$ is the connected component of $G_i$ containing $v$, $i=1,2,3$ \cite{Ch_benzen}.
Since trees isometrically embed into  hypercubes, every benzenoid $G$ also isometrically embeds into a hypercube, i.e., benzenoids are partial cubes \cite{KlGuMo}.
As in the case of Cartesian products, two edges $uv,xy$ of a benzenoid $G$ are  called \emph{parallel} if they belong to the same
class $E_i$ and $u_i = x_i, v_i = y_i$, i.e., if they come from the same edge of the tree $T_i$. The parallelism relation is symmetric and transitive.
Therefore the edges of each $E_i, i=1,2,3$ are partitioned into \emph{parallelism classes}. The parallelism classes of a benzenoid $G$ correspond to the coordinates
of the isometric embedding of $G$ into the smallest hypercube \cite{KlGuMo}. Removing the edges of each parallelism class, the vertex-set of the
benzenoid $G$ is partitioned into two convex sets \cite{Dj}. This implies the following two properties of parallelism  classes  of $G$, which we will use next:
\begin{itemize}
\item[(1)] any shortest path $P$ of $G$ contains at most one edge from each parallelism class.
\item[(2)] if $C$ is a cycle of $G$ and $xy$ is an edge of $C$, then the parallelism class of the edge $xy$ contains at least one other edge of $C$.
\end{itemize}
Finally notice that a parallelism class of the hexagonal grid may define several parallelism classes of a benzenoid $G$.
The inner faces of $G$ are called \emph{hexagons}. They correspond to the hexagons of the underlying hexagonal grid.

A path $P$ of a benzenoid $G$ is an \emph{incomplete hexagon} if $P$
is a path of length 3 that contains an edge from each class
$E_1,E_2,E_3$ and $P$ is not included in a hexagon of $G$. Since each
inner vertex of $G$ has degree three and belongs to three hexagons,
from the definition it follows that all vertices of an incomplete
hexagon belong to the external cycle $Z$ of $G$. Rephrasing this, a
path $P$ of length 3 is an incomplete hexagon if and only if $P$ is
included in a unique hexagon $H$ of the hexagonal grid and the
intersection of $H$ with $G$ equals
$P$. For an illustration of all these notions, see Fig. \ref{traversing-path}.

\begin{figure}[t]
{\includegraphics[scale=0.5]{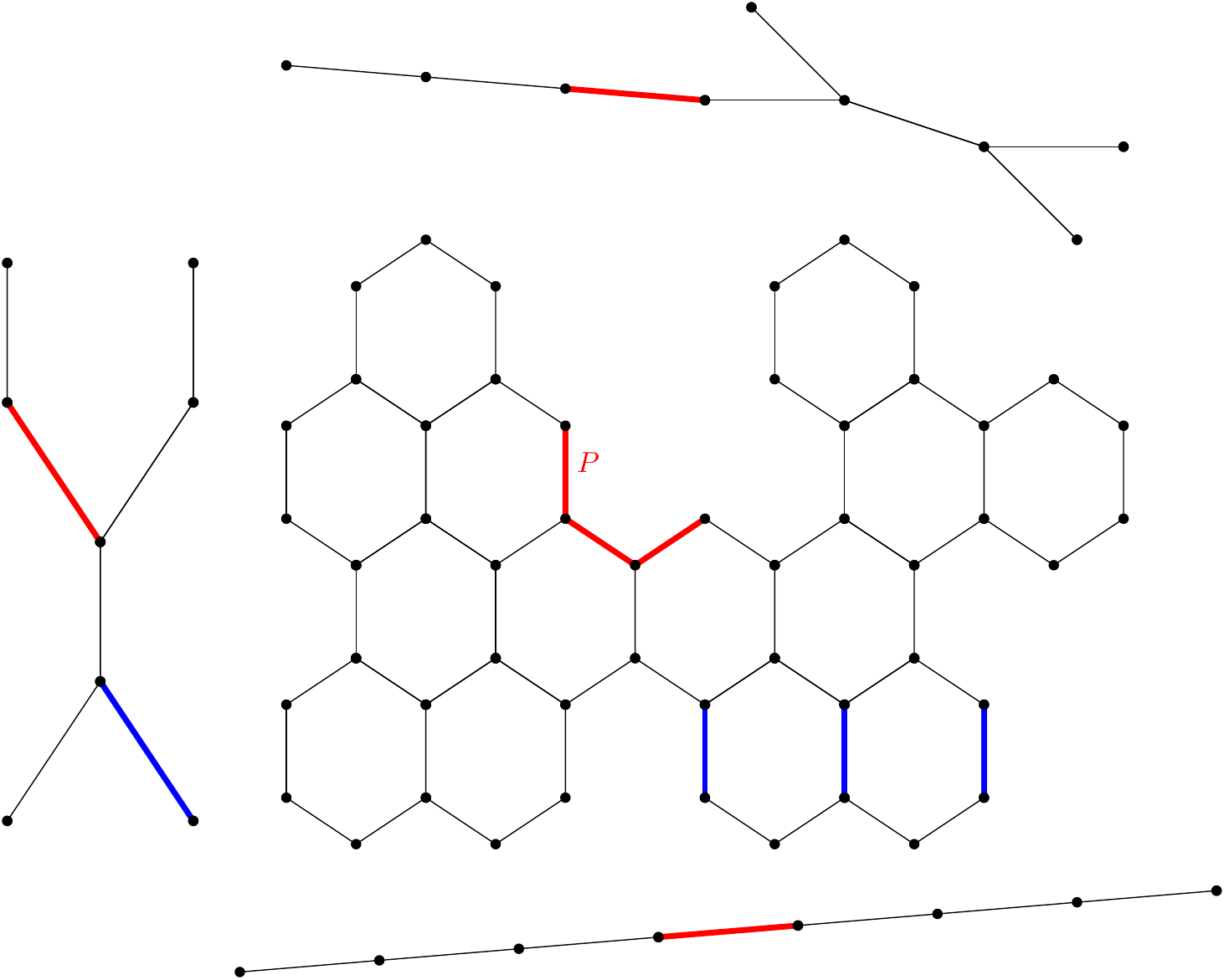}}
\caption{A benzenoid system $G$, its tree factors $T_1,T_2,T_3$, a parallelism
class (in blue), and an incomplete hexagon (in red) of
$G$.}\label{traversing-path}
\end{figure}

We proved above that $p(G)$ of a hypercellular graph $G$ is the supremum of $p(C)$ over all
gated cycles of $G$. All hexagons $C$ of any benzenoid $G$ are gated and for them $p(C)=2$, however
$G$ cannot be obtained from its hexagons via gated amalgams. Nevertheless, it is still true that
benzenoid systems $G$ have $G^2$-connected medians.

\begin{proposition}\label{benz->G2}
Any benzenoid system $G$ has $G^2$-connected medians.
\end{proposition}

\begin{proof}
By Theorem \ref{th-cmed-p} and Theorem \ref{wp-ltg} it suffices to
prove that any median function $F_{\pi}$ is locally $2$-weakly peakless.
Let $\varphi: V\rightarrow T_1\square T_2\square T_3$ be the isometric
embedding of $G$ into the Cartesian product of 3 trees $T_1,T_2,T_3$
defined above.  For a vertex $x$ of $G$, let $(x_1,x_2,x_3)$ be the
coordinates of $\varphi(x)$ in $T_1\square T_2\square T_3$.  We denote
the distance function of $T_i$ by $d_i$ instead of $d_{T_i}$,
$i=1,2,3$.  We also denote by $\pi_i$ the profile on $T_i$ obtained
from $\pi$ as we did for Cartesian products of graphs.  Benzenoids are
bipartite graphs. Therefore, if we pick any three vertices $x,y,u$ of
$G$ such that $x\sim y$, then $|d(u,x)-d(u,y)|=1$ and either
$x\in I(u,y)$ or $y\in I(u,x)$.  Indeed, otherwise consider a vertex
$u' \in I(u,x) \cap I(u,y)$ at maximum distance from $u$ and observe
that the edge $xy$ together with any shortest $(u',x)$- and
$(u',y)$-paths define an odd cycle.

We start with the following property of benzenoids.

\begin{claim}\label{6-cycles} Any hexagon or incomplete hexagon $H$ of $G$ is gated.
\end{claim}

  \begin{proof}
    Pick a vertex $u \in V$ and let $x$ be a vertex of $H$ minimizing $d(u,x)$.
    If $x$ is not the gate of $u$ in
    $H$, then there exists a vertex $y$ in $H$
    such that $d(u,y)<d(u,x)+d(x,y)$. Let $y$ be a closest to $u$
    such vertex. If $x\sim y$, then by minimality of $d(u,x)$, we have
    $d(u,y)=d(u,x)$, which is impossible since $G$ is bipartite.  Thus,
    necessarily $d(x,y)=2$ or $d(x,y)=3$.
    First, let  $d(x,y)=2$. Let $s$ be the unique
    common neighbor of $x$ and $y$. Since $G$ is bipartite, from the choice of $x$ and $y$ we conclude that
    $d(u,x)=d(u,y)=d(u,s)-1$.  Without loss of generality, assume that
    $xs \in E_1$ and $sy \in E_2$. If $x,y$ belong to a hexagon $H$, then
    both $x,y$ have a second neighbor in $H$. If $x,y$ belong to an incomplete hexagon $H$,
    then either $x$ or $y$ has a second neighbor in $H$. Suppose without loss of generality that $x$ has a
    neighbor $z$ in $H$ and note that $xz\in E_3$. By the choice
    of $x$, $d(u,z)=d(u,x)+1=d(u,s)$. Let $p$ be a neighbor of
    $x$ in $I(x,u)$. Since $xs\in E_1$ and $xz\in E_3$, necessarily $xp$ belongs to
    $E_2$. Consequently, in the tree $T_2$ we have two paths connecting the vertices $s_2=x_2=z_2$ and $u_2$:
    one starts with the edge $x_2p_2$ and another starts with the edge $s_2y_2$.  This is possible only if the
    edges $xp$ and $sy$ are parallel. Since $x$ and $s$ are adjacent in $G$, this is impossible. This concludes the
    analysis of the case $d(x,y)=2$.

    Suppose now that $d(x,y)=3$ and consider two vertices $s,t$ such
    that $(x,s,t,y)$ is a path of length $3$.  Let $xs \in E_1$,
    $st \in E_2$, and $ty \in E_3$. Since $G$ is bipartite, by the
    choice of $x$ and $y$ we have $d(u,s)=d(u,x)+1$,
    $d(u,t)=d(u,y)+1$, and $|d(u,s)-d(u,t)|=1$.  Since
    $d(u,x)\le d(u,y)$, we conclude that
    $d(u,t)=d(u,s)+1=d(u,y)+1=d(u,x)+2$.  Let $P_x$ and $P_y$ be
    shortest $(u,x)$- and $(u,y)$-paths of $G$. The concatenation of
    $P_x$, the path $(x,s,t,y)$, and of $P_y$ (traversed from $y$ to
    $u$) includes a simple cycle $C$ containing the vertices
    $x,s,t,y$.  From property (2) of parallelism classes, the edges
    $xs$, $st$, and $ty$ of $C$ are parallel to some edges
    $e_1, e_2,e_3$, respectively, of the cycle $C$.  Since $P_x$
    together with the edge $sx$ is a shortest $(s,u)$-path of $G$, by
    property (1) of parallelism classes, the edge $e_1$ cannot belong
    to $P_x$, thus $e_1\in P_y$.  Analogously, since $P_x$ together
    with the path $(x,s,t)$ is a shortest $(t,u)$-path $P'$ and $P_y$
    together with the edge $ty$ is a shortest $(t,u)$-path $P''$, we
    conclude that $e_2\notin P_x$ and $e_3\notin P_y$, whence
    $e_2\in P_y$ and $e_3\in P_x$. If $H$ is a hexagon and $z$ is the
    second neighbor of $x$ in $H$, then the edge $xz$ is parallel to
    $ty$ and thus to $e_3$. From the choice of $x$ as a closest to $u$
    vertex of $H$ and since $G$ is bipartite, we conclude that $P_x$
    plus the edge $xz$ is a shortest $(z,u)$-path. Since the edges
    $xz$ and $e_3$ are parallel and both belong to this path, we
    obtain a contradiction with property (1) of parallelism
    classes. Thus suppose that $H$ is an incomplete hexagon. Let $H'$
    be the hexagon of the hexagonal grid containing $H$ and let $z$
    and $z'$ be the neighbors of $x$ and $y$ in $H'$,
    respectively. Let also $p$ be the neighbor of $x$ in $P_x$ and $q$
    be a neighbor of $y$ in $P_y$. Then $xz\in E_3$ and $z'y\in
    E_1$. Since $H$ is an incomplete hexagon, either $z \notin V(G)$
    and $z \neq p$ or $z' \notin V(G)$ and $z' \neq q$. If
    $z' \neq q$, then necessarily $yq \in E_2$. Consequently, in the
    tree $T_2$ we have two paths between the vertices $t_2=y_2$ and
    $u_2$: one starts with the edge $t_2s_2$ and the other starts with
    the edge $y_2q_2$. This is possible only if the edges $st$ and
    $yq$ are parallel.  Since $t$ and $y$ are adjacent in $G$, this is
    impossible. If $z \neq p$, then in the tree $T_2$, we have two
    paths between the vertices $s_2 = x_2$ and $u_2$: one starts with
    the edge $x_2p_2$ and the other starts with the edge
    $s_2t_2$. This is possible only if the edges $st$ and $xp$ are
    parallel.  Since $s$ and $x$ are adjacent in $G$, this is
    impossible.  This contradiction finishes the analysis of the case
    $d(x,y)=3$ and concludes the proof.
\end{proof}

As in previous proofs, we will show that $F_{\pi}$ satisfies the lozenge
property $\Loz(u,v)$ for any $u,v$ such that $3 \leq d(u,v) \leq 4$. Namely, that $I^{\circ}(u,v)$
contains two vertices
$s$ and $t$ (not necessarily different) such that
$F_{\pi}(s)+F_{\pi}(t)\leq F_{\pi}(u)+F_{\pi}(v)$.  We distinguish two cases:
$d(u,v)=3$ and $d(u,v)=4$.
\begin{case}
	$d(u,v)=3$.
\end{case}

By definition of benzenoids, any two consecutive edges on a $(u,v)$-geodesic
belong to different classes of the triplet $E_1,E_2,E_3$. Then $I(u,v)$ is either
a hexagon or a  3-path. If $I(u,v)$ is a hexagon $C$, then we choose $s$ and $t$
such that $s\sim u$ and $t\sim v$ are opposite vertices of $C$. For any vertex $x$
from the profile $\pi$, let $x'$ be its gate in $C$. Then independently of the location
of $x'$ on $C$, we have $d(x',u)+d(x',v)=d(x',s)+d(x',t)=3$ and thus $d(x,u)+d(x,v)=d(x,s)+d(x,t)$.
Consequently, $F_{\pi}(u)+F_{\pi}(v)=F_{\pi}(s)+F_{\pi}(t)$. Analogously, if $I(u,v)$ is an incomplete hexagon $P$,
we choose $s,t\in P$ with $s\sim u$ and $t\sim v$. Let $x'$ be the gate of $x$ in $P$. Then independently
of the location of $x'$ on $P$, $d(u,x')+d(v,x')=3$ and $d(s,x')+d(t,x')\le 3$. Consequently, $d(x,s)+d(x,t) \leq d(x,u)+d(x,v)$,
yielding $F_{\pi}(s)+F_{\pi}(t)\le F_{\pi}(u)+F_{\pi}(v)$.

Finally, suppose that  $I(u,v)$ is a 3-path $P$ not included in a hexagon.
Then the first and the third
edge of $P$ belong to the same class, say to $E_1$, and the middle edge belongs to
another class, say to $E_2$. Again, we choose $s\sim u$ and $t\sim v$.
Note that the coordinates of $u,v,s,t$ satisfy the equalities
$u_3=s_3=t_3=u_3$, $u_2=s_2$ and $t_2=v_2$.  Since in
$T_1$ we have $s_1=t_1 \in I(u_1, v_1)$, for any $x\in \supp(\pi)$ and any $i=1,2,3$ we have
$d_i(s_i,x_i)+d_i(t_i,x_i) \leq d_i(u_i,x_i)+d_i(v_i,x_i)$ for each $i$, establishing $F_{\pi}(s)+F_{\pi}(t)\le F_{\pi}(u)+F_{\pi}(v)$.

\begin{case}
	$d(u,v)=4$.
\end{case}

Analogously to the case above, $I(u,v)$ is either a $4$-path $P$ or a hexagon
with a pending edge. First suppose that $I(u,v)$ is a hexagon $C$ with a pending edge $e=zv$.
Then $u$ is the vertex of $C$ opposite to $z$. Suppose that the edge
$e$ belongs to $E_1$ and  notice that $C$ contains two
other edges from $E_1$.  Let $s,t$ be the two neighbors of $z$ in $C$. Then $s,t$ have distance 2 to $u$ and to $v$.
Observe that $s_1=t_1 \in I(u_1,v_1)$
and that we can assume that $s_2=u_2$, $t_2=v_2$, $s_3=v_3$, and
$t_3=u_3$. Consequently, for any $x \in \supp(\pi)$, we have
$d_i(s,x) + d_i(t,x) \leq d_i(u,x) + d_i(v,x)$ and thus
$F_{\pi}(s)+F_{\pi}(t)\le F_{\pi}(u)+F_{\pi}(v)$.

Now suppose that $I(u,v)$ is $4$-path $P=(u,s,r,t,v)$. Then (up to relabeling) the
classes of edges of $P$ occur in one of the following orders: (a) $E_1,E_2,E_1,E_2$, (b) $E_1,E_2,E_3,E_1$,
or (c) $E_1,E_2,E_3,E_2$. In the last two cases, $P$ consists of a (gated) incomplete hexagon  $P'=(u,s,r,t)$ plus one edge $tv$.
For a vertex $x\in \supp(\pi)$, we denote by $x'$ the gate of $x$ in $P'$.

In case (a), let $r$ be the middle vertex of $P$. Then $d_1(u_1,v_1)=d_1(u_1,r_1)+d_1(r_1,v_1)$, $d_2(u_2,v_2)=d_2(u_2,r_2)+d_2(r_2,v_2)$, and
$u_3=r_3=v_3$. Since the distance function of trees is convex, for any $x=(x_1,x_2,x_3) \in \supp(\pi)$, we have
$2d_i(r_i,x_i)\leq d_i(u_i,x_i)+d_i(v_i,x_i)$ and  thus $2F_{\pi}(r)\le F_{\pi}(u)+F_{\pi}(v)$. Thus in this case we
can suppose that both vertices occurring in $\Loz(u,v)$ coincide with $r$.

Now, consider the case (b), i.e.,  $us,tv\in E_1$, $sr\in E_2$, and $rt\in E_3$. In this case, the 3-path $P''=(s,r,t,v)$
is an incomplete hexagon and thus gated.
Pick any vertex $x \in \supp(\pi)$. If its gate $x'$ in $P'$
is different from $u$, then we have $d(s,x)=d(u,x)-1$. Since $d(t,x)\le d(v,x)+1$,  we obtain
$d(s,x)+d(t,x)\leq d(u,x)+ d(v,x)$. Similarly, if the gate $x''$ of $x$ in $P''$
is different from $v$, then we obtain $d(s,x)+d(t,x)\leq d(u,x)+d(v,x)$. Finally, if
$u$ is the gate of $x$ in $P'$ and $v$ is the gate of $x$ in $P''$, then
we must have $d(x,t)=d(x,s)+2$ and $d(x,s)=d(x,t)+2$, which is impossible.

Finally, consider the case (c), i.e.,  $us\in E_1$, $sr, tv\in E_2$, and $rt\in E_3$. As in case (b),
we can assume that $u$ is the gate of $x$ in the
incomplete hexagon $P'$ and that $d(x,v)=d(x,t)-1$. Since $u$ is
the gate of $x$ in $P'$, we have
$d_2(x_2,u_2)=d_2(x_2,s_2) = d_2(x_2,r_2)-1 = d_2(x_2,t_2)-1$. Since $tv \in E_2$
and $d(x,v)=d(x,t)-1$, we have $d_2(x_2,v_2) =
d_2(x_2,t_2)-1$. Consequently, in  $T_2$ the vertex $r_2=t_2$ has two
distinct neighbors $s_2=u_2$ and $v_2$ on a shortest path between $r_2$ and $x_2$, contradicting that
$T_2$ is a tree.
\end{proof}

\section{Conclusion}
In this paper, we characterized the graphs $G$ in which median sets always induce connected subgraphs in the $p$th power $G^p$ of $G$. We proved that those are
the graphs in which the median functions are $p$-weakly peakless/convex (such functions  are unimodal on $G^p$).
We showed that bridged and weakly bridged graphs (and thus chordal graphs), graphs with convex balls, bucolic graphs, hypercellular, bipartite absolute retracts, and benzenoids have $G^2$-connected medians.
We provided sufficient local conditions for a modular graph to  have $G^2$-connected medians. We also characterized the isometric subgraphs of Johnson
graphs and of halved-cubes with connected medians.

In the theory of discrete convexity \cite{Hirai,Murota}, submodular functions on Boolean cubes, $L^{\#}$-convex and $N$-convex functions
on ${\mathbb Z}^d$, products of trees, and median graphs are investigated. Lozenge functions, introduced in \cite{BaCh_median} and used in most of our proofs, can be viewed as a generalization of such discrete convex functions to general graphs.

As we already mentioned in the Introduction, the unimodality/convexity of the median function in a geodesic space is a fundamental property behind the existence of fast and convergent algorithms of its minimization.
In the setting of graphs with $n$ vertices and $m$ edges,  the median problem can be trivially solved in $O(nm)$ time by solving the All Pairs Shortest
Paths (APSP) problem.  One may ask if APSP is necessary to compute the
median. However, by~\cite{AbGrVa} the APSP and median problems are equivalent under subcubic reductions.
It was also shown in~\cite{AbVaWa} that computing the medians of sparse graphs in
subquadratic time refutes the HS (Hitting Set) conjecture. We believe that classes of graphs with connected or $G^2$-connected medians are good candidates
in which the median problem can be solved faster than in $O(nm)$ time. Our belief is based on the fact that all known such algorithms
(for median graphs \cite{BeChChVa}, for planar bridged triangulations \cite{ChFaVa}, for Helly graphs \cite{Du_Helly}, and for basis graphs of matroids \cite{BaCh_median}) use the unimodality of median functions.
However, designing such minimization algorithms is a not so easy problem because
they cannot use the entire distance matrix of the graph. To start, one may ask
if \emph{the medians of chordal graphs can be computed faster than $O(nm)$ time?}

A nice axiomatic characterization of medians of median graphs via three basic axioms has been obtained by Mulder and Novik~\cite{MuNo}.
Recently, in \cite{BeChChVa_axiom} we considerably extended this result, by providing a similar axiomatic characterization of medians in all graphs
with connected medians and in modular graphs with $G^2$-connected medians (in particular, in bipartite absolute retracts). Our proof uses  the characterization of graphs with connected medians given in \cite{BaCh_median} and with $G^2$-connected medians given in this paper.
We do not know if medians of all graphs with $G^2$-connected medians (in particular, of chordal graphs) admit such axiomatic characterizations, but we hope that some of them do.

Recently, Puppe and Slinko \cite{PuSl} characterized median graphs as closed Condorcet domains. A \emph{domain} is any set $\mathcal D$ of linear
orders (complete, transitive, and antisymmetric binary relations) on a finite set $X$. The set of all linear orders on a set $X$, viewed as a graph, is the
permutahedron on the set $X$.  The \emph{majority relation} associated with a profile $\pi$ of a domain $\mathcal D$ is the binary relation $\prec_{\pi}$ on $X$
such that $x\prec_{\pi} y$ if and only if more than half of the voters rank $x$
above $y$.  A \emph{Condorcet domain}  is a domain $\mathcal D$ with the
property that,
whenever $\pi$ is a subset of $\mathcal D$, their majority relation has no cycles (i.e., has no Condorcet paradox).
Finally, a \emph{closed domain}  is a domain $\mathcal D$ with the property that, whenever an odd profile $\pi$ is a subset of $\mathcal D$, their majority
relation belongs to $\mathcal D$. All permutohedra admit isometric embeddings into a hypercube \cite{EpFaOv}. Conversely, all hypercubes are isometrically embeddable
into permutohedra. Thus, general domains can be viewed as arbitrary subsets of vertices of permutohedra or of hypercubes. \emph{Median domains} are the
domains which induce median graphs \cite{PuSl}.

Since median graphs are the bipartite graphs with connected medians \cite{SoCh_Weber} and hypercubes are bipartite, median graphs are exactly the isometric
subgraphs of hypercubes with connected medians. Moreover, it is
well-known \cite{BaBa,SoCh_Weber} and easily follows from the unimodality of the median function that  the medians in median graphs $G$ can be
computed by the majority rule: for any profile $\pi$, the median set $\Med_G(\pi)$ is the intersection of all halfspaces of $G$ containing at
least one half  of the profile $\pi$. It is shown in \cite{SoCh_Weber} that if the median graph $G$ is isometrically embedded into a hypercube $H(X)$, then
$\Med_{\pi}(G)=\Med_{\pi}(H(X))\cap V(G)$. If the profile is odd, then $\Med_{\pi}(H(X))$
is a single vertex, which necessarily belongs to $G$ by previous equality. This
explains why median domains are the closed Condorcet domains.

Halved-cubes are exactly the squares of hypercubes
and Johnson graphs are the isometric subgraphs of halved-cubes induced by the levels of hypercubes. Therefore, it will be interesting to investigate those domains of linear orders (or of other types of partial orders), which induce
particular partial halved-cubes or partial Johnson graphs. In this case, such domains are viewed not as subgraphs of  permutahedra (or hypercubes) but as induced subgraphs of halved-cubes or of Johnson graphs. Notice that
such domains are encoded by the vertices  of the hosting hypercube (viewed as $(0,1)$-vectors) which have an even number of 1's for halved-cubes (respectively, the same number of 1's for Johnson graphs).  We call a domain $\mathcal D$
inducing a partial halved-cube or a partial Johnson graph   \emph{median-closed} if for any profile $\pi$ included in $\mathcal D$ we have $\Med_{\pi}(G)=\Med_{\pi}(Q)\cap V(G)$, where $Q$ is either $\frac{1}{2}H(X)$ or $J(X,k)$
(equivalently, if $\Med_{\pi}(Q)\cap V(G)\ne\varnothing$).

From the results of \cite{BaCh_median} and this paper, it follows that basis graphs of matroids, of even $\Delta$-matroids, and, more generally, the meshed partial Johnson graphs are graphs with connected medians.
We also characterized  the partial halved-cubes with connected medians (they are also meshed). Notice that meshed partial Johnson graphs
present an interesting common generalization of median graphs and basis graphs of matroids. In fact, by Corollary \ref{meshed_hc-median}, the median graphs are exactly the
partial halved-cubes with connected  medians not containing triangles. It will be interesting to investigate which partial halved-cubes and partial Johnson graphs
with connected medians define median-closed domains.  In particular, \emph{are the domains defined by basis graphs of matroids or even $\Delta$-matroids median-closed?}
All such domains are not Condorcet domains unless they are median (because they contain triangles).  However one can ask if \emph{the median sets in such classes of graphs can be computed
by a relaxation of the majority rule?} For this, a relaxed majority rule should
be designed for Johnson graphs and halved-cubes.

\subsection*{Acknowledgements.} We would like to acknowledge the referee of the original submission for many useful suggestions, in particular about the presentation and the organisation of the paper.
We would like to acknowledge the referee of the revision for a very careful reading, numerous corrections, and useful suggestions. Finally, we would like to acknowledge the editor of the journal and the editor
of the special issue for their guidance. The work on this paper was supported by ANR project DISTANCIA (ANR-17-CE40-0015).


\end{document}